\documentclass[a4paper,twoside]{article}  

\usepackage{amssymb,amsmath,amsthm,dsfont,amsfonts,color,latexsym}   

\usepackage{hyperref}
\usepackage{CJK}

\usepackage{mathrsfs} 

\definecolor{blue}{rgb}{0.00,0.00,1.00}
\definecolor{red}{rgb}{1.00,0.00,0.00}

\newcommand{\red}{\color{red}}

\renewcommand{\baselinestretch}{1.2}
\def\bq{\begin{equation}}
\def\eq{\end{equation}}
\def\ba{\begin{array}{ccc}}
\def\bal{\begin{array}{lll}}
\def\ea{\end{array}}

\def\dcup{\displaystyle\bigcup} 
 
 \def\dsum{\displaystyle\sum}
 \def\lt#1{\left#1}\def\rt#1{\right#1}
\def\({\left(}\def\){\right)}
\def\[{\left[}\def\]{\right]}

    \def \R   {\mathbb{R}}

    \def\P    {\mathrm{P}}
    \def\i    {\mathrm{i}}

    \def\S    {\mathbb{S}}
    
    \def\intr {\int_{\R^3}}
    \def\ints {\int_{\S^2}}
    \def\intt {\int^t_0}
    \def\intrr {\int_{\R^6}}

    \def \Q    {\mathcal{Q}}
    
    \def \N    {\mathbb{N}}

    \def \pt   {\partial}
    
    \def \Dt   {\frac{\rm d}{{\rm d}t}}
    
    \def \dt    {\partial_t}
    
    \def \da    {\pt^\alpha}

    \def \dx    {\partial_x}

    \def \dxa   {\partial^{\alpha}_x}
    
    \def \dv    {\partial_v}
    
    \def \dvb   {\partial^{\beta}_v}

    \def \divx  {{\rm div}_x}

    \def\Tdx   {\nabla_x}
    \def\Tdv   {\nabla_v}

       \def\bq{\begin{equation}}
       \def\eq{\end{equation}}
       \def\be{\begin{equation}}
       \def\ee{\end{equation}}
       \def\bma#1\ema{{\allowdisplaybreaks\begin{align}#1\end{align}}}
       \def\bmas#1\emas{{\allowdisplaybreaks\begin{align*}#1\end{align*}}}
       \def\bln#1\eln{{\allowdisplaybreaks\begin{aligned}#1\end{aligned}}}
       \def\nnm{\notag}
       \def\bgr#1\egr{\allowdisplaybreaks\begin{gather}#1\end{gather}}
       \def\bgrs#1\egrs{\allowdisplaybreaks\begin{gather*}#1\end{gather*}}

       \theoremstyle{plain}
       \newtheorem{lem}{\bf Lemma}[section]
       \newtheorem{thm}[lem]{\textbf{Theorem}}
       
       \newtheorem{rem}[lem]{\textbf{Remark}}

       \newtheorem{remark}[lem]{\bf Remark}

\begin{document}


\title{Spectrum Structure and  Behaviors of the Vlasov-Maxwell-Boltzmann Systems}

\author{ Hai-Liang Li$^1$,\, Tong Yang$^2$,\, Mingying Zhong$^3$\\[2mm]
 \emph{\small\it  $^1$Department of  Mathematics,
    Capital Normal University, P.R.China.}\\
    {\small\it E-mail:\ hailiang.li.math@gmail.com}\\
    {\small\it $^2$Department of Mathematics, City University of Hong Kong,  Hong
    Kong}\\
    {\small\it E-mail: matyang@cityu.edu.hk} \\
    {\small\it  $^3$Department of  Mathematics and Information Sciences,
    Guangxi University, P.R.China.}\\
    {\small\it E-mail:\ zhongmingying@sina.com}\\[5mm]
    }
\date{ }

\pagestyle{myheadings}
\markboth{Vlasov-Maxwell-Boltzmann System}%
{H.-L. Li, T. Yang, M.-Y. Zhong}

 \maketitle

 \thispagestyle{empty}

\begin{abstract}\noindent
The spectrum structures and behaviors of the Vlasov-Maxwell-Boltzmann (VMB) systems for both two species and one species are studied in this paper. The analysis shows the effect of the Lorentz force induced by the electro-magnetic field leads to some different structure of spectrum from the classical Boltzmann equation and the closely related Vlasov-Poisson-Boltzmann system. And the significant difference between the two-species VMB model and one-species VMB model are given. The structure in high frequency illustrates the hyperbolic structure of the Maxwell equation. Furthermore, the long time behaviors and the optimal convergence rates to the equilibrium of the Vlasov-Maxwell-Boltzmann systems for both two species and one species are established based on the spectrum analysis,  and in particular the phenomena of the electric field dominating and magnetic field dominating are observed for the one-species Vlasov-Maxwell-Boltzmann system.


\end{abstract}

%
\tableofcontents

\section{Introduction}
\label{sect1}
\setcounter{equation}{0}

The Vlasov-Maxwell-Boltzmann system is a fundamental model in plasma physics for the describing the time evolution of dilute charged particles,
such as electrons and ions,  under the influence of the self-induced Lorentz forces governed by Maxwell equations, cf. \cite{ChapmanCowling}
for derivation and the physical background. In the literatures, there are two basic models, one is called two-species Vlasov-Maxwell-Boltzmann system that describes both the time evolution of ions and electrons
\be
 \left\{\bln     \label{VMB1z}
 &  \dt F_++v\cdot\Tdx F_++(E+v\times B)\cdot\Tdv F_+ =\Q(F_+,F_+)+\Q(F_+,F_-),\\
 &  \dt F_-+v\cdot\Tdx F_--(E+v\times B)\cdot\Tdv F_- =\Q(F_-,F_+)+\Q(F_-,F_-), \\
 & \dt E=\Tdx\times B-\intr (F_+-F_-)vdv,\\
 & \dt B=-\Tdx\times E,\\
 & \Tdx\cdot E=\intr (F_+-F_-)dv,\quad \Tdx\cdot B=0,
 \eln\right.
\ee
where $F_{\pm}=F_{\pm}(t,x,v)$ are number
density distribution functions of charged particles, and $E(t,x)$, $B(t,x)$ denote the electro and magnetic fields, respectively.  Here, the operator $\Q(\cdot,\cdot)$ describing
 the binary elastic collisions is given by
 \bq
 \Q(F,G)=\intr\ints
 |(v-v_*)\cdot\omega|(F(v'_*)G(v')-F(v_*)G(v))dv_*d\omega,\label{binay_collision}
 \eq
where
$$
 v'=v-[(v-v_*)\cdot\omega]\omega,\quad
 v'_*=v_*+[(v-v_*)\cdot\omega]\omega,\quad \omega\in\S^2.
$$
The other one called one-species Valsov-Maxwell-Boltzmann system that takes account of the fact that the ion is much heavier than the electron and the electron moves faster than the ion so that the time evolution of electron can be considered under a fixed background of ion distribution
\be
 \left\{\bln    \label{1VMB1z}
 &  \dt F+v\cdot\Tdx F+(E+v\times B)\cdot\Tdv F =\Q(F,F),\\
 & \dt E=\Tdx\times B-\intr Fvdv,\\
 & \dt B=-\Tdx\times E,\\
 & \Tdx\cdot E=\intr Fdv-n_b,\quad \Tdx\cdot B=0,
 \eln\right.
\ee
where  the time evolution of electrons is considered under the influence of  a fixed ions background
$n_b(x)$, and $F=F(x,v,t)$ is the number density function of electrons. And the operator $\Q(\cdot,\cdot)$ is defined by \eqref{binay_collision}.

The Vlasov-Maxwell-Boltzmann system has been intensively studied and many important progress has been made in \cite{Duan4,Duan5,Guo4,Jang,Strain}.
For instance, Guo \cite{Guo4} has first established the global existence of classical solutions in three-dimensional torus when the initial data is a small perturbation of a global Maxwellian, and Strain \cite{Strain} proved the the corresponding global existence result of  classical solutions in  $\R^3$. The diffusive limit for two-species VMB system  was shown in \cite{Jang}.
%
The recent important investigation of long time behavior of global solution near the global Maxwellian studied  in \cite{Duan4,Duan5} shed light on the complicacy of the Valsov-Maxwell-Boltzmann (VMB) system. Therein, it was shown by the method of compensated functions in \cite{Duan5} that the total energy of the linearized one-species VMB system decays at the rate $(1+t)^{-\frac38}$ (but the decay rate of the nonlinear system has not been obtained  since the decay rate of the linear system obtained therein seems to be not enough to deal with the time evolution of the nonlinear terms) and in \cite{Duan4} that the total energy of nonlinear two-species VMB system decays at the rate $(1+t)^{-\frac34}$.

A natural and interesting question follows then, what is the main characters of the structures and time-asymptotical behaviors of the Valsov-Maxwell-Boltzmann system on the transport of charged particles under the influence of electromagnetic fields governed by the Maxwell equation and/or mutual interaction between charged particles. One of the methods to investigate this properties is the analysis of spectrum structures of the Valsov-Maxwell-Boltzmann system. However, in contrast to the works on Boltzmann equation~\cite{Ellis,Liu1,Liu2,Ukai1,Ukai2,Ukai3} and Vlasov-Poisson-Boltzmann system \cite{Li2,Li3}, the spectrum  of the linearized Valsov-Maxwell-Boltzmann system around a global equilibrium has not been given despite of its importance. The main purpose of this paper is to fill in this gap.

The main purpose of this paper is to  consider the structures of the linearized systems for the above VMB systems \eqref{VMB1z} and \eqref{1VMB1z} around a global equilibrium so that some specific properties influenced by the electromagnetic fields and/or the mutual interaction between charged particles are revealed. These spectrum structures are important for understanding the behavior of the solutions to these systems both locally in space and globally in time. Indeed, the main purpose in the present paper is to continue the project to study the structures and behaviors on the transport of charged particles under the influence of electric fields, electromagnetic fields, or magnetic field. As it has already been studied in \cite{Duan1,Li2,Li3} about the structure structures and behaviors of both one-species and two-species Vlasov-Poisson-Boltzmann systems, the influence of electric field  gives rise to some complicated phenomena on the transport of charged particles. Indeed, it was shown in \cite{Duan1,Li2} that the global distribution function to the one-species Vlasov-Poisson-Boltzmann system tends to the global Maxwellian at the optimal rate $(1+t)^{-\frac14}$ in $\R^3$ which is slower than the Boltzmann equation and is caused mainly due to the slower but optimal decay $(1+t)^{-\frac14}$ of the electric fields. On the other hand, with the influence of mutual interaction of charged particles included,  the global distribution function and electric field of the two-species Vlasov-Poisson-Boltzmann system was proved in \cite{Yang4,Li3} to converge to the equilibrium at the optimal rate $(1+t)^{-\frac34}$ for the distribution functions and $e^{-\mathcal{O}(1)t}$ for the electric field,
where the key issue lies in the fact that the mutual interaction of charged particles leads to spectral gap.

In the present paper, we shall establish the structure of the spectrum in details for both two-species and one-species linearized VMB systems, analyze the corresponding semigroups to the linearized operators and show the optimal decay rates of global solutions to the linearized  VMB systems, and finally obtain the (optimal) time-asymptotical behaviors of global solutions to the nonlinear VMB systems of both two-species and one-species types. To be more precisely, we first establish the structure of the spectrum in details for two linearized VMB systems and reveal the effect of electromagnetic fields on the distribution of spectrum of linearized operators. This effect gives rise to a completely different distribution of the spectrum of the linearized operators for both two-species and one-species VMB system.
Indeed, the influence of electromagnetic fields on the transport of one-species charged particles causes the linearized VMB system admits higher order eigenvalues (spectrum) $\lambda_6=\lambda_7= -\mathcal{O}(1)|\xi|^4 $ at lower frequency $0<|\xi|\ll1$ besides those behaving like $\lambda_j= j{\rm \i}- \mathcal{O}(1)|\xi|^2$ for $j=0,\pm1$ (refer to Theorem~\ref{spectrum2}). This is carried out in terms of the relation between the rotational part of macroscopic velocity vector field and the electromagnetic fields and in turn causes the slower time-convergence rates of the global distribution to the global Maxwellian (refer to Theorems~\ref{time1a}--\ref{time3a} for details).
However, the mutual interaction of particles with different type of charges cancels this particular influence of electromagnetic fields and only the spectrum like $\lambda_j= -\mathcal{O}(1)|\xi|^2$ is kept finally. 
In addition, the appearance of electromagnetic fields causes the additional spectrum (eigenvalues)  around $\pm\i |\xi|$  and in particular ${\rm Re}\lambda=- \mathcal{O}(1)|\xi|^{-1}$ at high frequency $|\xi|\gg1$ for both one-species and two species VMB systems. This unfortunately leads to the loss of regularity of global solutions (refer to Theorem~\ref{spectrum1} and Theorem~\ref{spectrum2} for details).

Then, in terms of the analysis on spectrum structures and the semigroups of both two-species and one-species linearized VMB systems, we are further able to establish the optimal time convergence rates of the global solutions for both linearized systems and nonlinear systems in three-dimensional whole space.
For two-species VMB system, we can observe some phenomena of wave propagation and  magnetic field domination on long time behaviors of charge transport of charge transport due to the effect of magnetic filed and mutual interaction between the particles of two species. Indeed, for the global solution $(f_1,f_2,E,B)$ to the linearized two-species VMB system, we can show that the distribution function $f_1$, corresponding to the total summation of the distribution functions between the two species, is governed by the linearized Boltzmann equation and its optimal time decay rate $(1+t)^{-\frac34}$ in $L^2$-norm has been already established for instance in \cite{Ukai1,Zhong2012Sci}.  Meanwhile, the magnetic field $B$ is also shown to tend to zero at the optimal time decay rates $(1+t)^{-\frac34}$ in $L^2$-norm, but the distribution function $f_2$, corresponding to the difference of the distribution functions between the two species, and the  electric field $E$ decay at the faster optimal time rate $(1+t)^{-\frac54}$ in $L^2$-norm. In particular, the macroscopic part of the distribution function $f_2$ decay at exponentially and and microscopic part of the distribution function $f_2$ decays at the optimal rate $(1+t)^{-\frac54}$ (refer to Theorem~\ref{time1} for details). Here we recall that the macroscopic part and microscopic part related to the distribution function $f_1$ decay at the different optimal rates $(1+t)^{-\frac34}$ and $(1+t)^{-\frac54}$ respectively  as shown in \cite{Zhong2012Sci}. These optimal algebraic time decay rates also established for the distribution solution $(f_1,f_2,E,B)$ to the nonlinear two-species VMB system, and we are able to show that the distribution function $f_2$ and the electric field $E$ converge to zero state at the faster optimal time rate  $(1+t)^{-\frac54}$ than the optimal rate $(1+t)^{-\frac34}$ for the function $f_1$ and the electric field $B$  (refer to Theorem~\ref{time3} and Theorem~\ref{time4}  for details).

For one-species VMB system, some more subtle phenomena on long time behaviors of charge transport of  charge transport are observed. Indeed, we can show that there are different long time behaviors of global solutions to one-species VMB system characterized and dominated by the effect of either the magnetic field or the electric field, which depends on whether the relation $\nabla_x\cdot E_0=n_0$ holds or not with $n_0$ denoting the first moment of the initial distribution.
We fist prove the phenomena of electric field dominating in the case that $\nabla_x\cdot E_0\neq n_0$ for the linearized VMB system. Namely, we show that both the distribution function and the electric field tend to the equilibrium state at the optimal decay rate $(1+t)^{-\frac14}$ in $L^2$-norm which is slower than the faster optimal convergence rate $(1+t)^{-\frac38}$ in $L^2$-norm of the magnetic field to equilibrium state, and in particular the macroscopic density,  momentum and energy, corresponding to the first three moments of the distribution function $f$, decay at the optimal rates $(1+t)^{-\frac34}$, $(1+t)^{-\frac14}$, and $(1+t)^{-\frac34}$ respectively (refer to Theorem~\ref{time1a} for details). This phenomena of electric field dominating has not been observed before. However, one can not extend these linear theory to the nonlinear VMB system although the global existence of strong solution can be established (refer to Theorem~\ref{time3a} for details), because these optimal time decay rates are too weak to be empolyed to control the expected long time rates of nonlinear terms  (refer to Remark~\ref{rem3} for some verification).
%

%
In the case that $\nabla_x\cdot E_0= n_0$ we prove the phenomena of magnetic field dominating for both linearized and nonlinear VMB system. Indeed, we are able to show that the magnetic field tend to zero at the optimal  time rate $(1+t)^{-\frac38}$ in $L^2$-norm which is slower than the optimal time decay rate $(1+t)^{-\frac58}$ of the distribution function and the optimal time decay rate $(1+t)^{-\frac34}$ of the electric field in $L^2$-norm. In particular, the macroscopic density, momentum and  energy related to the distribution function $f$ are shown to decay at the different optimal rates $(1+t)^{-\frac54}$, $(1+t)^{-\frac58}$ and $(1+t)^{-\frac34}$ respectively (refer to Theorem~\ref{time2a} for details). Furthermore, we can also study rigorously the time-asymptotical behaviors of global solutions to the nonlinear VMB system and in particular obtain the optimal time decay rate $(1+t)^{-\frac58}$  of the distribution function $f$, the optimal time decay rate $(1+t)^{-\frac38}$  of the magnetic fields $B$, and the faster time decay rate $(1+t)^{-\frac34}\ln(1+t)$  of electric field $E$ (refer to  Theorem~\ref{time3a} for details).
This gives more information than those obtained by energy method in \cite{Duan4,Duan5} where only the upper bound of the decay rate of total energy was obtained. 
Here, we should mention that the time-convergence rate $(1+t)^{-\frac38}$ of the total energy of global solution for linearized one-species VMB system in \cite{Duan4} indeed corresponds to the case of the magnetic field dominating phenomena but without the above analysis in details made in the present paper, and the results corresponding to the electric field dominating phenomena is never observed before. In particular, we obtain the optimal decay rates of the solution to the nonlinear one-species VMB system (refer to Theorem~\ref{time3a} for details) which was not solved in \cite{Duan4}.

The rest of this paper will be organized as follows. In  Section~\ref{result}, the main results on spectrum structures and time-asymptotic behaviors of global solutions for two-species VMB and one-species VMB are stated in section 2.1 and section 2.2 respectively.
In Section~\ref{spectrum-two} and Section~\ref{spectrum-one}, the spectrum  structures of the two linearized systems for both two-species and one-species charge motion will be analyzed with detailed description in low and high frequency regions.
Based on this analysis on the linearized operators, the decomposition of the corresponding semigroups generated by these operators will be
given in Section~\ref{behavior-linear} together with the optimal convergence rates to the equilibrium in time.
The optimal convergence rates of the global solution to the original nonlinear system will be studied in the last section.

\section{Main results}
\label{result}

\subsection{Two-species VMB system}

We first consider the Cauchy problem for the two-species Valsov-Maxwell-Boltzmann system \eqref{VMB1z} as follows
 \be
 \left\{\bln     \label{VMB1}
 &  \dt F_++v\cdot\Tdx F_++(E+v\times B)\cdot\Tdv F_+ =\Q(F_+,F_+)+\Q(F_+,F_-),\\
 &  \dt F_-+v\cdot\Tdx F_--(E+v\times B)\cdot\Tdv F_- =\Q(F_-,F_+)+\Q(F_-,F_-), \\
 & \dt E=\Tdx\times B-\intr (F_+-F_-)vdv,\\
 & \dt B=-\Tdx\times E,\\
 & \Tdx\cdot E=\intr (F_+-F_-)dv,\quad \Tdx\cdot B=0,\\
 & F_{\pm}(0,x,v)=F_{\pm,0}(x,v), \quad  E(0,x)=E_0(x),\quad B(0,x)=B_0(x).
 \eln\right.
 \ee

In order to study the spectrum structure of the system \eqref{VMB1}, it
is convenient to consider the following system for
$ F_1=F_++F_-$ and $ F_2=F_+-F_-$ that takes care of the cancellation
in the original system:
\be
 \left\{\bln     \label{VMB2}
&\dt F_1+v\cdot\Tdx F_1+(E+v\times B)\cdot\Tdv F_2
=\Q(F_1,F_1),\\
&\dt F_2+v\cdot\Tdx F_2+(E+v\times B)\cdot\Tdv F_1
=\Q(F_2,F_1),\\
& \dt E=\Tdx\times B-\intr F_2vdv,\\
&  \dt B=-\Tdx\times E,\\
& \Tdx\cdot E=\intr F_2dv,\quad \Tdx\cdot B=0,\\
&F_1(0,x,v)=F_{1,0}=F_{+,0}+F_{-,0},\quad F_2(0,x,v)=F_{2,0}=F_{+,0}-F_{-,0}.
\eln\right.
 \ee

In the following, we will consider the spectrum of the operator
by linearizing the  system \eqref{VMB2} around an equilibrium
state $(F^*_1,F^*_2,E^*,B^*)=(M(v),0,0,0)$ with $M(v)$ being
the normalized Maxwellian given by
$$
 M=M(v)=\frac1{(2\pi)^{3/2}}e^{-\frac{|v|^2}2},\quad v\in\R^3.
$$
Set
$$F_1=M+\sqrt M f_1,\quad F_2=\sqrt M f_2.$$
Then the  system \eqref{VMB2} for $(F_1,F_2,E,B)$ can be written as the following
system for $(f_1,f_2,E,B)$:
\bgr
\dt f_1+v\cdot\Tdx f_1-Lf_1
=\frac12(v\cdot E)f_2-(E+v\times B)\cdot \Tdv f_2+\Gamma(f_1,f_1),\label{VMB3}\\
\dt f_2+v\cdot\Tdx f_2-L_1f_2-v\sqrt M\cdot E
=\frac12(v\cdot E)f_1-(E+v\times B)\cdot \Tdv f_1+\Gamma(f_2,f_1),\label{VMB3a}\\
 \dt E=\Tdx\times B-\intr f_2v\sqrt Mdv,\label{VMB3b}\\
  \dt B=-\Tdx\times E,\label{VMB3c}\\
\Tdx\cdot E=\intr f_2\sqrt Mdv,\quad \Tdx\cdot B=0,\label{VMB3d}\\
f_1(0,x,v)=\frac{F_{1,0}-M}{\sqrt M},\quad f_2(0,x,v)=\frac{F_{2,0}}{\sqrt M},\quad  E(0,x)=E_0,\quad B(0,x)=B_0,\label{VMB3e}
\egr
where
 \bmas
Lf&=\frac1{\sqrt M}[\Q(M,\sqrt{M}f)+\Q(\sqrt{M}f,M)],\\
L_1f&=\frac1{\sqrt M}\Q(\sqrt{M}f,M),\quad
\Gamma(f,g)=\frac1{\sqrt M}\Q(\sqrt{M}f,\sqrt{M}g).\emas
The linearized collision operators $L$ and $L_1$ can be written as, cf. \cite{Cercignani,Yu},
 \bmas
(Lf)(v)&=(Kf)(v)-\nu(v) f(v),\quad (L_1f)(v)=(K_1f)(v)-\nu(v) f(v),\\
\nu(v)&=\intr\ints |(v-v_*)\cdot\omega|M_*d\omega dv_*,\\
(Kf)(v)&=\intr\ints
|(v-v_*)\cdot\omega|(\sqrt{M'_*}f'+\sqrt{M'}f'_*-\sqrt{M}f_*)\sqrt{M_*}d\omega
dv_*\\
&=\intr k(v,v_*)f(v_*)dv_*,\\
(K_1f)(v)&=\intr\ints
|(v-v_*)\cdot\omega|\sqrt{M'_*}\sqrt{M_*}f'd\omega
dv_*=\intr k_1(v,v_*)f(v_*)dv_*,
\emas
where $\nu(v)$ is called
the collision frequency, $K$ and $K_1$ are self-adjoint compact operators
on $L^2(\R^3_v)$ with real symmetric integral kernels $k(v,v_*)$ and $k_1(v,v_*)$.
The null space of the operator $L$, denoted by $N_0$, is a subspace
spanned by the orthonoraml basis $\{\chi_j,\ j=0,1,\cdots,4\}$  given by
\bq \chi_0=\sqrt{M},\quad \chi_j=v_j\sqrt{M} \ (j=1,2,3), \quad
\chi_4=\frac{(|v|^2-3)\sqrt{M}}{\sqrt{6}};\label{basis}\eq
and the null space of the operator $L_1$, denoted by $N_1$, is
spanned only by $\sqrt{M}$.

For later use,
 denote by $L^2(\R^3)$ the Hilbert space of complex valued functions
on $\R^3$ with the inner product and the norm given by
$$
(f,g)=\intr f(v)\overline{g(v)}dv,\quad \|f\|=\(\intr |f(v)|^2dv\)^{1/2}.
$$
And let $\P_0,P_{\rm d}$ be the projection operators from $L^2(\R^3_v)$ to the subspace $N_0, N_1$ with
\bma
 &\P_0f=\sum_{i=0}^4(f,\chi_i)\chi_i,\quad \P_1=I-\P_0, \label{P10}
 \\
 &P_{\rm d}f=(f,\sqrt M)\sqrt M,   \quad P_r=I-P_{\rm d}. \label{Pdr}
 \ema

From the Boltzmann's H-theorem, the linearized collision operators $L$ and $L_1$ are non-positive, precisely,  there is a constant $\mu>0$ such that \bma
 (Lf,f)&\leq -\mu \| \P_1f\|^2, \quad  \ f\in D(L),\\
 (L_1f,f)&\leq -\mu \|P_rf\|^2, \quad  \ f\in D(L_1),\label{L_4}
 \ema
where $D(L)$ and $D(L_1)$ are the domains of $L$ and $L_1$ given by
$$ D(L)=D(L_1)=\left\{f\in L^2(\R^3)\,|\,\nu(v)f\in L^2(\R^3)\right\}.$$
In addition, for the hard sphere model, $\nu$ satisfies
 \be
\nu_0(1+|v|)\leq\nu(v)\leq \nu_1(1+|v|).  \label{nuv}
 \ee
Without the loss of generality, we choose $\nu(0)\ge \nu_0\ge \mu>0$ throughout this paper.

From the system \eqref{VMB3}--\eqref{VMB3d} for $(f_1,f_2,E,B)$, we have the following decoupled
linearized  system for $f_1$ and $(f_2, E, B)$:
\bgr
\dt f_1+v\cdot\Tdx f_1-Lf_1
=0,\label{VMB4a}\\
\dt f_2+v\cdot\Tdx f_2-L_1f_2-v\sqrt M\cdot E
=0,\label{VMB5a}\\
\dt E=\Tdx\times B-\intr f_2v\sqrt Mdv,\label{VMB6a}\\
\dt B=-\Tdx\times E,\label{VMB7a}\\
\Tdx\cdot E=\intr f_2\sqrt Mdv,\quad \Tdx\cdot B=0.\label{VMB8a}
\egr
The equation \eqref{VMB4a} is simply the linearized Boltzmann equation around a global Maxwellian
so that its spectrum structure is well established since 1970s. 
Therefore, we only need to study the spectrum structure  of  the linear system~\eqref{VMB5a}--\eqref{VMB8a} on $(f_2, E, B)$.

\def\BB{\mathbb{B}}
\def\AA{\mathbb{A}}

For convenience of notations,
rewrite the linearized system for $f_1\in L^2(\R^3_x\times\R^3_v)$ and $U=(f_2,E,B)^T\in L^2(\R^3_x\times\R^3_v)\times L^2(\R^3_x)\times L^2(\R^3_x)$ as
\bq
 \dt f_1=\BB_0f_1,\quad f_1(0,x,v)=f_{1,0}(x,v),\label{LVMB0}
 \eq
and
 \bq
 \left\{\bln            \label{LVMB1a}
 &\dt U=\AA_0U, \quad t>0,\\
 &\Tdx\cdot E=(f_2,\sqrt M),\quad \Tdx\cdot B=0,\\
 &U(0,x,v)=U_0(x,v)=(f_{2,0},E_0,B_0),
 \eln\right.
 \eq
where the operators $\BB_0$ and $\AA_0$ are operators on $L^2(\R^3_x\times\R^3_v)$ and $L^2(\R^3_x\times\R^3_v)\times L^2(\R^3_x)\times L^2(\R^3_x)$ defined by
\bma
\BB_0&=L-(v\cdot\Tdx),\\
\AA_0 &=\left( \ba
L_1-(v\cdot\Tdx) &v\sqrt M &0\\
-P_m &0 &\Tdx\times\\
0 &-\Tdx\times &0
\ea\right),
\ema
with
\bq P_mf=(f,v\sqrt M).\eq

Take the Fourier transform in \eqref{LVMB1a} with respect to $x$ to have
 \bq
 \left\{\bln            \label{LVMB1}
 &\dt \hat{U}=\hat{\AA}_0(\xi)\hat{U}, \quad t>0,\\
 &\i(\xi\cdot \hat{E})=(\hat{f}_2,\sqrt M),\quad \i(\xi\cdot \hat{B})=0,\\
 &\hat{U}(0,\xi,v)=\hat{U}_0(\xi,v)=(\hat{f}_{2,0},\hat{E}_0,\hat{B}_0),
 \eln\right.
 \eq where $\hat{\AA}_0(\xi)$ is the Fourier transform of  $\AA_0$.

Note that it is difficult to study the spectrum structure of the system \eqref{LVMB1}
directly because of the constraints on $\hat{E}$ and $\hat{B}$ given in the
second and third equations. One of the key observations in this paper is that
by using the identity $F=(F\cdot y)y-y\times y\times F$ for any
$F\in \R^3$ and  $y\in \mathbb{S}^2$,
we can firstly solve for $\hat{V}=(\hat f_2,\omega\times\hat  E,\omega\times \hat B)$ with $\omega=\xi/|\xi|$
so that by combining these two constraints, we have the full information on
$\hat U$. For this reason, we consider the following system for $\hat{V}$: 
 \bq
 \left\{\bln            \label{LVMB2a}
 &\dt \hat{V}=\hat{\AA}_1(\xi)\hat{V}, \quad t>0,\\
 &\hat{V}(0,\xi,v)=\hat{V}_0(\xi,v)=(\hat{f}_{2,0},\omega\times\hat{E}_0,\omega\times\hat{B}_0),
 \eln\right.
 \eq
with
\bma
\hat{\AA}_1(\xi)=\left( \ba
\hat{B}_1(\xi) &-v\sqrt M\cdot\omega\times &0\\
-\omega\times P_m &0 &\i\xi\times\\
0 &-\i\xi\times &0
\ea\right).
\ema
Here,  for $\xi\ne0$,
 \bq
\hat{B}_1(\xi) =L_1-\i(v\cdot\xi)-\mbox{$\frac{\i(v\cdot\xi)}{|\xi|^2}$}P_{\rm d}.  \label{B(xi)1}
 \eq

Before further discussion, we will give a remark on the eigenvalues
and eigenfunctions of the original system and the reduced system
\eqref{LVMB2a}.

\begin{remark}\label{rem1.1} Set
$F_+=\frac12M+\sqrt{M}f_+,\,  F_-=\frac12M+\sqrt{M}f_-$ to have
\bq \label{LVMB3}
\left\{\bln
&\dt f_{\pm}+v\cdot\Tdx f_{\pm}-\frac12(L\pm L_1)f_{\pm}-\frac12(L\mp L_1)f_{\pm}\mp\frac12v\sqrt M\cdot E=0, \\
 &\dt E=\Tdx\times B-\intr (f_+-f_-)v\sqrt Mdv, \quad
  \dt B=-\Tdx\times E, \\
&\Tdx\cdot E=\intr (f_+-f_-)\sqrt Mdv,\quad \Tdx\cdot B=0.
\eln \right.
\eq
The eigenvalues of the system \eqref{LVMB3} are same as those
of \eqref{LVMB0} and \eqref{LVMB2a},  and the eigenfunctions of the system  \eqref{LVMB3} can be obtained as linear combinations
 of those for  \eqref{LVMB0} and \eqref{LVMB2a}. In fact, let $\lambda$  be an eigenvalue with the corresponding eigenfunction denoted by
$\phi$ of \eqref{LVMB0}, and $\beta$  be an eigenvalue with
 the corresponding eigenvector denoted by $\Psi=(\psi,E,B)$ of \eqref{LVMB2a}. Then $U=(\phi,\phi,0,0)$ is the corresponding
 eigenvector with the eigenvalue $\lambda$, and $V=(\psi,-\psi, -\frac{\i \xi}{|\xi|^2}(\psi,\chi_0)-\frac{\xi}{|\xi|}\times E,-\frac{\xi}{|\xi|}\times B)$ is
 the corresponding
eigenvector with the eigenvalue $\beta$.
\end{remark}

By the above argument, from now on, we will focus  on  the system~\eqref{LVMB1a}. It is interesting to find out that the spectrum
structure depends  on the decomposition
of the asymptotics in low frequency and high frequency
  like the classical Boltzmann equation and the Vlasov-Poisson-Boltzmann system, but depends on the low-intermediate-high  frequencies. This  is mainly due to the hyperbolic structure of the Maxwell equations, in particular the effect of the magnetic field on the velocity field. More precisely,  the spectrum of linearized operator contain an eigenvalue in low frequency located in a small neighborhood of the origin, and two eigenvalues in high frequency located in two small neighborhoods centered at  the points $\lambda=\pm\i|\xi|$ respectively. Except these eigenvalues, there is  a spectral gap for the intermediate frequency.
Precisely, we have the following result on the spectrum structure.

 \begin{thm}\label{spectrum1}
There exist two constants $r_0>0$ and $b_2>0$ so that the spectrum $\lambda\in\sigma(\hat{\AA}_1(\xi))\subset\mathbb{C}$ for $\xi=s\omega$ with $s=|\xi|\leq r_0$ and $\omega\in \mathbb{S}^2$ consists of two points $\{\lambda_j(s),\ j=1,2\}$ in the domain $\mathrm{Re}\lambda>-b_2$. The spectrum $\lambda_j(s)$ are $C^\infty$ functions of $s$ for $|s|\leq r_0$. In particular, the eigenvalues admit the following asymptotic expansion for $|s|\leq r_0$
 $$                                  
 \lambda_{1}(s) =\lambda_{2}(s) = -a_1s^2+o(s^2),
 $$
where $a_1>0$ is a constant defined in Theorem \ref{eigen_3}.

There exists a constant $r_1>0$ such that the spectrum $\beta\in\sigma(\hat{\AA}_1(\xi))\subset\mathbb{C}$ for  $s=|\xi|> r_1$ consists of four eigenvalues $\{\beta_j(s),\ j=1,2,3,4\}$ in the domain $\mathrm{Re}\lambda>-\mu/2$. In particular, the eigenvalues satisfy
 \bgrs
 \beta_1(s) = \beta_2(s) =-\i s+O(s^{-1/2}),
 \\
 \beta_3(s) = \beta_4(s) =\i s+O(s^{-1/2}),
 \\
\frac{c_1}{s}\le -{\rm Re}\beta_{j}(s)\le \frac{c_2}{s},
\egrs
for two positive constants $c_1$ and $c_2$.

For any $r_1>r_0>0$, there
exists $\alpha =\alpha(r_0,r_1)>0$ such that for  $r_0\le |\xi|\le r_1$,
$$ \sigma(\hat{\AA}_1(\xi))\subset\{\lambda\in\mathbb{C}\,|\, \mathrm{Re}\lambda(\xi)\leq-\alpha\} .$$
\end{thm}

Based on  the spectrum structure given in the above
theorem, the semigroup generated by the linearized operator of the Vlasov-Maxwell-Boltmann system can be decomposed in three parts according to the frequency in low, high and intermediate regions so that the optimal time decay rates of the solution can be obtained for the linearized system. Note that a higher regularity on the  initial data is needed because of the spectrum structure in the high frequency region.
With the estimates on the semigroup, the optimal decay in time of the solution to the original
nonlinear problem can also be obtained. Before stating  results on nonlinear problem, let us first introduce the following notations.\\

\noindent\textbf{Notations:} \ \ The Fourier transform of $f=f(x,v)$
is denoted by
$\hat{f}(\xi,v)=\mathcal{F}f(\xi,v)=\frac1{(2\pi)^{3/2}}\intr f(x,v)e^{- \i x\cdot\xi}dx.$

Set a weight function $w(v)$ by
$$
w(v)=(1+|v|^2)^{1/2},
$$
so that the Sobolev spaces $ H^N$ and $ H^N_w$ are given by
$$
 H^N=\{\,f\in L^2(\R^3_x\times \R^3_v)\,|\,\|f\|_{H^N}<\infty\,\},\quad
 H^N_w=\{\,f\in L^2(\R^3_x\times \R^3_v)\,|\,\|f\|_{H^N_w}<\infty\,\},
$$
equipped with the norms
$$
 \|f\|_{H^N}=\sum_{|\alpha|+|\beta|\le N}\|\dxa\dvb f\|_{L^2(\R^3_x\times \R^3_v)},
 \quad
 \|f\|_{H^N_w}=\sum_{|\alpha|+|\beta|\le N}\|w\dxa\dvb f\|_{L^2(\R^3_x\times \R^3_v)}.
$$
For $q\ge1$, denote
$$
L^{2,q}=L^2(\R^3_v,L^q(\R^3_x)),\quad
\|f\|_{L^{2,q}}=\bigg(\intr\bigg(\intr|f(x,v)|^q dx\bigg)^{2/q}dv\bigg)^{1/2}.
$$
In the following,  denote by $\|\cdot\|_{L^2_{x,v}}$ and $\|\cdot\|_{L^2_{\xi,v}}$ the norms of the function spaces $L^2(\R^3_x\times \R^3_v)$ and $L^2(\R^3_\xi\times \R^3_v)$ respectively, and by $\|\cdot\|_{L^2_x}$, $\|\cdot\|_{L^2_\xi}$ and $\|\cdot\|_{L^2_v}$  the norms of the function spaces $L^2(\R^3_x)$, $L^2(\R^3_\xi)$ and $L^2(\R^3_v)$ respectively. 
For any integer $m\ge1$, denote by $\|\cdot\|_{H^m_x}$ and $\|\cdot\|_{L^2_v(H^m_x)}$ the norms in the spaces $H^m(\R^3_x)$ and $L^2(\R^3_v,H^m(\R^3_x))$ respectively.
Moreover, introduce a weighted Hilbert space $L^2_\xi(\R^3_v)$ for $\xi\ne 0$
defined by
$$
 L^2_\xi(\R^3)=\{f\in L^2(\R^3_v)\,|\,\|f\|_\xi=\sqrt{(f,f)_\xi}<\infty\},
$$
equipped with the inner product
$$
 (f,g)_\xi=(f,g)+\frac1{|\xi|^2}(P_{\rm d} f,P_{\rm d} g).
$$
For any vectors $U=(f,E_1,B_1),V=(g,E_2,B_2)\in L^2_\xi(\R^3_v)\times \mathbb{C}^3\times \mathbb{C}^3$,  define a weighted inner product and the corresponding
norm by
$$ (U,V)_\xi=(f,g)_\xi+(E_1,E_2)+(B_1,B_2),\quad \|U\|_\xi=\sqrt{(U,U)_\xi}, $$
and another $L^2$ inner product and norm by
$$ (U,V)=(f,g)+(E_1,E_2)+(B_1,B_2),\quad \|U\|=\sqrt{(U,U)}. $$
For simplicity, denote
$$ \|U\|^2_{Z^2}=\|f\|^2_{L^{2}_{x,v}}+\|E\|^2_{L^2_x}+\|B\|^2_{L^2_x},\quad \|U\|^2_{Z^1}=\|f\|^2_{L^{2,1}}+\|E\|^2_{L^1_x}+\|B\|^2_{L^1_x}.$$

With the above preparation, we first state the estimates on the semigroup  to linearized system.
\begin{thm}\label{semigroup-1}
The semigroup $S(t,\xi)=e^{t\hat{\AA}_1(\xi)}$ with $\xi=s\omega\in \R^3$ and $s=|\xi|\neq0$  can be decomposed into
 $$
 S(t,\xi)U=S_1(t,\xi)U+S_2(t,\xi)U+S_3(t,\xi)U,
     \quad U\in L^2_\xi(\R^3_v)\times \mathbb{C}^3_\xi\times \mathbb{C}^3_\xi, \ \ t>0, 
 $$
 that has the following properties
 \bmas
 S_1(t,\xi)U&=\sum^2_{j=1}e^{t\lambda_j(s)}
              (U, \Psi^*_j(s,\omega)\,) \Psi_j(s,\omega)
               1_{\{|\xi|\leq r_0\}}, \\
 S_2(t,\xi)U&=\sum^4_{j=1} e^{t\beta_j(s)}
              (U,\Phi^*_j(s,\omega)\,) \Phi_j(s,\omega)
               1_{\{|\xi|\ge r_1\}},
 \emas
where $\mathbb{C}^3_\xi=\{y\in \mathbb{C}^3: y\cdot \xi=0\}$.
Here,  $(\lambda_j(s),\Psi_j(s,\omega))$ and $(\beta_j(s),\Phi_j(s,\omega))$
are the eigenvalues and eigenvectors of the operator $\hat{\AA}_1(\xi)$ in low
and high frequency regions with properties
given in Theorem~\ref{eigen_3} and Theorem~\ref{eigen_4}.
And $S_3(t,\xi)U =: S(t,\xi)U-S_1(t,\xi)U-S_2(t,\xi)U$ satisfies that there
is  a constant $\kappa_0>0$ independent of $\xi$ such that
 $$
 \|S_3(t,\xi)U\|_\xi\leq Ce^{-\kappa_0t}\|U\|_\xi,\quad t\ge0.
 $$
\end{thm}

Then, we have the optimal time convergence rates of global solutions to linearized system~\eqref{LVMB1a}.

\begin{thm}
 \label{time1}
 Let $(f_2(t),E(t),B(t))$ be a solution of the system \eqref{LVMB1a}. If the initial data $U_0=(f_0,E_0,B_0)\in L^2(\R^3_{v};H^l(\R^3_{x}) \cap L^1(\R^3_{x}))\times H^l(\R^3_x)\cap L^1(\R^3_{x})\times H^l(\R^3_x)\cap L^1(\R^3_{x})$ for $l\ge 0$, then it holds for any $\alpha,\alpha'\in\N^3$ with $\alpha'\le \alpha$ that
 \bma
 \|\da_x f_2(t)\|_{L^2_{x,v}}
\leq
&C(1+t)^{-\frac54-\frac{k}2}
  (\|\da_x U_0\|_{Z^2}+\|\dx^{\alpha'}U_0\|_{Z^1})+C(1+t)^{-(m+\frac12)}\|\Tdx^{m}\da_x U_0\|_{Z^2},\label{D_2x}
\\
\|\da_x E(t)\|_{L^2_x}
\leq
&C(1+t)^{-\frac54-\frac{k}2}
  (\|\da_x U_0\|_{Z^2}+\|\dx^{\alpha'}U_0\|_{Z^1})+C(1+t)^{-m}\|\Tdx^{m}\da_x U_0\|_{Z^2}, \label{D_4}
\\
\|\dxa B(t)\|_{L^2_x}
\leq
&C(1+t)^{-\frac34-\frac{k}2}
  (\|\da_x U_0\|_{Z^2}+\|\dx^{\alpha'}U_0\|_{Z^1})+C(1+t)^{-m}\|\Tdx^{m}\da_x U_0\|_{Z^2},\label{D_0}
 \ema
 where $k=|\alpha-\alpha'|$ and $m\ge 0$. In particular, it holds for $f_2= P_{\rm d}f_2+ P_rf_2$ that
 \bma
\|\da_x P_{\rm d}f_2(t)\|_{L^2_{x,v}}
\leq
&Ce^{-\kappa_0t} \|\da_x U_0\|_{Z^2},\label{D_2}
\\
\|\da_x P_rf_2(t)\|_{L^2_{x,v}}
\leq
&C(1+t)^{-\frac54-\frac{k}2}
  (\|\da_x U_0\|_{Z^2}+\|\dx^{\alpha'}U_0\|_{Z^1})+C(1+t)^{-(m+\frac12)}\|\Tdx^{m}\da_x U_0\|_{Z^2}. \label{D_3}
\ema

Furthermore, if we assume that $l\ge 2$ and
the Fourier transform $\hat B_0(\xi)$ of initial magnetic field $B_0(x)$ satisfies that
$\inf_{|\xi|\le r_0}|\frac{\xi}{|\xi|}\times\hat{B}_0(\xi)|\geq d_0>0$, then
  \bgr
 C_1(1+t)^{-\frac54}
  \leq\| f_2(t)\|_{L^2_{x,v}}\leq C_2(1+t)^{-\frac54},\label{H_1bx}
\\
 C_1(1+t)^{-\frac54}
  \leq\| P_rf_2(t)\|_{L^2_{x,v}}\leq C_2(1+t)^{-\frac54},\label{H_1b}
\\
 C_1(1+t)^{-\frac54}
  \leq\|E(t)\|_{L^2_x}\leq C_2(1+t)^{-\frac54},\label{H_2b}
\\
 C_1(1+t)^{-\frac34}
  \leq\|B(t)\|_{L^2_x}\leq C_2(1+t)^{-\frac34}, \label{H_3}
 \egr
for $t>0$  large enough with $C_2\ge C_1>0$ two constants.
\end{thm}

Finally for the two-species system, we have the optimal time convergence rates of global solutions to the original nonlinear system~\eqref{VMB3}--\eqref{VMB3e} as follows.
\begin{thm}\label{time3} When $(f_{1,0},f_{2,0})\in H^{N+5}_w\cap L^{2,1}$ and $(E_0,B_0)\in H^{N+5}(\R^3_x)\cap L^{1}(\R^3_x)$ for $N\ge 4$ satisfying $\|(f_{1,0},f_{2,0})\|_{H^{N+5}_{w}\cap L^{2,1}}+\|(E_0,B_0)\|_{H^{N+5}(\R^3_x)\cap L^{1}(\R^3_x)}\le \delta_0$ with  $\delta_0>0$ being small, there exists  a globally unique solution $(f_1,f_2,E,B)$ to the  system~\eqref{VMB3}--\eqref{VMB3e} satisfying
\bmas
\|\dx^kf_1(t)\|_{L^2_{x}} +\|\dx^k B(t)\|_{L^2_x} &\le C\delta_0(1+t)^{-\frac34-\frac k2},
\\
\|\dx^k  f_2(t)\|_{L^2_{x,v}}  +\|\dx^k E(t)\|_{L^2_x} &\le C\delta_0(1+t)^{-\frac54-\frac k2},
\emas
for $k=0,1$. In particular, it holds for  $f_1= \P_0f_1+ \P_1f_1$ and $f_2= P_{\rm d}f_2+ P_rf_2$ that
\bmas
\|\dx^k(f_1(t),\chi_j)\|_{L^2_{x}}&\le C\delta_0(1+t)^{-\frac34-\frac k2},\quad j=0,1,2,3,4,
\\
\|\dx^k  \P_1f_1(t)\|_{L^2_{x,v}}&\le C\delta_0(1+t)^{-\frac54-\frac k2},
\\
\|\dx^k  P_{\rm d}f_2(t)\|_{L^2_{x,v}}&\le C\delta_0(1+t)^{-2-\frac k2},
\\
\|\dx^k P_rf_2(t)\|_{L^2_{x,v}}+\|\dx^k E(t)\|_{L^2_x}&\le C\delta_0(1+t)^{-\frac54-\frac k2},
\\
\|\dx^k B(t)\|_{L^2_x}&\le C\delta_0(1+t)^{-\frac34-\frac k2},\\
\|(\P_1f_1,P_rf_2)(t)\|_{H^N_w} +\|\Tdx (\P_0f_1P_{\rm d}f_2)(t)\|_{L^2_v(H^{N-1}_x)}
&+\|\Tdx (E,B)(t)\|_{H^{N-1}_x}\le
C\delta_0(1+t)^{-\frac54},
\emas
for $k=0,1$.
\end{thm}

\begin{thm}\label{time4} %
Under the conditions given in  Theorem~\ref{time3}, if we
 further assume that there exist positive constants $d_0,d_1>0$ and a small constant $r_0>0$ so that the Fourier transform $(\hat{f}_{1,0},\hat{f}_{2,0},\hat E_0,\hat B_0)$ of the initial data $(f_{1,0},f_{2,0},E_0,B_0)$ satisfies that  $\inf_{|\xi|\le r_0}|(\hat f_{1,0}(\xi),\chi_0)|\ge d_0$, $\inf_{|\xi|\le r_0}|(\hat f_{1,0}(\xi),\chi_4)|\ge d_1\sup_{|\xi|\le r_0}|(\hat f_{1,0}(\xi),\chi_0)|$, $\sup_{|\xi|\le r_0}|(\hat f_{1,0}(\xi),v\sqrt M)|=0$ and $\inf_{|\xi|\le r_0}|\frac{\xi}{|\xi|}\times \hat B_0(\xi)|\ge d_0$,  then the global solution $(f_1,f_2,E,B)$  to the  system~\eqref{VMB3}--\eqref{VMB3e} satisfies
 \bmas
 C_1\delta_0(1+t)^{-\frac34}  &\le  \|f_1(t)\|_{L^2_{x,v}}  \le C_2\delta_0(1+t)^{-\frac34},
\\
 C_1\delta_0(1+t)^{-\frac54}&\le  \| f_2(t)\|_{L^2_{x,v}}   \le C_2\delta_0(1+t)^{-\frac54},
 \\
 C_1\delta_0(1+t)^{-\frac54}&\le \|E(t)\|_{L^2_{x}}\le C_2\delta_0(1+t)^{-\frac54},
 \\
 C_1\delta_0(1+t)^{-\frac34}&\le \|B(t)\|_{L^2_{x}}\le C_2\delta_0(1+t)^{-\frac34},
\emas
for $t>0$ large with two constants $C_2>C_1>0$, and in particular
\bmas
 C_1\delta_0(1+t)^{-\frac34}&\le \|(f_1(t),\chi_j)\|_{L^2_{x}}\le C_2\delta_0(1+t)^{-\frac34},  
 \\
 C_1\delta_0(1+t)^{-\frac54}&\le \|  \P_1f_1(t)\|_{L^2_{x,v}}\le C_2\delta_0(1+t)^{-\frac54},
 \\
 C_1\delta_0(1+t)^{-\frac54}&\le \|P_rf_2(t)\|_{L^2_{x,v}}\le C_2\delta_0(1+t)^{-\frac54},
\emas
for $ j=0,1,2,3,4$
\end{thm}

\begin{rem}
From the above theorems, Remark~\ref{rem1.1} and those estimates
obtained in \cite{Li2,Li3}, we  justify the different time-asymptotic phenomena of charged particle transport at mesoscope level
that are determined by the irrotational electric field and
 electromagnetic field respectively. Indeed, it was shown in \cite{Li2} that the distribution function converges to the global Maxwellian at the same optimal rate $(1+t)^{-\frac14}$ as the irrotational electric field for the unipolar Vlasov-Poisson-Boltzmann in $\R^3_x\times\R^3_v$; and in \cite{Li3} it
shows that the  distribution function converges to the global Maxwellian at the optimal rate $(1+t)^{-\frac34}$, which is slower than the exponential decay rate of the irrotational electric field for bipolar Vlasov-Poisson-Boltzmann equations in $\R^3_x\times\R^3_v$.  While in the appearance of electro-magnetic force, the electric field decays at the optimal rate $(1+t)^{-\frac54}$, which is faster than those of the  distribution function and the magnetic field. In particular, it is the rotational vector field part of electric field  for linear system that decays at algebraic rate, while the gradient vector field part of electric field decays exponentially in time.
\end{rem}

\subsection{One-species VMB system}

 We now turn to the one-species
Vlasov-Maxwell-Boltzmann system
where  the time evolution of electrons
is considered under the influence of  a fixed ion background
$n_b(x)$. That is, consider
 \be
 \left\{\bln     \label{1VMB1}
 &  \dt F+v\cdot\Tdx F+(E+v\times B)\cdot\Tdv F =\Q(F,F),\\
 & \dt E=\Tdx\times B-\intr Fvdv,\\
 & \dt B=-\Tdx\times E,\\
 & \Tdx\cdot E=\intr Fdv-n_b,\quad \Tdx\cdot B=0,\\
 & F(x,v, 0)=F_{0}(x,v), \quad  E_0(x,0)=E_0(x),\quad B_0(0,x)=B_0(x).
 \eln\right.
 \ee
Here,  $F=F(x,v,t)$ is the number
density function of   electrons, and  $n_b(x)>0$  is assumed
to be a  constant representing the spatially uniform density of the ionic background. Take  $n_b$ = 1 without loss of generality.

 The one-species VMB system \eqref{1VMB1} has a stationary solution $(F^*,E^*,B^*)=(M(v),0,0)$ with  $M(v)$ being the normalized global Maxwellian defined above.
Set
$$F=M+\sqrt M f.$$
Then the one-species VMB system \eqref{1VMB1} for $(F,E,B)$ is reformulated in terms of $(f,E,B)$ into
\bgr
\dt f+v\cdot\Tdx f-Lf-v\sqrt M\cdot E
=\frac12(v\cdot E)f-(E+v\times B)\cdot \Tdv f+\Gamma(f,f),\label{1VMB3}\\
\dt E=\Tdx\times B-\intr fv\sqrt Mdv,\label{1VMB3a}\\
\dt B=-\Tdx\times E,\label{1VMB3b}\\
\Tdx\cdot E=\intr f\sqrt Mdv,\quad \Tdx\cdot B=0,\label{1VMB3c}\\
f(0,x,v)=\frac{F_{0}-M}{\sqrt M}, \quad  E_0(x,0)=E_0(x),\quad B_0(0,x)=B_0(x).\label{1VMB3d}
\egr

Consider the  linearized one-speices Vlasov-Maxwell-Boltzmann system:
\bgr
\dt f+v\cdot\Tdx f-Lf-v\sqrt M\cdot E
=0,\label{1VMB4}\\
\dt E=\Tdx\times B-\intr fv\sqrt Mdv,\label{1VMB5}\\
\dt B=-\Tdx\times E,\label{1VMB6}\\
\Tdx\cdot E=\intr f\sqrt Mdv,\quad \Tdx\cdot B=0.\label{1VMB7}
\egr

%
For convenience of notations,
rewrite the linearized system for $U=(f,E,B)^T\in L^2(\R^3_x\times\R^3_v)\times L^2(\R^3_x)\times L^2(\R^3_x)$ as
\bq
 \left\{\bln            \label{LVMB}
 &\dt U=\AA_2U, \quad t>0,\\
 &\Tdx\cdot E=(f,\sqrt M),\quad \Tdx\cdot B=0,\\
 &U(0,x,v)=U_0(x,v)=(f_0,E_0,B_0),
 \eln\right.
 \eq
with the operator $\AA_2$ defined by
\bma
\AA_2 &=\left( \ba
L-v\cdot\Tdx &v\sqrt M &0\\
-P_m &0 &\Tdx\times\\
0 &-\Tdx\times &0
\ea\right).
\ema

Similarly,  consider the following system for $\hat{V}=(\hat f,\omega\times\hat  E,\omega\times \hat B)$:
 \bq
 \left\{\bln            \label{LVMB2}
 &\dt \hat{V}=\hat{\AA}_3(\xi)\hat{V}, \quad t>0,\\
 &\hat{V}(0,\xi,v)=\hat{V}_0(\xi,v)=(\hat{f}_0,\omega\times\hat{E}_0,\omega\times\hat{B}_0),
 \eln\right.
 \eq
with 
\bma
\hat{\AA}_3(\xi)=\left( \ba
\hat{B}_2(\xi) &-v\sqrt M\cdot\omega\times &0\\
-\omega\times P_m &0 &\i\xi\times\\
0 &-\i\xi\times &0
\ea\right).
\ema
Here,  for $\xi\ne0$,
 \bma
\hat{B}_2(\xi) &=L-\i(v\cdot\xi)-\mbox{$\frac{\i(v\cdot\xi)}{|\xi|^2}$}P_{\rm d}.  \label{B(xi)}
 \ema


Different from the two-species VMB, the spectrum of the linearized operator of one-species VMB consists of nine eigenvalues in low frequency located in small neighborhoods of three points $0,\pm \i$, and two eigenvalues in high frequency located in two small neighborhoods centered at  the points $\lambda=\pm\i|\xi|$ respectively in complex plane. Except these eigenvalues, there is a spectral gap for the intermediate frequency. More precisely, we have

 \begin{thm}\label{spectrum2}
There exists a constant $r_0>0$ so that the spectrum $\lambda\in\sigma(\hat{\AA}_3(\xi))\subset\mathbb{C}$ for $\xi=s\omega$ with $|s|\leq r_0$ and $\omega\in \mathbb{S}^2$ consists of nine points $\{\lambda_j(s),\ -1\le j\le 7\}$ in the domain $\mathrm{Re}\lambda>-\mu /2$. The spectrum $\lambda_j(s)$  are $C^\infty$ functions of $s$ for $|s|\leq r_0$. In particular, the eigenvalues $\lambda_j(s)$ have the following asymptotic expansions when $|s|\leq r_0$
 $$                                   
 \left\{\bln
 \lambda_{\pm1}(s)=&\pm \i+(-a_1\pm\i b_1)s^2+o(s^2),\quad  \overline{\lambda_1}=\lambda_{-1},\\
 \lambda_{0}(s) =& -a_0s^2+o(s^2),\\
 \lambda_{2}(s) =& \lambda_{3}(s) =-\i+(-a_2-\i b_2)s^2+o(s^2),\quad \overline{\lambda_{2}}=\lambda_{4},\\
 \lambda_{4}(s) =& \lambda_{5}(s) =\i+(-a_2+\i b_2)s^2+o(s^2),\\
  \lambda_{6}(s) =& \lambda_{7}(s) =-a_3s^4+o(s^4),
 \eln\right.
 $$
where $a_j>0$ $(0\le j\le3)$ and $b_j>0$ $(1\le j\le2)$ are  constants defined in Theorem \ref{eigen_3a}.

There exists a constant $r_1>0$ such that the spectrum $\beta\in\sigma(\hat{\AA}_3(\xi))\subset\mathbb{C}$ for  $s=|\xi|> r_1$ consists of four eigenvalues $\{\beta_j(s),\ j=1,2,3,4\}$ in the domain $\mathrm{Re}\lambda>-\mu/2$. In particular, the eigenvalues satisfy
 \bgrs
 \beta_1(s) = \beta_2(s) =-\i s+O(s^{-1/2}),
 \\
 \beta_3(s) = \beta_4(s) =\i s+O(s^{-1/2}),
 \\
\frac{c_1}{s}\le -{\rm Re}\beta_{j}(s)\le \frac{c_2}{s},
\egrs
for two positive constants $c_1$ and $c_2$.

For any $r_1>r_0>0$, there
exists $\alpha =\alpha(r_0,r_1)>0$ such that for  $r_0\le |\xi|\le r_1$,
$$ \sigma(\hat{\AA}_3(\xi))\subset\{\lambda\in\mathbb{C}\,|\, \mathrm{Re}\lambda(\xi)\leq-\alpha\} .$$
\end{thm}

Based on the spectrum of the linearized operator of one-species VMB system, we are able to analyze the corresponding semigroup for \eqref{LVMB2} below.

\begin{thm}\label{semigroup-2}
The semigroup $S(t,\xi)=e^{t\hat{\AA}_3(\xi)}$ with $\xi=s\omega\in \R^3$ and $s=|\xi|\neq0$  satisfies
 $$
 S(t,\xi)U=S_1(t,\xi)U+S_2(t,\xi)U+S_3(t,\xi)U,
     \quad U\in L^2_\xi(\R^3_v)\times \mathbb{C}^3_\xi\times \mathbb{C}^3_\xi, \ \ t>0, 
 $$
 where
 \bmas
 S_1(t,\xi)U&=\sum^7_{j=-1}e^{t\lambda_j(s)}
              (U, \Psi^*_j(s,\omega)\,)_\xi \Psi_j(s,\omega)
               1_{\{|\xi|\leq r_0\}}, 
               \\
 S_2(t,\xi)U&=\sum^4_{j=1} e^{t\beta_j(s)}
              (U,\Phi^*_j(s,\omega)\,) \Phi_j(s,\omega)
               1_{\{|\xi|\ge r_1\}}.   
 \emas
Here, $(\lambda_j(s),\Psi_j(s,\omega))$ and $(\beta_j(s),\Phi_j(s,\omega))$ are the eigenvalues and eigenvectors of the operator $\hat{\AA}_3(\xi)$ given in Theorem~\ref{eigen_3a} and Theorem~\ref{eigen_4a} for $|\xi|\le r_0$ and $|\xi|>r_1$ respectively,
and $S_3(t,\xi)U =: S(t,\xi)U-S_1(t,\xi)U-S_2(t,\xi)U$ satisfies
that there exists  a constant $\kappa_0>0$ independent of $\xi$ so that
 $$
 \|S_3(t,\xi)U\|_\xi\leq Ce^{-\kappa_0t}\|U\|_\xi,\quad t\ge0.
 $$
\end{thm}

Then, we have the optimal time convergence rates of global solutions to linearized system.

\begin{thm}[Electric Field Dominating]
\label{time1a}
If the initial data  $U_0=(f_0,E_0,B_0)\in L^2(\R^3_{v};H^l(\R^3_{x}) \cap L^1(\R^3_{x}))\times H^l(\R^3_x)\cap L^1(\R^3_{x})\times H^l(\R^3_x)\cap L^1(\R^3_{x})$ for $l\ge 0$ with $\Tdx\cdot E_0\ne (f_0,\chi_0)$ being held,
then the unique solution $(f(t),E(t),B(t)) $ to the system \eqref{LVMB} exists globally in time and satisfies  for any $\alpha,\alpha'\in\N^3$ with $\alpha'\le \alpha$ that
 \bma
\|\da_x f(t)\|_{L^2_{x,v}}
\leq
& C\delta(\alpha,\alpha')[(1+t)^{-\frac14-\frac{k}2}+(1+t)^{-\frac58-\frac{k}4}]
  +C(1+t)^{-m-\frac12}\|\Tdx^{m}\da_x U_0\|_{Z^2},  \label{2D_3a}
\\
\|\da_x E(t)\|_{L^2_x}
\leq
& C\delta(\alpha,\alpha')[(1+t)^{-\frac14-\frac{k}2}+(1+t)^{-\frac58-\frac{k}4}]
  +C(1+t)^{-m}\|\Tdx^{m}\da_x U_0\|_{Z^2},  \label{2D_3b}
\\
\|\dxa B(t)\|_{L^2_x}
\leq
& C\delta(\alpha,\alpha')(1+t)^{ -\frac38-\frac{k}4}
  +C(1+t)^{-m}\|\Tdx^{m}\da_x U_0\|_{Z^2},   \label{2D_0}
 \ema
where and below $\delta(\alpha,\alpha')=:\|\da_x U_0\|_{Z^2}+\|\dx^{\alpha'}U_0\|_{Z^1}$, $k=|\alpha-\alpha'|$ and $m\ge 0$. In particular, it holds for $f= \P_0f+ \P_1f$ that
 \bma
\|\da_x (f(t),\chi_0)\|_{L^2_{x}}
\leq
&C\delta(\alpha,\alpha')(1+t)^{-\frac34-\frac{k}2},\label{2D_2}
\\
\|\da_x (f(t),v\sqrt M)\|_{L^2_{x}}
\leq
&C\delta(\alpha,\alpha')[(1+t)^{-\frac14-\frac{k}2}+(1+t)^{-\frac58-\frac{k}4}]
  +C(1+t)^{-m-\frac12}\|\Tdx^{m}\da_x U_0\|_{Z^2},   \label{2D_3}
\\
\|\da_x (f(t),\chi_4)\|_{L^2_{x}}
\leq
&C\delta(\alpha,\alpha')(1+t)^{-\frac34-\frac{k}2},\label{2D_4}
\\
  \|\da_x \P_1f(t)\|_{L^2_{x,v}}
\leq
&C\delta(\alpha,\alpha')[(1+t)^{-\frac34-\frac{k}2}+(1+t)^{-\frac78-\frac{k}4}]
 +C(1+t)^{-m-\frac12}\|\Tdx^{m}\da_x U_0\|_{Z^2}.      \label{2D_5}
\ema

Furthermore, if we assume further that  $l\ge 1$ and there exist two constants $d_0,d_1>0$ such that
   the Fourier transform $\hat U_0=(\hat f_0,\hat E_0,\hat B_0)$  of the initial data $U_0$  satisfies that
  $\inf_{|\xi|\le r_0}|(\hat{f}_0(\xi),\chi_0)|\geq d_0$, $\inf_{|\xi|\le r_0}|\frac{\xi}{|\xi|}\times\hat{E}_0(\xi)|\geq d_0$, $\inf_{|\xi|\le r_0}|\frac{\xi}{|\xi|}\times\hat{B}_0(\xi)|\geq d_0$, $\inf_{|\xi|\le r_0}|(\hat{f}_0(\xi),\chi_4)|\geq d_1\sup_{|\xi|\le r_0}|(\hat{f}_0(\xi),\chi_0)|$  and $\sup_{|\xi|\le r_0}|(\hat f_0(\xi),v\sqrt M)|=0$, then the following optimal time decay rates  holds
\bma
   C_1(1+t)^{-\frac14}
  \leq&\| f(t)\|_{L^2_{x,v}}\leq C_2(1+t)^{-\frac14}, \label{2H_2c}
\\
 C_1(1+t)^{-\frac14}
  \leq&\|E(t)\|_{L^2_x}\leq C_2(1+t)^{-\frac14},\label{2H_2a}
\\
 C_1(1+t)^{-\frac38}
  \leq&\|B(t)\|_{L^2_x}\leq C_2(1+t)^{-\frac38}, \label{2H_3a}
 \ema
for $t>0$ being large enough and $C_2\ge C_1>0$ being two constants, and in particular
\bma
  C_1(1+t)^{-\frac34}
  \leq&\|(f(t),\chi_0)\|_{L^2_x}\leq C_2(1+t)^{-\frac34},\label{2H_1}
\\
 C_1(1+t)^{-\frac14}
  \leq&\|( f(t),\chi_j)\|_{L^2_x}\leq C_2(1+t)^{-\frac14},\quad j=1,2,3,\label{2H_2}
\\
 C_1(1+t)^{-\frac34}
  \leq&\|( f(t),\chi_4)\|_{L^2_x}\leq C_2(1+t)^{-\frac34}, \label{2H_3}
\\
  C_1(1+t)^{-\frac34}
  \leq&\| \P_1 f(t)\|_{L^2_{x,v}}\leq C_2(1+t)^{-\frac34}. \label{2Q_2a}
\ema
\end{thm}

\begin{thm}[Magnetic Field Dominating]\label{time2a}
If the initial data $U_0=(f_0,E_0,B_0)\in L^2(\R^3_{v};H^l(\R^3_{x}) \cap L^1(\R^3_{x}))\times H^l(\R^3_x)\cap L^1(\R^3_{x})\times H^l(\R^3_x)\cap L^1(\R^3_{x})$ for $l\ge 0$ with $\Tdx\cdot E_0= (f_0,\chi_0)$ being held, then there exists globally in time a unique solution $(f(t),E(t),B(t))$ to the system \eqref{LVMB} which satisfies for any $\alpha,\alpha'\in\N^3$ with $\alpha'\le \alpha$ that
 \bma
\|\da_x f(t)\|_{L^2_{x,v}}
\leq
 &C\delta(\alpha,\alpha')(1+t)^{-\frac58-\frac{k}4}
   +C(1+t)^{-m-\frac12}\|\Tdx^{m}\da_x U_0\|_{Z^2},  \label{1D_3a}
\\
\|\da_x E(t)\|_{L^2_x}
\leq
 & C\delta(\alpha,\alpha')[(1+t)^{-\frac34-\frac{k}2}+(1+t)^{-\frac98-\frac{k}4}]
 +C(1+t)^{-m}\|\Tdx^{m}\da_x U_0\|_{Z^2},  \label{1D_6}
\\
\|\dxa B(t)\|_{L^2_x}
\leq
 &C\delta(\alpha,\alpha')(1+t)^{-\frac38-\frac{k}4}
  +C(1+t)^{-m}\|\Tdx^{m}\da_x U_0\|_{Z^2},   \label{1D_0}
 \ema
where and below $\delta(\alpha,\alpha')=:\|\da_x U_0\|_{Z^2}+\|\dx^{\alpha'}U_0\|_{Z^1}$, $k=|\alpha-\alpha'|$ and $m\ge 0$. In particular, it holds for $f= \P_0f+ \P_1f$ that
 \bma
\|\da_x (f(t),\chi_0)\|_{L^2_{x}}
\leq
&C\delta(\alpha,\alpha')(1+t)^{-\frac54-\frac{k}2},\label{1D_2}
\\
\|\da_x (f(t),v\sqrt M)\|_{L^2_{x}}
\leq
 &C\delta(\alpha,\alpha')(1+t)^{-\frac58-\frac{k}4}
  +C(1+t)^{-m-\frac12}\|\Tdx^{m}\da_x U_0\|_{Z^2},   \label{1D_3}
\\
\|\da_x (f(t),\chi_4)\|_{L^2_{x}}
\leq
 &C\delta(\alpha,\alpha')(1+t)^{-\frac34-\frac{k}2}, \label{1D_4}
  \\
 \|\da_x \P_1f(t)\|_{L^2_{x,v}}
\leq
&C\delta(\alpha,\alpha')(1+t)^{-\frac78-\frac{k}4}
  +C(1+t)^{-m-\frac12}\|\Tdx^{m}\da_x U_0\|_{Z^2}.  \label{1D_5}
\ema

Furthermore, if we also assume that  $l\ge 2$ and there exist a constant $d_0>0$ such that the Fourier transform $\hat U_0=(\hat f_0,\hat E_0,\hat B_0)$ of the initial data $U_0$ satisfies that
  $\inf_{|\xi|\le r_0}|\hat{E}_0(\xi)\cdot\frac{\xi}{|\xi|}|\geq d_0$, $\inf_{|\xi|\le r_0}|\frac{\xi}{|\xi|}\times\hat{E}_0(\xi)|\geq d_0$, $\inf_{|\xi|\le r_0}|\frac{\xi}{|\xi|}\times\hat{B}_0(\xi)|\geq d_0$, $\inf_{|\xi|\le r_0}|(\hat{f}_0(\xi),\chi_4)|\geq d_0$  and $\sup_{|\xi|\le r_0}|(\hat f_0(\xi),v\sqrt M)|=0$,  then the global solution $(f(t),E(t),B(t))$ satisfies the following optimal time decay rates
  \bma
   C_1(1+t)^{-\frac58}
  \leq&\| f(t)\|_{L^2_{x,v}}\leq C_2(1+t)^{-\frac58}, \label{1H_2c}
\\
 C_1(1+t)^{-\frac34}
  \leq&\|E(t)\|_{L^2_x}\leq C_2(1+t)^{-\frac34},\label{1H_2a}
\\
 C_1(1+t)^{-\frac38}
  \leq&\|B(t)\|_{L^2_x}\leq C_2(1+t)^{-\frac38}, \label{1H_3a}
 \ema
for $t>0$ being large enough and $C_2\ge C_1>0$ being two constants, and in particular
\bma
  C_1(1+t)^{-\frac54}
  \leq&\|(f(t),\chi_0)\|_{L^2_x}\leq C_2(1+t)^{-\frac54},\label{1H_1}
\\
 C_1(1+t)^{-\frac58}
  \leq&\|( f(t),\chi_j)\|_{L^2_x}\leq C_2(1+t)^{-\frac58},\ j=1,2,3,\label{1H_2}
\\
 C_1(1+t)^{-\frac34}
  \leq&\|( f(t),\chi_4)\|_{L^2_x}\leq C_2(1+t)^{-\frac34}, \label{1H_3}
\\
  C_1(1+t)^{-\frac78}
  \leq&\| \P_1 f(t)\|_{L^2_{x,v}}\leq C_2(1+t)^{-\frac78}. \label{1Q_2a}
\ema
\end{thm}

Finally, we have the time convergence rates of global solutions to the nonlinear system~\eqref{1VMB3}--\eqref{1VMB3d} as follows.

\begin{thm}\label{time3a} %
Assume that  the initial data $f_{0}\in H^{N+3}_w\cap L^{2,1}$ and $(E_0,B_0)\in H^{N+3}(\R^3_x)\cap L^{1}(\R^3_x)$ for $N\ge 4$ satisfy  $\|f_{0}\|_{ H^{N+3}_w\cap L^{2,1}}+\|(E_0,B_0)\|_{H^{N+3}_x\cap L^1_x}\le \delta_0$ with $\delta_0>0$ being small enough. Then, the unique solution  $(f,E,B)$ to the VMB system~\eqref{1VMB3}--\eqref{1VMB3d} exists globally in time and belongs to $H^{N+3}_w\times H^{N+3}_x\times H^{N+3}_x $.

Moreover, if it also holds  $\Tdx\cdot E_0 = (f_0,\chi_0)$, then the global solution  satisfies for $k=0,1$ that
 \bma
 \|\dx^kf(t)\|_{L^2_{x,v}}
 & \le C \delta_0  (1+t)^{-\frac58-\frac k4},   \label{t4.2}
  \\
  \|\dx^kE(t)\|_{L^2_x}
   & \le C \delta_0 (1+t)^{-\frac34-\frac k4}\ln(1+t),   \label{t4.3}
 \\
\|\dx^k B(t)\|_{L^2_{x}}&\le C\delta_0(1+t)^{-\frac38-\frac k4}, \label{t4.4}
 \ema
and in particular 
 \bma
\|\dx^k(f(t),\chi_0)\|_{L^2_x}
 &\le   C \delta_0(1+t)^{-1-\frac k4},   \label{t4.1z}
\\
 \|\dx^k(f(t),\chi_j)\|_{L^2_x}
 & \le C \delta_0(1+t)^{-\frac58-\frac k4},\quad j=1,2,3,  \label{t4.2z}
\\
 \|\dx^k(f(t),\chi_4)\|_{L^2_x}
   & \le C \delta_0(1+t)^{-\frac34-\frac k2},   \label{t4.3z}
 \\
\|\dx^k \P_1f(t)\|_{L^2_{x,v}}&\le C\delta_0(1+t)^{-\frac78-\frac k4},
\\
 \| \P_1f(t)\|_{H^N_w} +\|\Tdx  \P_0f(t)\|_{L^2_v(H^{N-1}_x)}& +\|\Tdx (E,B)(t)\|_{H^{N-1}_x}
 \le
 C \delta_0(1+t)^{-\frac58}. \label{t4.4a}
 \ema

Furthermore, if there also exist a constant $d_0>0$ and a small constant $r_0>0$ so that the Fourier transform
$\hat U_0=(\hat{f}_{0},\hat E_0,\hat B_0)$ of the initial data $U_0=(f_{0},E_0,B_0)$ satisfies that
 $\inf_{|\xi|\le r_0}|\hat{E}_0(\xi)\cdot\frac{\xi}{|\xi|}|\geq d_0$,
 $\inf_{|\xi|\le r_0}|\frac{\xi}{|\xi|}\times\hat{E}_0(\xi)|\geq d_0$,
 $\inf_{|\xi|\le r_0}|\frac{\xi}{|\xi|}\times\hat{B}_0(\xi)|\geq d_0$,
 $\inf_{|\xi|\le r_0}|(\hat{f}_0(\xi),\chi_4)|\geq d_0$ and
 $\sup_{|\xi|\le r_0}|(\hat f_0(\xi),v\sqrt M)|=0$.
Then, the global solution $(f,E,B)$ satisfies 
 \bma
 C_1\delta_0(1+t)^{-\frac58}&\le  \|f(t)\|_{L^2_{x,v}}  \le C_2\delta_0(1+t)^{-\frac58},   \label{B_1e}
\\
 C_1\delta_0(1+t)^{-\frac38}&\le \|B(t)\|_{L^2_{x}}\le C_2\delta_0(1+t)^{-\frac38},   \label{B_5}
\ema
 for $t>0$ large enough with two constants $C_2>C_1$, and in particular
 \bma
 C_1\delta_0(1+t)^{-\frac58}&\le \|(f(t),\chi_j)\|_{L^2_{x}}\le C_2\delta_0(1+t)^{-\frac58},\quad j=1,2,3, \label{B_1}
 \\
 C_1\delta_0(1+t)^{-\frac34}&\le \|(f(t),\chi_4)\|_{L^2_{x}}\le C_2\delta_0(1+t)^{-\frac34}, \label{B_2}
 \\
 C_1\delta_0(1+t)^{-\frac78}&\le \| \P_1f(t)\|_{L^2_{x,v}}\le C_2\delta_0(1+t)^{-\frac78}. \label{B_3}
\ema
\end{thm}

\begin{rem}
 Let us give an example of the initial function $(f_0,E_0,B_0)$ which satisfies the assumptions of Theorem~\ref{time3a}. For a positive constant $d_0$, we define $ (f_0,E_0,B_0)$ as
\bgrs
f_0(x,v)
 = \frac1{(2\pi)^{3/2}}d_0e^{r_0^2/2}\intr |\xi|e^{-|\xi|^2/2}e^{\i x\cdot\xi}d\xi\chi_0(v)+d_0e^{r_0^2/2}e^{-|x|^2/2}\chi_4(v),\\
 E_0(x)=\frac1{(2\pi)^{3/2}}d_0e^{r_0^2/2}\intr \Big(\frac{\xi}{|\xi|}+\frac{(-\xi_2,\xi_1,0)}{(\xi_1^2+\xi_2^2)^{1/2}}\Big)e^{-|\xi|^2/2}e^{\i x\cdot\xi}d\xi,\\
 B_0(x)=\frac1{(2\pi)^{3/2}}d_0e^{r_0^2/2}\intr \frac{(-\xi_2,\xi_1,0)}{(\xi_1^2+\xi_2^2)^{1/2}}e^{-|\xi|^2/2}e^{\i x\cdot\xi}d\xi.
 \egrs
\end{rem}

\begin{rem}\label{rem3}
 In the case of $\Tdx\cdot E_0\ne (f_0,\sqrt M)$, the solutions of the nonlinear VMB system have no decay rates. The main reason is that
the magnetic field decay too slow. For instance, we assume that $f,E,B$ decay at the same rates $(1+t)^{-\frac14}$, $(1+t)^{-\frac14}$ and $(1+t)^{-\frac38}$ respectively as the linear solution, then we can obtain that the key quadratic nonlinear term decays at most at the rate $(1+t)^{-\frac12}$. Thus, making use of the Duhamel's principle to represent the solution via the semigroup, the corresponding nonlinear term to magnetic field $B$ generated from the nonhomogeneous source
decays as
$$
\intt (1+t-s)^{-\frac58}(1+s)^{-\frac12}ds\le C(1+t)^{-\frac18}.
$$
Thus, the bootstrap argument breaks down and one can not expect the magnetic field to decay as
$(1+t)^{-\frac38}$ in the nonlinear case.
\end{rem}

\section{Spectral analysis for two species case}
\label{spectrum-two}
\setcounter{equation}{0}

For the study on the spectrum structure of the two-species VMB,
in the following
 subsection, we will first investigate some properties of the operator
$\hat{\AA}_1(\xi)$ that lead to the decription of its spectra and resolvent. And then
the asymptotics of its eigenvalues and eigenfunctions in low and high frequency
regions will be given in the
next two subsections.

\subsection{Spectrum structure}

First of all, note that
 $P_{\rm d}$ is a self-adjoint operator satisfying
$(P_{\rm d} f,P_{\rm d} g)=(P_{\rm d} f, g)=( f,P_{\rm d} g)$. Hence,
 \bq
  (f,g)_\xi=(f,g+\frac1{|\xi|^2}P_{\rm d}g)=(f+\frac1{|\xi|^2}P_{\rm d}f,g).\label{C_1}
 \eq
By
\eqref{C_1}, we have for any $f,g\in L^2_\xi(\R^3_v)\cap D(\hat{B}_1(\xi))$,
\bma (\hat{B}_1(\xi)
f,g)_\xi=(\hat{B}_1(\xi) f,g+\frac1{|\xi|^2}P_{\rm d} g)
=(f,(L_1+\i(v\cdot\xi)+\frac{\i(v\cdot\xi)}{|\xi|^2}P_{\rm d})g)_\xi=(f,\hat{B}_1(-\xi)g)_\xi.\label{L_7}\ema

Also, note that $\hat{B}_1(\xi)$ is a linear operator from the space $L^2_\xi(\R^3)$ to itself, and  for any $ y\in \mathbb{C}^3_\xi$,
\bq \frac{\xi}{|\xi|}\times\frac{\xi}{|\xi|}\times y=-y.\label{rotat}\eq

Since $L^2_\xi(\R^3_v)\times \mathbb{C}^3_\xi\times \mathbb{C}^3_\xi$ is an invariant subspace of the operator $\hat{\AA}_1(\xi)$,  $\hat{\AA}_1(\xi)$ can
be regarded as a linear operator on $L^2_\xi(\R^3_v)\times \mathbb{C}^3_\xi\times \mathbb{C}^3_\xi$.
Firstly, we have

\begin{lem}\label{SG_1}
The operator $\hat{\AA}_1(\xi)$ generates a strongly continuous contraction semigroup on
$L^2_\xi(\R^3_v)\times \mathbb{C}^3_\xi\times \mathbb{C}^3_\xi$ satisfying
 \bq
\|e^{t\hat{\AA}_1(\xi)}U\|_\xi\le\|U\|_\xi, \quad\mbox{for}\ t>0,\,\,U\in
L^2_\xi(\R^3_v)\times \mathbb{C}^3_\xi\times \mathbb{C}^3_\xi.
 \eq
\end{lem}
\begin{proof}
We first show that both $\hat{\AA}_1(\xi)$ and $\hat{\AA}_1(\xi)^*$ are
dissipative operators on $L^2_\xi(\R_v^3)$. By \eqref{L_7}, we
obtain for any $U,V\in L^2_\xi(\R_v^3)\cap D(\hat{B}_1(\xi))\times \mathbb{C}^3_\xi\times \mathbb{C}^3_\xi$ that
$
 (\hat{\AA}_1(\xi)U,V)_\xi=(U,\hat{\AA}_1(\xi)^*V)_\xi
$
with
$
 \hat{\AA}_1(\xi)^*= \hat{\AA}_1(-\xi).
$

Direct computation shows the dissipation of both
$\hat{\AA}_1(\xi)$ and $\hat{\AA}_1(\xi)^*$,  namely,
$$\mathrm{Re}(\hat{\AA}_1(\xi)U,U)_\xi=\mathrm{Re}(\hat{\AA}_1(\xi)^*U,U)_\xi=(L_1f,f)\leq0.$$
Since $\hat{\AA}_1(\xi)$ is a densely defined closed operator, it
follows from Corollary 4.4 on p.15 of \cite{Pazy} that the operator $\hat{\AA}_1(\xi)$
generates a $C_0$-contraction semigroup on $L^2_\xi(\R^3_v)\times \mathbb{C}^3_\xi\times \mathbb{C}^3_\xi$.
\end{proof}

Define a $6\times 6$ matrix by
\bq B_3(\xi)=\left(\ba  0 & \i\xi\times \\ -\i\xi\times & 0 \ea\right)_{6\times6}.\label{B1}\eq
Since $\mathbb{C}^3_\xi\times \mathbb{C}^3_\xi$ is an invariant subspace of the operator $B_3(\xi)$, we can regard $B_3(\xi)$ as an operator on $\mathbb{C}^3_\xi\times \mathbb{C}^3_\xi$. Then we have

\begin{lem}For any $\lambda\ne \pm\i|\xi|$, the operator $\lambda-B_3(\xi)$ is invertible on $\mathbb{C}^3_\xi\times \mathbb{C}^3_\xi$ and satisfies
\bq  \|(\lambda-B_3(\xi))^{-1}\|= \max_{j=\pm1}|\lambda-j\i|\xi||^{-1}.\label{b_1(xi)}\eq
\end{lem}

\begin{proof}
First, we compute the eigenvalues of the operator $B_3(\xi)$. For this,  consider
$$(\lambda-B_3(\xi))X=0,\quad X=(X_1,X_2)\in \mathbb{C}^3_\xi\times \mathbb{C}^3_\xi.$$
It follows that
\bma
\lambda X_1-\i\xi\times X_2=0,\label{111}\\
\lambda X_2+\i\xi\times X_1=0.\label{222}
\ema
Multiplying \eqref{111} by $\lambda$ and using \eqref{222} and \eqref{rotat}
give
\bq \lambda^2 X_1+|\xi|^2X_1=0,\eq
which implies that $\lambda_j=j\i |\xi|$ for $j=\pm1$ are the eigenvalues of $B_3(\xi)$.
Thus $\lambda-B_3(\xi)$ is invertible on $\mathbb{C}^3_\xi\times \mathbb{C}^3_\xi$ for any $\lambda\ne \pm\i|\xi|$. Since
\bma
(\i B_3(\xi)X,Y)=(X,\i B_3(\xi)Y), \quad \forall \ X,Y\in \mathbb{C}^3_\xi\times \mathbb{C}^3_\xi,\label{b_1a}
\ema
it follows that $\i B_3(\xi)$ is a self-adjoint operator on $\mathbb{C}^3_\xi\times \mathbb{C}^3_\xi$ and hence
$$\|(\lambda- B_3(\xi))^{-1}\|= \max_{j=\pm1}|\lambda-j\i|\xi||^{-1}.$$
This completes the proof of the lemma.
\end{proof}

Denote by $\rho(\hat{\AA}_1(\xi))$  the resolvent set and
by $\sigma(\hat{\AA}_1(\xi))$ the spectrum set of $\hat{\AA}_1(\xi)$.  We have

\begin{lem}\label{Egn}
For each $\xi\ne 0$, the spectrum set $\sigma(\hat{\AA}_1(\xi))$ of the
operator $\hat{\AA}_1(\xi)$ in the domain
$\mathrm{Re}\lambda\geq-\nu_0+\delta$ for any constant $\delta>0$
consists of isolated eigenvalues $\Sigma=:\{\lambda_j(\xi)\}$ with
$\mathrm{Re}\lambda_j (\xi)<0$.
\end{lem}

\begin{proof}
Define
 \bma
G_1(\xi)&=\left( \ba
c(\xi) &0 &0\\
0 &0 &\i\xi\times\\
0 &-\i\xi\times &0
\ea\right),\quad G_2(\xi)=\left( \ba
K_1-\frac{\i(v\cdot\xi)}{|\xi|^2}P_{\rm d} &-v\sqrt M\cdot\omega\times &0\\
-\omega\times P_m &0 &0\\
0 &0 &0
\ea\right), \label{A12}\\
  c(\xi)&=-\nu(v)-\i(v\cdot\xi).  \label{Cxi}
 \ema
It is obvious that $\lambda-G_1(\xi)$ is invertible for ${\rm
Re}\lambda>-\nu_0$ and $\lambda\ne \pm\i|\xi|$. Since $G_2(\xi)$ is a compact
operator on $L^2_\xi(\R^3_v)\times \mathbb{C}^3_\xi\times \mathbb{C}^3_\xi$ for any fixed $\xi\ne 0$,
$\hat{\AA}_1(\xi)$ is a compact perturbation of $G_1(\xi)$. Hence,  by
 Theorem 5.35 on p.244 of \cite{Kato}, $\hat{\AA}_1(\xi)$ and $G_1(\xi)$ have the same essential spectrum where $\sigma_{\rm e}(G_1(\xi))={\rm Ran}(c(\xi))$ and $\sigma_{\rm d}(G_1(\xi))=\pm \i|\xi|$. Thus
 the spectrum of $\hat{\AA}_1(\xi)$ in the domain ${\rm Re}\lambda>-\nu_0$ consists of
discrete eigenvalues $\lambda_j(\xi)$ with possible accumulation
points only on the line ${\rm Re}\lambda= -\nu_0$.

We claim that for any discrete eigenvalue $\lambda(\xi)$ of
$\hat{\AA}_1(\xi)$ in the region $\mathrm{Re}\lambda\geq-\nu_0+\delta$
for any constant $\delta>0$, it holds that ${\rm Re}\lambda(\xi)<0$
for $\xi\ne 0$. Indeed, set $\xi=s\omega$ and let $U=(f,E,B)\in L^2_\xi(\R^3_v)\times \mathbb{C}^3_\xi\times \mathbb{C}^3_\xi$ be the
eigenvector corresponding to the eigenvalue $\lambda$ so that
 \bq           \label{L_6}
  \left\{\bal
 \lambda f=L_1f-\i s(v\cdot\omega) (f+\frac1{s^2}P_{\rm d} f )-v\sqrt M\cdot(\omega\times E),\\
 \lambda E=-\omega\times (f,v\sqrt M)+\i \xi\times B,\\
 \lambda B=-\i\xi\times E.
 \ea\right.
 \eq
Taking the inner product $(\cdot,\cdot)_\xi$ of \eqref{L_6} with
$U$, we have
$$
 (L_1f,f)=\text{Re}\lambda\Big(\|f\|^2 +\frac1{s^2}\|P_{\rm d} f\|^2+|E|^2+|B|^2\Big),
$$
which together with \eqref{L_4} implies $\text{Re}\lambda\leq 0$.

Furthermore, if there exists an eigenvalue $\lambda$ with ${\rm
Re}\lambda=0$, then it follows from the above that $(L_1f,f)=0$,
namely, 
$f=C_0\sqrt M\in N_1$. Substitute this into \eqref{L_6}, we obtain
 \bq
 \lambda C_0\sqrt M=-\i (v\cdot\omega)\Big(s+\frac1{s} \Big)C_0\sqrt M-v\sqrt M\cdot(\omega\times E),
 \eq
which   implies that  $C_0=0$ and $\omega\times E=0$. Therefore,  $f\equiv0$ and $ E\equiv0$. Substitute this into \eqref{L_6}, we obtain $B\equiv0$. This is a contradiction and
thus it holds $\text{Re}\lambda<0$ for any discrete eigenvalue
$\lambda\in\sigma(\hat{\AA}_1(\xi))$.
\end{proof}

Now denote by $T$ a linear operator on $L^2(\R^3_v)$ or
$L^2_\xi(\R^3_v)$, and we define the corresponding norms of $T$ by
$$
 \|T\|=\sup_{\|f\|=1}\|Tf\|,\quad
 \|T\|_\xi=\sup_{\|f\|_\xi=1}\|Tf\|_\xi.
$$
Obviously,
 \bq
(1+|\xi|^{-2})^{-1}\|T\|\le \|T\|_\xi\le (1+|\xi|^{-2})\|T\|.\label{eee}
 \eq

Also, if  $T$ is a linear operator on $L^2(\R^3_v)\times \mathbb{C}^3_\xi\times \mathbb{C}^3_\xi$ or
$L^2_\xi(\R^3_v)\times \mathbb{C}^3_\xi\times \mathbb{C}^3_\xi$, then
$$
 \|T\|=\sup_{\|U\|=1}\|TU\|,\quad
 \|T\|_\xi=\sup_{\|U\|_\xi=1}\|TU\|_\xi.
$$

We will make use of the following decomposition associated with the
operator $\hat{\AA}_1(\xi)$ for  $|\xi|>0$
\bma
\lambda-\hat{\AA}_1(\xi)=\lambda- G_1(\xi)-G_2(\xi)=(I-G_2(\xi)(\lambda-G_1(\xi))^{-1})(\lambda-G_1(\xi)),  \label{B_d}
\ema
where $G_1(\xi), G_2(\xi)$ are defined by \eqref{A12}. For ${\rm Re}\lambda>-\nu_0$ and $\lambda\ne \pm\i|\xi|$, we have
\bma
(\lambda-G_1(\xi))^{-1}&= \left(\ba (\lambda -c(\xi))^{-1} & 0 \\  0 & (\lambda -B_3(\xi))^{-1} \ea\right)_{7\times7},\\
G_2(\xi)(\lambda-G_1(\xi))^{-1}&= \left(\ba X_1(\lambda,\xi) & X_2(\lambda,\xi) \\  X_3(\lambda,\xi) & 0 \ea\right)_{7\times7},
\ema
where $B_3(\xi)$ is defined in \eqref{B1}, and
\bma
&X_1(\lambda,\xi)=(K_1-\frac{\i (v\cdot\xi)}{|\xi|^2} P_{\rm d})(\lambda -c(\xi))^{-1},
\\
& X_2(\lambda,\xi)=(v\sqrt M\cdot\omega\times,0_{1\times3})_{1\times6}(\lambda -B_3(\xi))^{-1},\label{X_2}
\\
 &X_3(\lambda,\xi)=\left(\ba -\omega\times P_m(\lambda-c(\xi))^{-1} \\ 0_{3\times1} \ea\right)_{6\times1}.\label{X_3}
\ema

Let $K_1,K_4$ be the operators on the space $X$ and $Y$, and $K_2,K_3$ be the operators  $Y\to X$ and $X\to Y$ respectively. Let $K$ be a matrix operator on $X\times Y$ defined by
$$
K=\left(\ba K_1 & K_2 \\  K_3 & K_4 \ea\right).
$$
Then, we have
\begin{lem}\label{inver}
If the norms of $K_1,K_2,K_3$ and $K_4$ satisfy 
$$\|K_1\|<1,\quad \|K_4\|<1,\quad \|K_2\|\|K_3\|<(1-\|K_1\|)(1-\|K_4\|),$$
then the operator $I+K$ is invertible on $X\times Y$.
\end{lem}
\begin{proof}
Decompose $I+K$ into
$$I+K=\left(\ba I+K_1 & K_2 \\  0 & I+K_4 \ea\right)+\left(\ba 0 & 0 \\  K_3 & 0 \ea\right).$$
Since
$$\left(\ba I+K_1 & K_2 \\  0 & I+K_4 \ea\right)^{-1}=\left(\ba (I+K_1)^{-1} & -(I+K_1)^{-1}K_2(I+K_4)^{-1} \\  0 & (I+K_4)^{-1} \ea\right),$$
it follows that
\bmas
I+K
=\left(\ba I+K_1 & K_2 \\  0 & I+K_4 \ea\right)\left(\ba I-(I+K_1)^{-1}K_2(I+K_4)^{-1}K_3 & 0 \\  (I+K_4)^{-1}K_3 & I    \ea\right)
\emas
is invertible on $X\times Y$ because  $\|(I+K_1)^{-1}K_2(I+K_4)^{-1}K_3\|<1$.
\end{proof}

\begin{lem}
\label{LP03}
 There exists a constant  $C>0$ such that
\begin{enumerate}
\item For any $\delta>0$, we have
 \bgr
\sup_{x\geq-\nu_0+\delta,y\in\R}\|K_1(x+\i y-c(\xi))^{-1}\|
  \leq C\delta^{-15/13}(1+|\xi|)^{-2/13}. \label{T_7}
 \egr

\item For any $\delta>0,\, r_0>0$, there is a constant  $y_0=(2r_0)^{5/3}\delta^{-2/3}>0$ such that
if $|y|\geq y_0$, we have
 \bgr
 \sup_{x\geq -\nu_0+\delta,|\xi|\leq r_0}\|K_1(x+\i y-c(\xi))^{-1}\|
 \leq C\delta^{-7/5}(1+|y|)^{-2/5}.\label{T_8}
 \egr

\item For any $\delta>0,\, r_0>0$, we have
 \bgr
 \sup_{x\geq-\nu_0+\delta,y\in\R}\|P_m(x+\i y-c(\xi))^{-1}\|_{L^2(\R^3)\to \mathbb{C}^3}
  \leq C\delta^{-1/2}(1+|\xi|)^{-1/2}, \label{T_7a}\\
 \sup_{x\geq -\nu_0+\delta,|\xi|\leq r_0}\|P_m(x+\i y-c(\xi))^{-1}\|_{L^2(\R^3)\to \mathbb{C}^3}
 \leq C(\delta^{-1}+1)(r_0+1)|y|^{-1}.\label{T_8a}
 \egr
 \item  For any $\delta>0,\, r_0>0$, we have \bgr
   \sup_{x\geq -\nu_0+\delta,y\in\R}
  \|(v\cdot\xi)|\xi|^{-2}P_{\rm d}(x+\i y-c(\xi))^{-1}\|
 \leq C\delta^{-1}|\xi|^{-1},\label{L_9}
 \\
   \sup_{x\geq -\nu_0+\delta,|\xi|\geq r_0}
 \|(v\cdot\xi)|\xi|^{-2}P_{\rm d}(x+\i y-c(\xi))^{-1}\|
 \leq C(\delta^{-1}+1)(r_0^{-1}+1)|y|^{-1}.\label{L_10}
 \egr
\end{enumerate}
\end{lem}

\begin{proof}The proof of \eqref{T_7}, \eqref{T_8}, \eqref{L_9} and \eqref{L_10} can be found in Lemma 2.3 of \cite{Li2}. We only need to prove \eqref{T_7a} and \eqref{T_8a}. Since
\bmas
&\|(x-\i y+\nu(v)-\i (v\cdot\xi))^{-1}v\sqrt M\|^2\\
\le& C\intr \frac1{(x+\nu_0)^2+(y+ (v\cdot\xi))^2} e^{-\frac{|v|^2}4}dv=C\intr \frac1{(x+\nu_0)^2+(y+ v_1|\xi|)^2} e^{-\frac{|v|^2}4}dv\\
=& C\frac1{|\xi|}\intr \frac1{(x+\nu_0)^2+v_1^2} e^{-\frac{(v_1-y)^2}{4|\xi|^2}}e^{-\frac{v_2^2}4}e^{-\frac{v_3^2}4}dv\le C(x+\nu_0)^{-1}|\xi|^{-1},
\emas
we obtain
\bmas
|P_m(x+\i y-c(\xi))^{-1}f|
&\le \|(x-\i y+\nu(v)-\i (v\cdot\xi))^{-1}v\sqrt M\|\|f\|\\
&\le C(x+\nu_0)^{-1/2}|\xi|^{-1/2}\|f\|.
\emas
This proves \eqref{T_7a}.  Since
$
P_m(\lambda-c(\xi))^{-1}=\frac1{\lambda}P_m+\frac1{\lambda}P_m(\lambda-c(\xi))^{-1}c(\xi),
$
it follows that
$
\|P_m(\lambda-c(\xi))^{-1}\|\le |\lambda|^{-1}+C\delta^{-1}|\lambda|^{-1}(1+|\xi|),
$
which proves \eqref{T_8a}.
\end{proof}

With  Lemma~\ref{LP03}, we can investigate the spectrum set
of the operator $\hat{\AA}_1(\xi)$ in the intermediate and high frequency
regions.
\begin{lem}
\label{LP01}
For the high and intermediate frequencies, the following statements
hold.
 \begin{enumerate}
\item  For any $\delta_1,\delta_2>0$, there
exists $ R_1= R_1(\delta_1,\delta_2)>0$  such that for $|\xi|>R_1$,
 \bq
 \sigma(\hat{\AA}_1(\xi))\cap\{\lambda\in\mathbb{C}\,|\,\mathrm{Re}\lambda\ge-\nu_0+\delta_1\}
 \subset
 \dcup_{j=\pm1}\{\lambda\in\mathbb{C}\,|\,|\lambda-j\i|\xi||\le\delta_2\}.\label{sg4}
 \eq
\item For any $r_1>r_0>0$, there
exists $\alpha =\alpha(r_0,r_1)>0$ such that for  $r_0\le |\xi|\le r_1$,
\bq \sigma(\hat{\AA}_1(\xi))\subset\{\lambda\in\mathbb{C}\,|\, \mathrm{Re}\lambda(\xi)\leq-\alpha\} .\label{sg3}\eq
 \end{enumerate}
\end{lem}
\begin{proof}
We prove \eqref{sg4} first.
By Lemma \ref{LP03}, \eqref{b_1(xi)} and \eqref{eee}, there is $R_1=R_1(\delta_1,\delta_2)>0$ such that for $\mathrm{Re}\lambda\ge-\nu_0+\delta_1$, $\min_{j=\pm1}|\lambda-j\i|\xi||>\delta_2$ and $|\xi|>R_1$,
$$
\|(K_1-\mbox{$\frac{\i(v\cdot\xi)}{|\xi|^2}$}P_{\rm d})(\lambda-c(\xi))^{-1}\|_\xi\leq 1/2,\quad  \| P_m(\lambda-c(\xi))^{-1}\|_{L^2_\xi(\R^3)\to \mathbb{C}^3}\leq \delta_2/4,\quad \|(\lambda-B_3(\xi))^{-1}\|\le \delta_2^{-1},
$$
which leads to
$$
\|X_1(\lambda,\xi)\|_{\xi}\leq 1/2, \quad \|X_2(\lambda,\xi)\|_{\mathbb{C}^6\to L^2_\xi(\R^3)}\|X_3(\lambda,\xi)\|_{L^2_\xi(\R^3)\to \mathbb{C}^6}\leq 1/4.
$$
This and Lemma \ref{inver} imply that the operator  $I-G_2(\xi)(\lambda-G_1(\xi))^{-1}$ is invertible on
$L^{2}_\xi(\R^3_v)\times \mathbb{C}^3_\xi\times \mathbb{C}^3_\xi$ and thus  $\lambda-\hat{\AA}_1(\xi)$ is invertible on $L^{2}_\xi(\R^3_v)\times \mathbb{C}^3_\xi\times \mathbb{C}^3_\xi$.
Therefore, we have
 $\rho(\hat{\AA}_1(\xi))\supset\{\lambda\in\mathbb{C}\,|\,\min_{j=\pm1}{|\lambda-j\i|\xi||}>\delta_2,\mathrm{Re}\lambda\ge-\mu+\delta_1\}$
for $ |\xi|>R_1$ which gives \eqref{sg4}.

Next, we turn to prove \eqref{sg3}. For this, we first show that  $\sup_{r_0\le |\xi|\le r_1}|{\rm
Im}\lambda(\xi)|<+\infty$ for any $\lambda(\xi)\in
\sigma(\hat{\AA}_1(\xi))$ with  $\mathrm{Re}\lambda\geq -\nu_0+\delta_1$.
Indeed, by Lemma \ref{LP03}, \eqref{b_1(xi)}  and \eqref{eee}, there exists
$y_1=y_1(r_0,r_1,\delta_1)>0$ large enough such that for $\mathrm{Re}\lambda\geq -\nu_0+\delta_1$, $|\mathrm{Im}\lambda|>y_1 $ and $r_0\le |\xi|\le r_1$,
$$
\|(K_1-\mbox{$\frac{\i(v\cdot\xi)}{|\xi|^2}$}P_{\rm d})(\lambda-c(\xi))^{-1}\|_\xi\leq 1/6,\quad  \| P_m(\lambda-c(\xi))^{-1}\|_{L^2_\xi(\R^3)\to \mathbb{C}^3}\leq 1/6,\quad \|(\lambda-B_3(\xi))^{-1}\|\le 1/6,
$$
which leads to
$$
\|X_1(\lambda,\xi)\|_{\xi}+\|X_2(\lambda,\xi)\|_{\mathbb{C}^6\to L^2_\xi(\R^3)}+\|X_3(\lambda,\xi)\|_{L^2_\xi(\R^3)\to \mathbb{C}^6}\leq 1/2.
$$
This implies that the operator
$I-G_2(\xi)(\lambda-G_1(\xi))^{-1}$
is invertible on $L^2_\xi(\R^3_v)\times \mathbb{C}^3_\xi\times \mathbb{C}^3_\xi$, which together with
\eqref{B_d} yield that  $\lambda-\hat{\AA}_1(\xi)$ is also invertible on $L^2_\xi(\R^3_v)\times \mathbb{C}^3_\xi\times \mathbb{C}^3_\xi$ when
$\mathrm{Re}\lambda\geq -\nu_0+\delta_1$,  $|\mathrm{Im}\lambda|>y_1 $ and $r_0\le |\xi|\le r_1$. Note that it
satisfies
 \bq
 (\lambda-\hat{\AA}_1(\xi))^{-1}
 =(\lambda-G_1(\xi))^{-1}( I-G_2(\xi)(\lambda-G_1(\xi))^{-1})^{-1}.\label{E_6}
  \eq
 Thus, we conclude that for $r_0\le |\xi|\le r_1$,
 \bq
 \sigma(\hat{\AA}_1(\xi))
 \cap\{\lambda\in\mathbb{C}\,|\,\mathrm{Re}\lambda\ge-\nu_0+\delta_1\}
\subset
 \{\lambda\in\mathbb{C}\,|\,\mathrm{Re}\lambda\ge
    -\nu_0+\delta_1,\,|\mathrm{Im}\lambda|\le y_1 \}.   \label{SpH}
 \eq

Finally, we prove  $\sup_{ r_0\le |\xi|\le r_1}{\rm Re}\lambda(\xi)<0$ by
contradiction.
 If it does not
hold, then there exists a sequence of
$\{(\xi_n,\lambda_n,f_n,E_n,B_n)\}$ satisfying $|\xi_n|\in[r_0,r_1]$, $(f_n,E_n,B_n)\in
L^2(\R^3)\times \mathbb{C}^3_{\xi_n}\times \mathbb{C}^3_{\xi_n}$ with $\|f_n\|+|E_n|+|B_n|=1$, and $\lambda_n\in
\sigma(\hat{\AA}_1(\xi_n))$ with ${\rm Re}\lambda_n\to0$ as $n\to\infty$. That is,
 $$
 \left\{\bal
 \lambda_nf_n=(L_1-\i(v\cdot\xi_n)-\frac{\i(v\cdot\xi_n)}{|\xi_n|^2}P_{\rm d} )f_n-v\sqrt M\cdot(\omega_n\times E_n),\\
 \lambda_n E_n=-\omega_n\times (f_n,v\sqrt M)+\i\xi_n\times B_n,\quad
 \lambda_n B_n=-\i \xi_n\times E_n.
\ea\right.
$$
Rewrite the first  equation as
 $$
 (\lambda_n+\nu+\i(v\cdot\xi_n))f_n=K_1f_n-\frac{\i(v\cdot\xi_n)}{|\xi_n|^2}P_{\rm d} f_n-v\sqrt M\cdot(\omega_n\times E_n).
 $$
Since $K_1$ is a compact operator  on $L^2(\R^3)$, there exists a
subsequence $\{f_{n_j}\}$ of $\{f_n\}$ and $g_1\in L^2(\R^3)$ such that
$$
K_1f_{n_j}\rightarrow g_1,\quad \mbox{as}\quad j\to\infty.
$$
By using  the fact that $|\xi_n|\in[r_0,r_1]$, $P_{\rm d} f_n=C_0^n\sqrt{M}$ with
$|C_0^n|\leq1$ and $|E_n|+|B_n|\le 1$, there exists a subsequence of (still denoted by) $\{(\xi_{n_j},f_{n_j},E_{n_j},B_{n_j})\}$, and $(\xi_0,C_0,E_0, B_0)$ with $|\xi_0|\in[r_0,r_1]$ and $|C_0|\leq1$
such that $(\xi_{n_j},C^{n_j}_0,E_{n_j},B_{n_j})\to (\xi_0,C_0,E_0,B_0)$ as $j\to\infty$. In particular
$$
 \frac{\i(v\cdot\xi_{n_j})}{|\xi_{n_j}|^{2}}P_{\rm d} f_{n_j}
 \rightarrow g_2=:\frac{\i(v\cdot\xi_0)}{|\xi_0|^{2}}C_0\sqrt{M}, \quad \omega_{n_j}\times E_{n_j} \to \omega_0\times E_0=:Y_0,\quad\mbox{as}\quad j\to\infty.
$$
Since $|\text{Im}\lambda_n|\leq y_1$ and ${\rm Re}\lambda_n\to 0$,
we can extract a subsequence of (still denoted by) $\{\lambda_{n_j}\}$
such that $\lambda_{n_j}\rightarrow \lambda_0$ with ${\rm
Re}\lambda_0=0$. Then
$$
 \lim_{j\rightarrow\infty}f_{n_j}
 =\lim_{j\rightarrow\infty}\frac{g_1-g_2-(v\cdot Y_0)\sqrt M}{\lambda_{n_j}+\nu+\i(v\cdot\xi_{n_j})}
 =\frac{g_1-g_2-(v\cdot Y_0)\sqrt M}{\lambda_0+\nu+\i(v\cdot\xi_0)}:=f_0 \quad {\rm in}\quad L^2(\R^3).
$$
It follows that $\hat{\AA}_1(\xi_0) U_0=\lambda_0 U_0$ with $U_0=(f_0,E_0,B_0)\in
L^2(\R^3)\times \mathbb{C}^3_{\xi_0}\times \mathbb{C}^3_{\xi_0}$ and $\lambda_0$ is an eigenvalue of $\hat{\AA}_1(\xi_0)$ with ${\rm Re}\lambda_0=0$. This
 contradicts to the fact that  ${\rm Re} \lambda(\xi)<0$
for $\xi\ne 0$ obtained by Lemma~\ref{Egn}. Thus, the proof the lemma
is completed.
\end{proof}

We now investigate the spectrum and resolvent sets of $\hat{\AA}_1(\xi)$ in
low frequency. For this, we decompose $\lambda-\hat{\AA}_1(\xi)$ into
\bq \lambda-\hat{\AA}_1(\xi)=\lambda-G_3(\xi)-G_4(\xi)=(I-G_4(\xi)(\lambda-G_3(\xi))^{-1})(\lambda-G_3(\xi)),\label{Bd3}\eq
where \bma
&G_3(\xi)=\left(\ba Q(\xi) & 0 & 0\\ 0 & 0 & \i \xi \times \\ 0 & -\i\xi\times & 0 \ea\right),\quad
G_4(\xi)=\left(\ba Q_1(\xi) & -v\sqrt M\cdot \omega\times & 0\\ -\omega\times P_m & 0 & 0 \\ 0 & 0 & 0 \ea\right),\\
&Q(\xi)=L_1-\i P_r(v\cdot\xi)P_r,\quad Q_1(\xi)=\i P_{\rm d}(v\cdot\xi)P_r+\i P_r(v\cdot\xi)(1+\frac1{|\xi|^2})P_{\rm d}.\label{Qxi}
\ema

\begin{lem}
\label{LP}
Let $\xi\neq0$ and $Q(\xi)$ defined by \eqref{Qxi}. We have
 \begin{enumerate}
\item If $\lambda\ne0$, then
\bq
\|\lambda^{-1}P_r(v\cdot\xi)(1+\frac1{|\xi|^2})P_{\rm d}\|_\xi\le C(|\xi|+1)|\lambda|^{-1}.\label{S_2}
\eq

\item If $\mathrm{Re} \lambda>-\mu $, then the operator $\lambda P_r-Q(\xi)$ is invertible on $N_1^\bot$ and satisfies
\bma
\|(\lambda P_r-Q(\xi))^{-1}\|&\leq(\mathrm{Re}\lambda+\mu )^{-1},\label{S_3}\\
\|P_{\rm d}(v\cdot\xi)P_r(\lambda P_r-Q(\xi))^{-1}P_r\|_\xi
&\leq C(1+|\lambda|)^{-1}[(\mathrm{Re}\lambda+\mu)^{-1}+1](1+|\xi|)^2,\label{S_5}\\
\|P_m(\lambda P_r-Q(\xi))^{-1}P_r\|_{L^2_\xi(\R^3)\to \mathbb{C}^3}
&\leq C(1+|\lambda|)^{-1}[(\mathrm{Re}\lambda+\mu)^{-1}+1](1+|\xi|).\label{S_5a}
\ema
\end{enumerate}
\end{lem}
\begin{proof}
The proof of \eqref{S_2}--\eqref{S_5} can be found in \cite{Li3}. Repeating a same argument as to the estimate \eqref{T_8a}, we can obtain \eqref{S_5a}. Hence,
we omit the details for brevity.
\end{proof}

By  Lemmas~\ref{Egn}--\ref{LP}, we analyze  the spectral
and resolvent sets of the operator $\hat{\AA}_1(\xi)$.

\begin{lem}
\label{spectrum3}
For any  $\delta_1,\delta_2>0$, there are two constants
 $r_1=r_1(\delta_1,\delta_2),y_1=y_1(\delta_1,\delta_2)>0$ such that for all $|\xi|\ne 0$  the resolvent set of $\hat{\AA}_1(\xi)$
satisfies
\bq \label{rb1}
 \rho(\hat{\AA}_1(\xi))\supset
 \left\{\bln &\{\lambda\in\mathbb{C}\,|\,
     \mathrm{Re}\lambda\ge-\mu +\delta_1,\,|\mathrm{Im}\lambda|\geq y_1\}
 \cup\{\lambda\in\mathbb{C}\,|\,\mathrm{Re}\lambda>0\},\quad |\xi|\le r_1,\\
 &\{\lambda\in\mathbb{C}\,|\,
     \mathrm{Re}\lambda\ge-\mu +\delta_1,\,\min_{j=\pm1}|\lambda-j\i|\xi||\ge \delta_2\}
 \cup\{\lambda\in\mathbb{C}\,|\,\mathrm{Re}\lambda>0\},\quad |\xi|\ge r_1.
 \eln\right.
\eq
\end{lem}
\begin{proof}
By Lemma \ref{LP01}, there exists $r_1=r_1(\delta_1,\delta_2)>0$ so that the second part of \eqref{rb1} holds. Thus we only need to prove the first part
of \eqref{rb1}.
By Lemma \ref{LP}, we have for $\rm{Re}\lambda>-\mu $ and
$\lambda\neq0$  that the operator
$\lambda-Q(\xi)=\lambda P_{\rm d}+\lambda P_r-Q(\xi)$ is invertible on
$L^2_\xi(\R^3_v)$ and it satisfies
 \bmas
  (\lambda P_{\rm d}+\lambda P_r-Q(\xi))^{-1} =\lambda^{-1}P_{\rm d}+(\lambda P_r-Q(\xi))^{-1}P_r,
 \emas
because the operator $\lambda P_{\rm d}$ is orthogonal to $\lambda
P_r-Q(\xi)$. Thus, for ${\rm Re}\lambda >-\mu$ and $\lambda\ne0,\pm\i|\xi|$, the operator $\lambda-G_3(\xi)$ is invertible on
$L^2_\xi(\R^3_v)\times \mathbb{C}^3_\xi\times \mathbb{C}^3_\xi$ and satisfies
\bma
&(\lambda-G_3(\xi))^{-1}= \left(\ba \lambda^{-1} P_{\rm d}+(\lambda P_r-Q(\xi))^{-1}P_r & 0 \\  0 & (\lambda -B_3(\xi))^{-1} \ea\right)_{7\times7}.
\ema
Therefore, we can rewrite \eqref{Bd3} as
\bma
\lambda-\hat{\AA}_1(\xi)=(I-G_4(\xi)(\lambda-G_3(\xi))^{-1})(\lambda-G_3(\xi)),
\ema
where
\bma
&G_4(\xi)(\lambda-G_4(\xi))^{-1}= \left(\ba X_4(\lambda,\xi) & X_2(\lambda,\xi) \\  X_5(\lambda,\xi) & 0 \ea\right)_{7\times7},
\\
&X_4(\lambda,\xi)=\i P_{\rm d}(v\cdot\xi)P_r(\lambda P_r-Q(\xi))^{-1}P_r+\i \lambda^{-1}P_r(v\cdot\xi)(1+\frac1{|\xi|^2})P_{\rm d},
\\
 &X_5(\lambda,\xi)=\left(\ba -\omega\times P_m(\lambda P_r-Q(\xi))^{-1}P_r \\ 0_{3\times1} \ea\right)_{6\times1}.
\ema
For  $|\xi|\leq r_1$, by \eqref{b_1(xi)} and \eqref{S_2}--\eqref{S_5} we can choose $y_1=y_1(\delta_1,r_1)>0$ such that it holds
for $\mathrm{Re}\lambda\ge-\mu +\delta_1$ and
$|\mathrm{Im}\lambda|\geq y_1$ that
\be
\|X_4(\lambda,\xi)\|_\xi+\|X_2(\lambda,\xi)\|_{\mathbb{C}^6\to L^2_\xi(\R^3)}+\|X_5(\lambda,\xi)\|_{L^2_\xi(\R^3)\to \mathbb{C}^6}\leq \frac12. \label{bound_1}
\ee
This implies that the operator $I-G_4(\xi)(\lambda-G_3(\xi))^{-1}$ is invertible on
$L^{2}_\xi(\R^3_v)\times \mathbb{C}^3_\xi\times \mathbb{C}^3_\xi$ and thus  $\lambda-\hat{\AA}_1(\xi)$ is invertible on $L^{2}_\xi(\R^3_v)\times \mathbb{C}^3_\xi\times \mathbb{C}^3_\xi$ and satisfies
 \bma
 (\lambda-\hat{\AA}_1(\xi))^{-1}
 =&(\lambda-G_3(\xi))^{-1}(I-G_4(\xi)(\lambda-G_3(\xi))^{-1})^{-1}.\label{S_8}
 \ema
Therefore, $\rho(\hat{\AA}_1(\xi))\supset \{\lambda\in\mathbb{C}\,|\,{\rm
Re}\lambda\ge-\mu+\delta_1, |{\rm Im}\lambda|\ge y_1\}$ for $|\xi|\le
r_1$. This completes the proof of the lemma.
\end{proof}

\subsection{Asymptotics in low frequency}

In this subsection, we study in the low frequency region,
the asymptotics of the eigenvalues and eigenvectors
of the operator $\hat{\AA}_1(\xi)$. In terms of \eqref{B(xi)1}, the eigenvalue problem $\hat{\AA}_1(\xi)U=\lambda  U$ for $U=(f,X,Y)\in L^{2}_\xi(\R^3_v)\times \mathbb{C}^3_\xi\times \mathbb{C}^3_\xi$
can be written as
 \bma
  \lambda f
  &=(L_1-\i(v\cdot\xi)
   -\frac{\i(v\cdot\xi)}{|\xi|^2}P_{\rm d})f-v\sqrt M\cdot(\omega\times X),\label{L_2}\\
  \lambda X&=-\omega\times (f,v\sqrt M)+\i\xi\times Y,\label{L_2a}\\
  \lambda Y&=-\i\xi\times X,\quad |\xi|\ne0.\nnm
 \ema

Let $f$ be the eigenfunction of \eqref{L_2}, we rewrite $f$ in the
form $f=f_0+f_1$, where $f_0=P_{\rm d}f=C_0\sqrt M$ and $f_1=(I-P_{\rm d})f=P_rf$.
Then  \eqref{L_2} gives
 \bma
 &\lambda f_0=- P_{\rm d}[\i(v\cdot\xi)(f_0+f_1)],\label{A_2}
\\
&\lambda f_1=L_1f_1- P_r[\i(v\cdot\xi)(f_0+f_1)]-\frac{\i(v\cdot\xi)}{|\xi|^2}f_0-v\sqrt M\cdot(\omega\times X).\label{A_3}
 \ema
By Lemma \ref{LP}, \eqref{Qxi} and \eqref{A_3}, the microscopic part $f_1$ can be represented  for any $\text{Re}\lambda>-\mu$ by
 \bq
 f_1=-(\lambda  P_r-Q(\xi) )^{-1} P_r[\i(v\cdot\xi)(1+\frac1{|\xi|^2})f_0+v\sqrt M\cdot(\omega\times X)],
 \quad   \text{Re}\lambda>-\mu. \label{A_4}
 \eq
Substituting \eqref{A_4} into \eqref{A_2} and \eqref{L_2a}, we obtain the eigenvalue problem  for  $(C_0,X,Y)$ as
 \bma
 \lambda C_0=&(1+\frac1{|\xi|^2})(R(\lambda,\xi)(v\cdot\xi)\sqrt M,(v\cdot\xi)\sqrt M)C_0+(R(\lambda,\xi)(v\sqrt M\cdot(\omega\times X)),(v\cdot\xi)\sqrt M)C_0,  \label{A_6}\\
  \lambda X=&-\omega\times \i(1+\frac1{|\xi|^2})(R(\lambda,\xi)(v\cdot\xi)\sqrt M,v\sqrt M)C_0\nnm\\
  &-\omega\times (R(\lambda,\xi)(v\sqrt M\cdot(\omega\times X)),v\sqrt M)+\i\xi\times Y, \label{A_7}\\
  \lambda Y=&-\i\xi\times X. \label{A_8}
 \ema
where
$
 R(\lambda,\xi)=-(\lambda  P_r-Q(\xi) )^{-1}=[L_1-\lambda  P_r-\i P_r(v\cdot\xi) P_r]^{-1}.
$

By changing variable $(v\cdot\xi)\to |\xi|v_1$ and using the rotational invariance of the operator $L_1$, we have the following transformation.
\begin{lem}Let  $e_1=(1,0,0)$, $\xi=s\omega$ with $s\in \R,\, \omega\in \S^2$.
Then
\bma
(R(\lambda,\xi)v_i\sqrt{M},v_j\sqrt{M})= \omega_i\omega_j(R(\lambda,se_1)\chi_1,\chi_1)+(\delta_{ij}-\omega_i\omega_j)(R(\lambda,se_1)\chi_2,\chi_2).\label{T_1}
\ema
\end{lem}

With  \eqref{T_1}, the equations
\eqref{A_6}--\eqref{A_8} can be simplified as
 \bma
 \lambda C_0=&(1+s^2)(R(\lambda,se_1)\chi_1,\chi_1)C_0,\label{A_9}
 \\
 \lambda X=&(R(\lambda,se_1)\chi_2,\chi_2)X+\i\xi\times Y,   \label{A_10}
 \\
\lambda Y=&-\i\xi\times X.
   \label{A_11}
  \ema

Multiply \eqref{A_10} by $\lambda$ and using \eqref{A_11} and \eqref{rotat}, we obtain
\bq (\lambda^2-(R(\lambda,se_1)\chi_2,\chi_2)\lambda+s^2)X=0.\eq

Denote
 \bma
 D_0(\lambda,s)&=:\lambda-(1+s^2)(R(\lambda,se_1)\chi_1,\chi_1),\label{D0}\\
 D_1(\lambda,s)&=:\lambda^2-(R(\lambda,se_1)\chi_2,\chi_2)\lambda+s^2.   \label{D1}
 \ema

The following result on $D_0(\lambda,s)=0$ was proved in \cite{Li3} in the
study on the bipolar Vlasov-Poisson-Boltzmann system.

\begin{lem}[\cite{Li3}]\label{eigen_1a}
There are constants $b_0>0$ and $r_0>0$ with
 $r_0$ being small such that the equation $D_0(\lambda,s)=0$ has no solution for ${\rm Re}\lambda\ge -b_0$ and $|s|\le r_0$.
\end{lem}

\begin{lem}\label{eigen_2a}
There are constants $b_1, r_0,r_1>0$ such that the equation $D_1(\lambda,s)=0$ with ${\rm Re}\lambda\ge -b_1$ has only one solution $\lambda(s)$  for $(s,\lambda)\in[-r_0, r_0]\times B_{r_1}(0)$ and it satisfies
$$\lambda(0)=0,\quad \lambda'(0)=0,\quad
\lambda''(0)=\frac1{(L_1^{-1}\chi_2,\chi_2)}.$$
\end{lem}

\begin{proof}
Since
  \bq D_1(0,0)=0,\quad  {\partial_s}D_1(0,0)=0, \quad  {\partial_\lambda}D_1(0,0)=-(L_1^{-1}\chi_2,\chi_2),  \label{lamd1}
 \eq
the application of the implicit function theorem implies that  there exist
constants $r_0,r_1>0$ and a unique $C^\infty$ function $\lambda_0(s)$ such that $D_1(\lambda_0(s),s)=0$ for $(s,\lambda)\in [-r_0,r_0]\times B_{r_1}(0)$. In particular,
\bq
 \lambda_0(0)=0\quad {\rm and}\quad \lambda_0'(0)=-\frac{{\partial_s}D(0,0)}{{\partial_\lambda}D(0,0)}=0.    \label{lamd2}
 \eq
A direct computation gives
$
{\partial^2_s}D(0,0)=2,
$
which together with \eqref{lamd1}  yield
 \bq
  \lambda_0''(0)
  =-\frac{\partial_s^2D(0,0)}{\partial_\lambda D(0,0)}
   =\frac2{(L_1^{-1}\chi_2,\chi_2)}.          \label{lamd3}
 \eq
Let
$$ D_2(\lambda,s)=\frac{D_1(\lambda,s)}{\lambda-\lambda_0(s)}.$$
Similarly to Lemma \ref{eigen_1a}, we can prove that there is $b_1>0$ so that
$ D_2(\lambda,0)=\lambda-((L_1-\lambda P_r)^{-1}\chi_2,\chi_2)\ne 0$ for ${\rm Re}\lambda\ge -b_1$. By \eqref{S_3}, we have $|((L_1-\lambda P_r)^{-1}\chi_1,\chi_1)|\le C$ for ${\rm Re}\lambda\ge -b_1$, which leads to $\lim_{|\lambda|\to\infty}|D_2(\lambda,0)|=\infty$. This together with
 the continuity of $D_2(\lambda,0)$ imply that there is a constant $\delta_1>0$ such that $|D_2(\lambda,0)|\ge \delta_1$ for ${\rm Re}\lambda\ge-b_1$.

Since
\bmas
 D_2(\lambda,s)&=\frac{D_1(\lambda,s)-D_1(\lambda_0(s),s)}{\lambda-\lambda_0(s)}\\
 &=\lambda-(R(\lambda,se_1)\chi_2,\chi_2)+\lambda_0(s) +\lambda_0(s)\frac{([R(\lambda,se_1)-R(\lambda_0(s),se_1)]\chi_2,\chi_2)}{\lambda-\lambda_0(s)},
 \emas
it follows that
\bmas
|D_2(\lambda,s)-D_2(\lambda,0)|\le C(|s|+|\lambda_0(s)|)\to0, \quad{\rm as}\quad  s\to0.
\emas
Thus for $r_0>0$ small enough,
$$|D_2(\lambda,s)|\ge |D_2(\lambda,0)|-|D_2(\lambda,s)-D_2(\lambda,0)|>0,\quad {\rm Re}\lambda>-b_1,\,\,\ |s|\le r_0.$$
We then conclude that  the equation $D(\lambda,s)=0$ with ${\rm Re}\lambda>-b_1$ has only one solution $\lambda_0(s)$ for $s\in[-r_0,r_0]$.
\end{proof}

With  Lemmas \ref{eigen_1a}--\ref{eigen_2a}, we are able to
construct the eigenvector $\Psi_j(s,\omega)$ corresponding to the
eigenvalue $\lambda_j$ in the low frequency. Indeed, we have
\begin{thm}\label{eigen_3}
There exist two constants $r_0>0$ and $b_2>0$ so that the spectrum $\lambda\in\sigma(\hat{\AA}_1(\xi))\subset\mathbb{C}$ for $\xi=s\omega$ with $|s|\leq r_0$ and $\omega\in \mathbb{S}^2$ consists of two points $\{\lambda_j(s),\ j=1,2\}$ in the domain $\mathrm{Re}\lambda>-b_2$. The spectrum $\lambda_j(s)$ and the corresponding eigenvector $\Psi_j(s,\omega)=(\psi_j,X_j,Y_j)(s,\omega)$ are $C^\infty$ functions of $s$ for $|s|\leq r_0$. In particular, the eigenvalues admit the following asymptotic expansion for $|s|\leq r_0$
 \be                                  \label{specr0}
 \lambda_{1}(s) =\lambda_{2}(s) = -a_1s^2+o(s^2),
 \ee
where
\bq
a_1=-\frac1{(L_1^{-1}\chi_2,\chi_2)}>0.
\eq
The eigenvectors $\Psi_j=(\psi_j,X_j,Y_j)$ are   orthogonal to each other and satisfy
 \be
 \left\{\bln
 &(\Psi_i(s,\omega),\Psi^*_j(s,\omega))=(\psi_i,\overline{\psi_j})-(X_i,\overline{X_j})-(Y_i,\overline{Y_j})=\delta_{ij},
  \quad  i\ne j=1,2,                                  \label{eigfr0}
 \\
&(\psi_j,X_j,Y_j)(s,\omega)
 =(\psi_{j,0},X_{j,0},Y_{j,0})(\omega) +(\psi_{j,1},X_{j,1},Y_{j,1})(\omega)s+O(s^2), \quad |s|\leq r_0,
 \eln\right.
 \ee
where $ \Psi^*_j=(\overline{\psi_j},-\overline{X_j},-\overline{Y_j})$, and the coefficients $\psi_{j,n},X_{j,n},Y_{j,n}$  are given by
 \bq
  \left\{\bln                      \label{eigf1}
 &\psi_{j,0}=0,\quad P_{\rm d}\psi_{j,n}=0\,\,\, (n\geq0),
 \quad X_{j,0}=0,\quad Y_{j,0}=\i W^j,\\
  &  \psi_{j,1}=- a_1L_1^{-1}P_r  (v\cdot W^j)\sqrt{M},\quad X_{j,1}=a_1\omega\times W^j,\quad Y_{j,1}=0.
  \eln\right.
  \eq
Here, $W^j$ $(j=1,2)$ are two orthonormal vectors satisfying
$W^j\cdot\omega=0$.
\end{thm}

\begin{proof}
The eigenvalues $\lambda_j(s)$ and the eigenvectors $\Psi_j(s,\omega)=(\psi_j,X_j,Y_j)(s,\omega)$, $j=1,2$, can be constructed as follows.  Let $b_2=\min\{b_0,b_1\}$ and take $\lambda_j=\lambda(s)$ to be the solution of the equation  $D_1(\lambda,s)=0$ defined in Lemma \ref{eigen_2a}, and
choose $C_0=0$, and $X_j=\omega\times W^j$ with  $W^j$, $j=1,2$,
 being two linearly independent vectors
satisfying  $W^j\cdot\omega=0$ and $W^1\cdot W^2=0$. The corresponding eigenvectors $\Psi_j(s,\omega)=(\psi_j,X_j,Y_j)(s,\omega)$ are defined by
 \bq
 \left\{\bln       \label{C_3}
  &\psi_j(s,\omega)
 =-[L_1-\lambda_j P_r-\i s P_r(v\cdot\omega) P_r]^{-1}
         P_r  (v\cdot W^j)\sqrt{M},  \\
  &X_j(s,\omega)=\omega\times W^j,\quad Y_j(s,\omega)= \frac{\i s}{\lambda_j} W^j, \quad j=1,2,
  \eln\right.
\eq
which satisfy   $(\Psi_1(s,\omega),  \Psi_2^*(s,\omega))=0.$
We can normalize them by taking
$$(\Psi_j(s,\omega),\Psi^*_j(s,\omega))=1, \quad j=1,2.$$
The coefficients $W^j=b_j(s)T^j(\omega)$ for $j=1,2$ with $b_j\in \R$ and $T^j=(T^j_1,T^j_2,T^j_3)\in \S^2$ are determined by the normalization condition  as
 \bq
 \left\{\bln    \label{C_2}
 &b_j(s)^2(D_j(s)-1+\frac{s^2}{\lambda_j(s)^2})=1,\quad j=1,2,\\
 & |T^1|=|T^2|=1,\quad T^1\cdot\omega=T^2\cdot\omega=T^1\cdot T^2=0,
 \eln\right.
 \eq
 where $D_j(s)=(R(\lambda_j(s),s e_1)\chi_1, R(\overline{\lambda_j(s)},-se_1)  \chi_1).$

To study the asymptotic expression of eigenvectors in the
low frequency, we can take  Taylor expansion of both eigenvalues and
eigenvectors as
$$
 \lambda_{j}(s)= \sum_{n=0}^2 \lambda_{j,n}s^n +O(s^3),\quad
 (\psi_j,X_j,Y_j)(s,\omega)=\sum_{n=0}^1 (\psi_{j,n},X_{j,n},Y_{j,n})(\omega)s^n +O(s^2).
$$

 Substituting the above expansions into \eqref{C_2}, we obtain $b_j(0)=0,\ b_j'(0)=a_1$ and $b_j(-s)=-b_j(s)$. This and \eqref{C_3} give the expansion of $\Psi_j(s,\omega)$ for $j=1,2$, stated in \eqref{eigf1}. And then it completes
the proof of the theorem.
\end{proof}

\subsection{Asymptotics in  high frequency}

We now turn to study the asymptotic expansions of the eigenvalues
and eigenvectors in the high frequency region. Firstly, recalling
the eigenvalue problem
 \bma
  \lambda f
  &=B_1(\xi)f-v\sqrt M\cdot(\omega\times X),\label{L_3}\\
  \lambda X&=-\omega\times (f,v\sqrt M)+\i\xi\times Y,\label{L_3a}\\
  \lambda Y&=-\i\xi\times X,\quad |\xi|\ne0.\nnm
\ema
By Lemma \ref{LP03}, there is $R_0>0$ large enough such that the operator $\lambda-B_1(\xi)$ is invertible on $L^2_\xi(\R^3)$ for ${\rm Re}\lambda\ge-\nu_0/2$ and $|\xi|>R_0$. Then it follows from \eqref{L_3} that
\bq f=(B_1(\xi)-\lambda)^{-1}v\sqrt M\cdot(\omega\times X),\quad |\xi|>R_0.\label{L_5}\eq
Substituting \eqref{L_5} into \eqref{L_3a} and using the transformation
$$((B_1(\xi)-\lambda)^{-1}\chi_i,\chi_j) =\omega_i\omega_j((B_1(|\xi|e_1)-\lambda)^{-1}\chi_1,\chi_1)+(\delta_{ij}-\omega_i\omega_j)((B(|\xi|e_1)-\lambda)^{-1}\chi_2,\chi_2),$$
we obtain
\bma
\lambda X&=((B_1(|\xi|e_1)-\lambda)^{-1}\chi_2,\chi_2)X+\i\xi\times Y,\label{A_12}\\
  \lambda Y&=-\i\xi\times X,\quad |\xi|>R_0.  \label{A_13}
\ema
Multiplying \eqref{A_12} by $\lambda$ and using \eqref{A_13} and \eqref{rotat}, we obtain
\bq (\lambda^2-((B_1(|\xi|e_1)-\lambda)^{-1}\chi_2,\chi_2)\lambda+|\xi|^2)X=0, \quad |\xi|>R_0.\eq
Denote
\bq D(\lambda,s)=\lambda^2-((B_1(se_1)-\lambda)^{-1}\chi_2,\chi_2)\lambda+s^2,\quad s>R_0.\label{D3a}\eq

Similar to the proof of Lemma \ref{LP03}, we can obtain
\begin{lem}\label{lem1}
 For any $\delta>0$, we have
 \bgr
\sup_{x\geq-\nu_0+\delta,y\in\R}\|(x+\i y-c(\xi))^{-1}K_1\|
  \leq C\delta^{-15/13}(1+|\xi|)^{-2/13},\label{T_2}
\\
   \sup_{x\geq -\nu_0+\delta,y\in\R}
  \|(x+\i y-c(\xi))^{-1}(v\cdot\xi)|\xi|^{-2}P_{\rm d}\|
 \leq C\delta^{-1}|\xi|^{-1},\label{T_3}
 \\
 \sup_{x\geq-\nu_0+\delta,y\in\R}\|(x+\i y-c(\xi))^{-1}v\sqrt M\|
  \leq C\delta^{-1/2}(1+|\xi|)^{-1/2}. \label{T_4}
 \egr
As consequence, it holds that for $|\xi|>R_0$
\bq \sup_{x\geq-\nu_0+\delta,y\in\R}\|(x+\i y-B_1(\xi))^{-1}v\sqrt M\|
  \leq C\delta^{-1/2}(1+|\xi|)^{-1/2}. \label{T_5}\eq
\end{lem}

We now study the equation \eqref{D3a} as follows.

\begin{lem}\label{eigen_5}
There is a constant $R_1>0$ such that the equation $D(\lambda,s)=0$ has two solutions $\lambda_j(s)$, $j=\pm1$,  for $s>R_1$ satisfying
$$|\lambda_j(s)-j\i s|\le Cs^{-1/2}\to 0, \quad {\rm as} \quad s\to \infty.$$
In particular, there are two constants $c_1,c_2>0$ such that
\bq c_1\frac1s\le -{\rm Re}\lambda_j(s)\le c_2\frac1s,\quad j=\pm1.\label{eigen_h}\eq
\end{lem}
\begin{proof}
For any fixed $s>R_0$, we define a function of $\lambda$ as
\bq G^1_j(\lambda)=\frac12\Big(R_{22}(\lambda,s)+j\sqrt{R_{22}(\lambda,s)^2-4s^2}\Big),\quad j=\pm1,\,\,\, s>R_0,\label{fp1}\eq
where $R_{22}(\lambda,s)=((B_1(se_1)-\lambda)^{-1}\chi_2,\chi_2)$.
It is straightforward to verify that a solution of $D(\lambda,s)=0$ for any fixed $s>R_0$ is a fixed point of $G^1_j(\lambda)$.

Consider an equivalent equation of \eqref{fp1} as
\bq G^2_j(\beta)=G^1_j(\lambda)-j\i s=\frac12\Big(B_j(\beta,s)+\frac{jB_j(\beta,s)^2}{\sqrt{B_j(\beta,s)^2-4s^2}+\i s}\Big),\quad j=\pm1,\,\,\, s>R_0,\eq
where $B_j(\beta,s)=((B_1(se_1)-j\i s-\beta)^{-1}\chi_2,\chi_2)$.
By \eqref{T_5}, when $R_1>0$ is large enough and $\delta>0$ is
small enough, it holds for $s>R_1$ and $|\beta|\le \delta$ that
\bmas
|G^2_j(\beta)|\le \delta,\quad |G^2_j(\beta_1)-G^2_j(\beta_2)|\le \frac12|\beta_1-\beta_2|.
\emas
Thus $G^2_j(\beta)$ is a contraction mapping on $B_\delta(0)$ and  there is a unique fixed point $\beta_j(s)$ of $G^2_j(\beta)$. Thus $\lambda_j(s)=j\i s+\beta_j(s)$ is the solution of $D(\lambda,s)=0$  and $|\beta_j(s)|\le Cs^{-1/2}$ because $|B_j(\beta,s)|\le Cs^{-1/2}$ due to \eqref{T_5}.

We now turn to prove \eqref{eigen_h}. For this, we decompose $\beta_j(s)$, $j=\pm1$, into
\bma
\beta_j(s)=\frac12 B_j(\beta_j,s)+\frac12\frac{jB_j(\beta_j,s)^2}{\sqrt{B_j(\beta_j,s)^2-4s^2}+\i s} =I_1+I_2.\label{ddd}
\ema
First, we estimate $I_1$. Since
\bmas
-{\rm Re}I_1=-\frac12((L_1-{\rm Re}\beta_j)g_j,g_j)=-\frac12(L_1g_j,g_j)+O(\frac1{\sqrt s})(g_j,g_j),
\emas
with $g_j=(L_1-\i(v_1+j)s-\i \frac{v_1}s P_{\rm d}-\beta_j)^{-1}\chi_2$ for $j=\pm1$, we obtain
\bq
C_0(g_j,g_j)\le -{\rm Re}I_1\le C_1(\nu(v)g_j,g_j).\label{aaa}
\eq
Note that
\bmas
(L_1-\i(v_1\pm 1)s-\i \frac{v_1}s P_{\rm d}-\beta_{\pm1} )^{-1}&=(I-(\nu+\i(v_1\pm 1)s)^{-1}(K_1-\i \frac{v_1}s P_{\rm d}-\beta_{\pm1}))^{-1}(-\nu-\i(v_1\pm 1)s)^{-1}\\
&=(-\nu-\i(v_1\pm 1)s)^{-1}+Z(\beta_{\pm1},s),
\emas
with
\bmas
Z(\beta_{\pm1},s)&=Y(\beta_{\pm1},s)(I+Y(\beta_{\pm1},s))^{-1}(-\nu-\i(v_1\pm 1)s)^{-1},\\
Y(\beta_{\pm1},s)&=(-\nu-\i(v_1\pm 1)s)^{-1}(K_1-\i \frac{v_1}s P_{\rm d}-\beta_{\pm1}),
\emas
we have
\bmas
(\nu(v)g_{\pm1},g_{\pm1})&\le 2(\nu(\nu+\i(v_1\pm 1)s)^{-1}\chi_2,(\nu+\i(v_1\pm 1)s)^{-1}\chi_2)+2(\nu Z(\beta_{\pm1},s)\chi_2,Z(\beta_{\pm1},s)\chi_2),\\
(g_{\pm1},g_{\pm1})&\ge \frac12((\nu+\i(v_1\pm 1)s)^{-1}\chi_2,(\nu+\i(v_1\pm 1)s)^{-1}\chi_2)-( Z(\beta_{\pm1},s)\chi_2,Z(\beta_{\pm1},s)\chi_2).
\emas
In the following, we will prove
\bq I_3=:((\nu+\i(v_1\pm 1)s)^{-1}\chi_2,(\nu+\i(v_1\pm 1)s)^{-1}\chi_2)\ge \frac{C_3}{s},\label{bbb}\eq
for some constant $C_3>0$.
Indeed, by changing variable $(u_1,u_2,u_3)=((v_1\pm1)s,v_2,v_3)$, we obtain for $s>1$ that
\bmas
I_3&\ge \intr \frac{1}{\nu_1^2(1+|v|^2)+(v_1\pm1)^2s^2}v_2^2Mdv\\
&\ge \frac1s\intr \frac{1}{\nu_1^2(1+(\frac{u_1}s\mp1)^2+u^2_2+u^2_3)+u_1^2}u_2^2e^{-\frac12(\frac{u_1}s\mp1)^2}e^{-\frac{u_2^2}2}e^{-\frac{u_3^2}2}du\\
&\ge \frac1s\intr \frac{1}{\nu_1^2(3+2u_1^2+u^2_2+u^2_3)+u_1^2}u_2^2e^{-u_1^2-1}e^{-\frac{u_2^2}2}e^{-\frac{u_3^2}2}du\ge \frac{C_3}{s}.
\emas
As for
\bq I_4=:(\nu(\nu+\i(v_1\pm 1)s)^{-1}\chi_2,(\nu+\i(v_1\pm 1)s)^{-1}\chi_2),\label{ccc}\eq
 by the change variables $(u_1,u_2,u_3)=((v_1\pm1)s,v_2,v_3)$, we obtain for $s>1$ that
\bmas
I_4&\le \intr \frac{\nu_0^2(1+|v|^2)}{\nu_1^2+(v_1\pm1)^2s^2}v_2^2Mdv\le C\intr \frac{1}{\nu_1^2+(v_1\pm1)^2s^2}v_2^2e^{-\frac{|v|^2}4}dv\\
&\le C\frac1s\intr \frac{1}{\nu_1^2+u_1^2}u_2^2e^{-\frac14(\frac{u_1}s\mp1)^2}e^{-\frac{u_2^2}4}e^{-\frac{u_3^2}4}du\\
&\le C\frac1s\intr \frac{1}{\nu_1^2+u_1^2}u_2^2e^{-\frac{u_2^2}4}e^{-\frac{u_3^2}4}du\le \frac{C_4}{s},
\emas
where $C_4>0$ is a constant.

By Lemma \ref{lem1}, for any $0<\epsilon\ll 1$ there exists $s>R_1$ such that
\bmas
&\| Y(\beta_{\pm1},s)\|\le \|(\nu+\i(v_1\pm 1)s)^{-1}(K_1-\frac{\i v_1}sP_{\rm d})\|+O(\frac1{\sqrt s})\|(\nu+\i(v_1\pm 1)s)^{-1}\|\le \epsilon,\\
 &\|\nu Y(\beta_{\pm1},s)\|\le \|\nu(\nu+\i(v_1\pm 1)s)^{-1}\|\|K_1-\i \frac{v_1}s P_{\rm d}-\beta_{\pm1}\|\le C.
\emas
This and \eqref{ccc} lead to
\bmas
(\nu Z(\beta_{\pm1},s)\chi_2,Z(\beta_{\pm1},s)\chi_2)&\le \|\nu Y(\beta_{\pm1},s)\|\| Y(\beta_{\pm1},s)\|\|(I+Y(\beta_{\pm1},s))^{-1}\|^2\|(\nu+\i(v_1\pm 1)s)^{-1}\chi_2\|^2\\
&\le C\epsilon I_4\le \frac{C\epsilon}{s}.
\emas
Thus, by combining  with \eqref{aaa}, \eqref{bbb} and \eqref{ccc},
there exist constants $C_5, C_6>0$ such that
\bq \frac{C_5}{s}\le -{\rm Re}I_1\le \frac{C_6}{s}.\label{eea}\eq
For $I_2$, we have
\bq |I_2|\le \frac{C}{s}|B_{\pm1}(\beta_{\pm1},s)|^2\le \frac{C}{s^2}.\label{abb}\eq
Combining \eqref{ddd}, \eqref{eea} and \eqref{abb}, we obtain \eqref{eigen_h}. The proof of the lemma is then completed.
\end{proof}

The following theorem gives the asymptotic expansions of eigenvalues
and eigenvectors in the high frequency region.

\begin{thm}\label{eigen_4}
There exists a constant $r_1>0$ such that the spectrum $\sigma(\hat{\AA}_1(\xi))\subset\mathbb{C}$ for $\xi=s\omega$ with $s=|\xi|> r_1$ and $\omega\in \mathbb{S}^2$ consists of four eigenvalues $\{\beta_j(s),\ j=1,2,3,4\}$ in the domain $\mathrm{Re}\lambda>-\mu/2$. In particular, the eigenvalues satisfy
 \bgr
 \beta_1(s) = \beta_2(s) =-\i s+O(s^{-1/2}),\label{specr1}\\
 \beta_3(s) = \beta_4(s) =\i s+O(s^{-1/2}),\label{specr2}\\
\frac{c_1}{s}\le -{\rm Re}\beta_{j}(s)\le \frac{c_2}{s},
\egr
for two positive constants $c_1$ and $c_2$.
The eigenvectors $\Phi_j(s,\omega)=(\phi_j,X_j,Y_j)(s,\omega)$ are   orthogonal to each other and satisfy
 \be
 (\Phi_i(s,\omega),\Phi^*_j(s,\omega))=(\phi_i,\overline{\phi_j})-(X_i,\overline{X_j})-(Y_i,\overline{Y_j})=\delta_{ij},
  \quad  1\le i\ne j\le 4,                                  \label{eigfr2}
 \ee
where $ \Phi^*_j=(\overline{\phi_j},-\overline{X_j},-\overline{Y_j})$.
Moreover,
 \bq
  \|\phi_j(s,\omega)\|=O(\frac1{\sqrt s}),\quad P_{\rm d}\phi_j(s,\omega)=0,\quad X_j(s,\omega)=O(1)\i (\omega\times W^j),\quad Y_j(s,\omega)=O(1)\i  W^j. \label{eigfr3}
  \eq
 Here, $W^j$ $(j=1,2,3,4)$ are  vectors satisfying
$W^j\cdot\omega=0$, $W^1\cdot W^2=0$,  $W^1=W^3, W^2=W^4$ and the normalization
condition \eqref{C_2a}.
\end{thm}

\begin{proof}
The eigenvalue $\lambda_j(s)$ and the eigenvector $\Phi_j(s,\omega)=(\phi_j,X_j,Y_j)(s,\omega)$ can be constructed as follows. For $j=1,2,3,4$, we take $\lambda_1=\lambda_2=\lambda_{-1}(s)$ and $\lambda_3=\lambda_4=\lambda_{1}(s)$ to be the solution of the equation  $D(\lambda,s)=0$ defined in Lemma \ref{eigen_5}.
Choose $X_j=\omega\times W^j$ with $W^1=W^3$
and $W^2=W^4$ as linearly independent vectors
so that  $W^j\cdot\omega=0$ and $W^1\cdot W^2=0$. The corresponding eigenvectors $\Phi_j(s,\omega)=(\phi_j,X_j,Y_j)(s,\omega)$, $1\le j\le 4$,  are defined by
 \bq
 \left\{\bln       \label{C_3a}
  &\phi_j(s,\omega)
 =-[L_1-\beta_j -\i s (v\cdot\omega)-\frac{\i(v\cdot\omega)}s P_{\rm d} ]^{-1}
           (v\cdot W^j)\sqrt{M},  \\
  &X_j(s,\omega)=\omega\times W^j,\quad Y_j(s,\omega)= \frac{\i s}{\lambda_j} W^j,
  \eln\right.
\eq
which satisfy   $(\Phi_1(s,\omega),  \Phi_2^*(s,\omega))=(\Phi_3(s,\omega),  \Phi_4^*(s,\omega))=0.$

Rewrite the eigenvalue problem as
$$
 \hat{\AA}_1(\xi)\Phi_j(s,\omega)
 =\beta_j(s)\Phi_j(s,\omega), \quad 1\leq j\leq 4,\ |s|\le r_0.$$
By taking the inner product
$(\cdot,\cdot)_\xi$ of it with $\Phi^*_j(s,\omega)$, and using
the fact that
 \bgrs
(\hat{\AA}_1(\xi) U,V)_\xi=(U,\hat{\AA}_1(\xi)^*V)_\xi,\quad U,V\in  D(\hat{B}_1(\xi))\times \mathbb{C}^3_\xi\times \mathbb{C}^3_\xi,\\
\hat{\AA}_1(\xi)^*\Phi^*_j(s,\omega)
=\overline{\lambda_j(s)} \Phi^*_j(s,\omega),\egrs
we have
 $$
(\beta_j(s)-\beta_{k}(s))(\Phi_j(s,\omega),\Phi^*_j(s,\omega))_\xi=0,\quad 1\le j, k\le 4.
$$
Since  $\beta_j(s)\neq \beta_{k}(s)$ for  $j=1,2,\, k=3,4$ and $P_{\rm d}\phi_j(s,\omega)=0$,
we have the orthogonal relation
 $$
(\Phi_j(s,\omega),\Phi^*_k(s,\omega))_\xi=(\Phi_j(s,\omega),\Phi^*_k(s,\omega))=0,\quad 1\leq j\neq k\leq 4.
 $$
 This can be normalized so that
$$(\Phi_j(s,\omega),\Phi^*_j(s,\omega))=1, \quad j=1,2,3,4.$$
Precisely, denote $W^j=b_j(s)T^j(\omega)$ for $j=1,2,3,4$, with $b_j\in \R$ and $T^j=(T^j_1,T^j_2,T^j_3)\in \S^2$,
then the coefficients $b_j$, $j=1,\cdots, 4$  are determined by the normalization condition
 \bq
 \left\{\bln    \label{C_2a}
 &b_j(s)^2(D_j(s)-1+\frac{s^2}{\lambda_j(s)^2})=1,\quad j=1,2,3,4,\\
 & |T^j|=1,\quad T^j\cdot\omega=T^1\cdot T^2=0,\quad T^1=T^3,\,\, T^2=T^4,
 \eln\right.
 \eq
 where $D_j(s)=((B_1(se_1)-\beta_j(s))^{-1}\chi_1, (B_1(-se_1)-\overline{\beta_j(s)})^{-1}  \chi_1).$ By substituting \eqref{T_5}, \eqref{specr1} and \eqref{specr2} into \eqref{C_2a} and \eqref{C_3a}, we obtain \eqref{eigfr3} so that
the proof of the theorem is completed.
\end{proof}

\section{Spectral analysis for one-species case}
\label{spectrum-one}

\setcounter{equation}{0}

In this section, we will study the spectrum structure of the linearized
system~\ref{LVMB2} of one-species VMB. It is interesting to find out that its
structure is very different in the low frequency region from the
case of two-species. And this
essential difference comes from the lack of the cancellation in
the one-species system.

\subsection{Spectrum structure}

Since $L^2_\xi(\R^3_v)\times \mathbb{C}^3_\xi\times \mathbb{C}^3_\xi$ is the invariant subspace of the operator $\hat{\AA}_3(\xi)$, we can regard $\hat{\AA}_3(\xi)$ as a linear operator on $L^2_\xi(\R^3_v)\times \mathbb{C}^3_\xi\times \mathbb{C}^3_\xi$.
We have for any $U,V\in L^2_\xi(\R^3_v)\cap D(\hat{B}_2(\xi))\times \mathbb{C}^3_\xi\times \mathbb{C}^3_\xi$,
$$(\hat{\AA}_3(\xi)U,V)_\xi=(U,\hat{\AA}_3(-\xi)V)_\xi.$$

Denote by $\rho(\hat{\AA}_3(\xi))$  the resolvent set and
by $\sigma(\hat{\AA}_3(\xi))$ the spectrum set of $\hat{\AA}_3(\xi)$.  Similar to  Lemmas \ref{SG_1} and \ref{Egn}, we have the following lemmas.

\begin{lem}\label{1SG_1}
The operator $\hat{\AA}_3(\xi)$ generates a strongly continuous contraction semigroup on
$L^2_\xi(\R^3_v)\times \mathbb{C}^3_\xi\times \mathbb{C}^3_\xi$ satisfying
 \bq
\|e^{t\hat{\AA}_3(\xi)}U\|_\xi\le\|U\|_\xi, \quad\mbox{for}\ t>0,\, U\in
L^2_\xi(\R^3_v)\times \mathbb{C}^3_\xi\times \mathbb{C}^3_\xi.
 \eq
\end{lem}

\begin{lem}\label{1Egn}
For each $\xi\ne 0$, the spectrum set $\sigma(\hat{\AA}_3(\xi))$ of the
operator $\hat{\AA}_3(\xi)$ in the domain
$\mathrm{Re}\lambda\geq-\nu_0+\delta$ for any constant $\delta>0$
consists of isolated eigenvalues $\Sigma=:\{\lambda_j(\xi)\}$ with
$\mathrm{Re}\lambda_j (\xi)<0$.
\end{lem}

We will make use of the following decomposition associated with the
operator $\hat{\AA}_3(\xi)$ for  $|\xi|>0$
\bma
\lambda-\hat{\AA}_3(\xi)=\lambda- G_1(\xi)-G_5(\xi)=(I-G_5(\xi)(\lambda-G_1(\xi))^{-1})(\lambda-G_1(\xi)),  \label{1B_d}
\ema
where $G_1(\xi)$ is defined by \eqref{A12}, and
\bma
 G_5(\xi)=\left( \ba
K-\frac{\i(v\cdot\xi)}{|\xi|^2}P_{\rm d} &-v\sqrt M\cdot\omega\times &0\\
-\omega\times P_m &0 &0\\
0 &0 &0
\ea\right).
 \ema
Here,
\bma
G_5(\xi)(\lambda-G_1(\xi))^{-1}&= \left(\ba X^1_1(\lambda,\xi) & X_2(\lambda,\xi) \\  X_3(\lambda,\xi) & 0 \ea\right)_{7\times7},
\ema
where 
$X_2(\lambda,\xi)$ and $X_3(\lambda,\xi)$ are defined by \eqref{X_2} and \eqref{X_3}, and
\bma
&X^1_1(\lambda,\xi)=(K-\frac{\i (v\cdot\xi)}{|\xi|^2} P_{\rm d})(\lambda -c(\xi))^{-1}.
\ema

By  a similar argument as the one for
Lemma \ref{LP01}, we can obtain the spectrum
of the operator $\hat{\AA}_3(\xi)$ in the intermediate and
high frequency regions.
\begin{lem}
\label{LP01a}
In the high and intermediate regions  of the frequency, we have
 \begin{enumerate}
\item  For any $\delta_1,\delta_2>0$, there
exists $ r_1= r_1(\delta_1,\delta_2)>0$  so that for $|\xi|>r_1$,
 \bq
 \sigma(\hat{\AA}_3(\xi))\cap\{\lambda\in\mathbb{C}\,|\,\mathrm{Re}\lambda\ge-\nu_0+\delta_1\}
 \subset
 \dcup_{j=\pm1}\{\lambda\in\mathbb{C}\,|\,|\lambda-j\i|\xi||\le\delta_2\}.\label{1sg4}
 \eq
\item For any $r_1>r_0>0$, there
exists $\alpha =\alpha(r_0,r_1)>0$ so that for all $r_0\le |\xi|\le r_1$,
\bq \sigma(\hat{\AA}_3(\xi))\subset\{\lambda\in\mathbb{C}\,|\, \mathrm{Re}\lambda(\xi)\leq-\alpha\} .\label{1sg3}\eq
 \end{enumerate}
\end{lem}

Then, we only need
to study the spectrum and resolvent sets of $\hat{\AA}_3(\xi)$ in
low frequency region. For this, we decompose $\lambda-\hat{\AA}_3(\xi)$ as
\bq \lambda-\hat{\AA}_3(\xi)=\lambda P_A-G_6(\xi)+\lambda P_B-G_7(\xi)+G_8(\xi)+G_9(\xi),\label{1Bd3}\eq
where \bma
&G_6(\xi)=\left(\ba B_4(\xi) &  -v\sqrt M\cdot \omega\times & 0\\ -\omega\times P_m & 0 & \i \xi \times \\ 0 & -\i\xi\times & 0 \ea\right),\quad
G_7(\xi)=\left(\ba B_5(\xi) &  0 & 0\\ 0 & 0 & 0 \\ 0 & 0 & 0 \ea\right),\label{1Qxi}\\
&G_8(\xi)=\left(\ba \i \P_1(v\cdot\xi)\P_0 & 0 & 0\\ 0 & 0 & 0 \\ 0 & 0 & 0 \ea\right),\quad G_9(\xi)=\left(\ba \i \P_0(v\cdot\xi)\P_1 & 0 & 0\\ 0 & 0 & 0 \\ 0 & 0 & 0 \ea\right),
\\
&B_4(\xi)=\i \P_0(v\cdot\xi)\P_0-\frac{\i(v\cdot\xi)}{|\xi|^2}P_{\rm d},\quad B_5(\xi)=L-\i \P_1(v\cdot\xi)\P_1,
\ema
and $P_A$, $P_B$ are the orthogonal  projection operators defined by
\bma
&P_A=\left(\ba \P_0 &  0 & 0\\ 0 & I_{3\times 3} & 0 \\ 0 & 0 & I_{3\times 3} \ea\right),\quad P_B=\left(\ba \P_1 &  0 & 0\\ 0 & 0 & 0 \\ 0 & 0 & 0 \ea\right).
\ema
It is straightforward to verify that $G_6(\xi)$ is a linear operator from the
 space $N_0\times \mathbb{C}^3_\xi\times \mathbb{C}^3_\xi$ to itself, 
which admits nine eigenvalues $\alpha_j(\xi)$ satisfying
 \be \label{eigen}
 \left\{\bln
 &\alpha_j(\xi)=0,\quad j=0,2,3,\quad
\alpha_{\pm1}(\xi)=\pm \i\mbox{$\sqrt{1+\frac53|\xi|^2}$}, \\
&\alpha_4(\xi)=\alpha_5(\xi)=-\i\mbox{$\sqrt{1+|\xi|^2}$},\quad \alpha_6(\xi)=\alpha_7(\xi)= \i\mbox{$\sqrt{1+|\xi|^2}$}.
\eln\right.
\ee

\begin{lem}\label{LP1}
Let $\xi\neq0$, we have the following properties for the linear operators $G_6(\xi)$ and $G_7(\xi)$ defined by \eqref{1Qxi}.
 \begin{enumerate}
\item  If $\lambda\neq\alpha_j(\xi)$, then  the operator $\lambda  P_A-G_6(\xi)$ is
invertible on $N_0\times \mathbb{C}^3_\xi\times \mathbb{C}^3_\xi$ and satisfies
\bgr
  \|(\lambda  P_A-G_6(\xi))^{-1}\|_\xi
  =\max_{-1\leq j \leq 7}\(|\lambda-\alpha_j(\xi)|^{-1}\),\label{1S_2a}
\\
  \| G_8(\xi) (\lambda  P_A-G_6(\xi))^{-1} P_A\|_\xi
 \le C|\xi|\max_{-1\leq j \leq 7}\(|\lambda-\alpha_j(\xi)|^{-1}\),\label{1S_2}
 \egr
where $\alpha_j(\xi)$, $-1\le j\le 7$, are the eigenvalues of $G_6(\xi)$
defined by \eqref{eigen}.

\item  If $\mathrm{Re}\lambda>-\mu $, then the operator $\lambda  P_B-G_7(\xi)$ is
invertible on $N_0^\bot\times \{0\}\times \{0\}$ and satisfies
 \bgr
 \|(\lambda  P_B-G_7(\xi))^{-1}\|\leq(\mathrm{Re}\lambda+\mu )^{-1},  \label{1S_3}
\\
 \| G_9(\xi) (\lambda  P_B-G_7(\xi))^{-1} P_B\|_\xi
 \leq
 C(1+|\lambda|)^{-1}[(\mathrm{Re}\lambda+\mu )^{-1}+1](|\xi|+|\xi|^2). \label{1S_5}
 \egr
\end{enumerate}
\end{lem}
\begin{proof}
Since the operator $\i G_6(\xi)$ is self-adjoint on $N_0\times \mathbb{C}^3_\xi\times \mathbb{C}^3_\xi$, namely,
 \bma
  (\i G_6(\xi)U,V)_\xi =(U,\i G_6(\xi)V)_\xi,\quad U,V\in N_0\times \mathbb{C}^3_\xi\times \mathbb{C}^3_\xi.\label{symmetric}
 \ema
 we can prove \eqref{1S_2a}--\eqref{1S_2}. And by the dissipative of the operator $G_7(\xi)$ on $N_0^\bot\times \{0\}\times \{0\}$, we can prove \eqref{1S_3}--\eqref{1S_5}.
\end{proof}

With the help of  Lemmas~\ref{1Egn}--\ref{LP1}, we obtain the spectral
and resolvent sets of the operator $\hat{\AA}_3(\xi)$  by applying the similar argument as Lemma \ref{spectrum3}.

\begin{lem}
\label{spectruma}
For any  $\delta_1,\delta_2>0$, there are
 $r_1=r_1(\delta_1,\delta_2),r_2=r_2(\delta_1,\delta_2),y_1=y_1(\delta_1,\delta_2)>0$ such that
 \begin{enumerate}
\item   it holds for all $|\xi|\ne 0$ that the resolvent set of $\hat{\AA}_3(\xi)$
contains the following regions
\bq \label{1rb1}
 \rho(\hat{\AA}_3(\xi))\supset
 \left\{\bln &\{\lambda\in\mathbb{C}\,|\,
     \mathrm{Re}\lambda\ge-\mu +\delta_1,\,|\mathrm{Im}\lambda|\geq y_1\}
 \cup\{\lambda\in\mathbb{C}\,|\,\mathrm{Re}\lambda>0\},\quad |\xi|\le r_1,\\
 &\{\lambda\in\mathbb{C}\,|\,
     \mathrm{Re}\lambda\ge-\mu +\delta_1,\,\min_{j=\pm1}|\lambda-j\i|\xi||\ge \delta_2\}
 \cup\{\lambda\in\mathbb{C}\,|\,\mathrm{Re}\lambda>0\},\quad |\xi|\ge r_1;
 \eln\right.
\eq
\item  it holds for $0<|\xi|\leq r_2$ that the spectrum set of $\hat{\AA}_3(\xi)$
is located in the following domain
 \bq
 \sigma(\hat{\AA}_3(\xi))\cap\{\lambda\in\mathbb{C}\,|\,\mathrm{Re}\lambda\ge-\mu+\delta_1\}
 \subset
 \dcup_{j=-1}^7\{\lambda\in\mathbb{C}\,|\,|\lambda-\alpha_j(\xi)|\le\delta_2\},\label{1sg4a}
 \eq
where $\alpha_j(\xi)$, $-1\le j\le 7,$ are the eigenvalues of $G_6(\xi)$
defined in \eqref{eigen}.
\end{enumerate}
\end{lem}

\subsection{Asymptotics in low frequency}
We study the low frequency asymptotics of the eigenvalues and eigenvectors of the operator $\hat{\AA}_3(\xi)$ in this subsection. In terms of \eqref{B(xi)}, the eigenvalue problem $\hat{\AA}_3(\xi)U=\lambda  U$ for $U=(f,X,Y)\in L^{2}_\xi(\R^3_v)\times \mathbb{C}^3_\xi\times \mathbb{C}^3_\xi$ can be written as
 \bq
\left\{\bln     \label{1L_2}
  \lambda f
  &=(L-\i(v\cdot\xi)
   -\mbox{$\frac{\i(v\cdot\xi)}{|\xi|^2}$}P_{\rm d})f-v\sqrt M\cdot(\omega\times X),\\
  \lambda X&=-\omega\times (f,v\sqrt M)+\i\xi\times Y,\\
  \lambda Y&=-\i\xi\times X,\quad |\xi|\ne0.
  \eln\right.
 \eq
By  macro-micro decomposition, the eigenfunction $f$ of \eqref{1L_2} can be decomposed into
$
f=f_0+f_1=: \P_0f+ \P_1f.
$
Then the first equation of \eqref{1L_2} gives rise to
 \bma
 &\lambda f_0=- \P_0[\i(v\cdot\xi)(f_0+f_1)]
              -\frac{\i(v\cdot\xi)}{|\xi|^2}P_{\rm d}f_0-v\sqrt M\cdot(\omega\times X),\label{1A_2}
\\
&\lambda f_1=Lf_1- \P_1[\i(v\cdot\xi)(f_0+f_1)]\  \Longrightarrow \
  f_1= -\i(\lambda  \P_1-B_5(\xi) )^{-1} \P_1(v\cdot\xi)f_0.\label{1A_3}
 \ema
Substituting \eqref{1A_3} into \eqref{1A_2}, we obtain the eigenvalue problem  for  $(f_0,X,Y)$ as
 \bq
 \left\{\bln    \label{1A_5}
 \lambda f_0
 &=-\i \P_0(v\cdot\xi)f_0
  -\mbox{$\frac{\i(v\cdot\xi)}{|\xi|^2}$}P_{\rm d}f_0
  + \P_0[(v\cdot\xi)R^1(\lambda,\xi) \P_1(v\cdot\xi)f_0]-v\sqrt M\cdot(\omega\times X),\\
  \lambda X&=-\omega\times (f_0,v\sqrt M)+\i\xi\times Y,\\
  \lambda Y&=-\i\xi\times X,
 \eln\right.
 \eq
where
 $
 R^1(\lambda,\xi)=-(\lambda  \P_1-B_5(\xi) )^{-1}=[L-\lambda  \P_1-\i \P_1(v\cdot\xi) \P_1]^{-1}.
 $

To solve  the eigenvalue problem \eqref{1A_5}, we write $f_0\in N_0$  as
$ f_0=\sum_{j=0}^4W_j\chi_j$. Taking the inner product
of the first equation of \eqref{1A_5} and $\chi_j$ for $j=0,1,2,3,4$ respectively,
we have the equations about $\lambda$ and $(W_0,W,W_4,X,Y)$ with $W=:(W_1,W_2,W_3)$ and $X=(X_1,X_2,X_3),Y=(Y_1,Y_2,Y_3)\in \mathbb{C}^3_\xi$ for $\text{Re}\lambda>-\mu$:
 \bma
 \lambda W_0&=-\i(W\cdot\xi),  \label{1A_6}
 \\
  \lambda W_i
 & =-\i W_0\(\xi_i+\frac{\xi_i}{|\xi|^2}\)
    -\i\sqrt{\frac23}W_4\xi_i
    +\sum^3_{j=1}W_j(R^1(\lambda,\xi) \P_1(v\cdot\xi)\chi_j,(v\cdot\xi)\chi_i)
 \nnm\\
&\quad+W_4(R^1(\lambda,\xi) \P_1(v\cdot\xi)\chi_4, (v\cdot\xi)\chi_i)-(\omega\times X)_i,\label{1A_7}
 \\
 \lambda W_4
 &=-\i\sqrt{\frac23}(W\cdot\xi)
   +\sum^3_{j=1}W_j(R^1(\lambda,\xi) \P_1(v\cdot\xi)\chi_j,(v\cdot\xi)\chi_4)
   \nnm\\
   &\quad+W_4(R^1(\lambda,\xi) \P_1(v\cdot\xi)\chi_4,(v\cdot\xi)\chi_4),  \label{1A_8}\\
  \lambda X&=-\omega\times W+\i\xi\times Y, \label{A_12b}\\
  \lambda Y&=-\i\xi\times X. \label{1A_13b}
 \ema

We apply the following transform to simplify the system \eqref{1A_6}-\eqref{1A_8}.
\begin{lem}[\cite{Li2}]\label{eigen_0}
Let  $e_1=(1,0,0)$, $\xi=s\omega$ with $s\in \R$, $\omega=(\omega_1,\omega_2,\omega_3)\in \S^2$. Then, it holds for $1\le i,j\le 3$ and $\text{Re}\lambda>-\mu$ that
\bma
 (R^1(\lambda,\xi) \P_1(v\cdot\xi)\chi_j,(v\cdot\xi)\chi_i)
 =&s^2(\delta_{ij}-\omega_i\omega_j)(R^1(\lambda,se_1) \P_1(v_1\chi_2),v_1\chi_2)
\nnm\\
 & +s^2\omega_i\omega_j(R^1(\lambda,se_1) \P_1(v_1\chi_1),v_1\chi_1),\label{1T_1}
\\
 (R^1(\lambda,\xi) \P_1(v\cdot\xi)\chi_4,(v\cdot\xi)\chi_i)
=&s^2\omega_i(R^1(\lambda,se_1) \P_1(v_1\chi_4),v_1\chi_1),\label{1T_2}
\\
 (R^1(\lambda,\xi) \P_1(v\cdot\xi)\chi_i,(v\cdot\xi)\chi_4)
=&s^2\omega_i(R^1(\lambda,se_1) \P_1(v_1\chi_1),v_1\chi_4),\label{1T_3}
\\
 (R^1(\lambda,\xi) \P_1(v\cdot\xi)\chi_4,(v\cdot\xi)\chi_4)
=&s^2(R^1(\lambda,se_1) \P_1(v_1\chi_4),v_1\chi_4).\label{1T_4}
\ema
\end{lem}

With the help of \eqref{1T_1}--\eqref{1T_4}, the equations
\eqref{1A_6}--\eqref{1A_8} can be simplified as
 \bma
 \lambda W_0&=-\i s(W\cdot\omega),\label{1A_9}
 \\
 \lambda W_i
 &=-\i W_0\(s+\frac1s\)\omega_i
   -\i s\sqrt{\frac23}W_4\omega_i
   +s^2(W\cdot\omega)\omega_i R_{11}
   \nnm\\
&\quad+s^2(W_i-(W\cdot\omega)\omega_i)R_{22}
       +s^2W_4\omega_i R_{41}
-(\omega\times X)_i,
      \quad i=1,2,3,    \label{1A_10}
 \\
 \lambda W_4
 &=-\i s\sqrt{\frac23}(W\cdot\omega)
    +s^2(W\cdot\omega)R_{14}
   +s^2W_4 R_{44},
   \label{1A_11}
  \ema
 where
 \bq R_{ij}=R_{ij}(\lambda,s)=:(R^1(\lambda,se_1) \P_1(v_1\chi_i),v_1\chi_j).\label{rij}\eq
Multiplying \eqref{1A_10} by $\omega_i$ and making the summation of resulted equations with respect to $i=1,2,3$, we have
 \bma
\lambda (W\cdot\omega)&=-\i W_0\(s+\frac1
s\)-\i s\sqrt{\frac23}W_4+s^2(W\cdot\omega)R_{11}
+s^2 W_4R_{41}.
\label{1A_12}
 \ema
Denote by $U=(W_0,W\cdot\omega,W_4)$ a vector 
in $\R^3$. The system \eqref{1A_9}, \eqref{1A_11} and \eqref{1A_12} can be written as $\mathbb{M}U=0$ with the matrix $ \mathbb{M}$ defined by
 \bq
 \mathbb{M}=\left(\ba
  \lambda & \i s & 0\\
   \i(s+\frac1s)
  &\lambda-s^2R_{11} 
  &\i s\sqrt{\frac23}-s^2R_{41}
  \\
    0
  &\i s\sqrt{\frac23}-s^2R_{14}
  &\lambda-s^2 R_{44}
  \ea\right).\label{BM}
 \eq
The equation  $\mathbb{M}U=0$ admits a non-trivial solution $U\neq 0$ for $\text{Re}\lambda>-\mu$ if and only if it holds $\rm{det}(\mathbb{M})=0$ for $\text{Re}\lambda>-\mu$.

Furthermore, by taking $\omega\times$ to \eqref{1A_10} and using \eqref{rotat}, we have
 \bma
 (\lambda-s^2R_{22})(\omega\times W)=X.  \label{1A_12a}
 \ema
Multiplying \eqref{A_12b} by $\lambda-s^2R_{22}$, we obtain
$$ \lambda(\lambda-s^2R_{22})X+X-\i s(\lambda-s^2R_{22})(\omega\times Y)=0.$$
Then we multiply above equation by $\lambda$ and using \eqref{1A_13b} and \eqref{rotat} to have
\bq (\lambda^3-s^2R_{22}\lambda^2+(1+s^2)\lambda-s^4R_{22})X=0.\eq
Denote
 \bq
 D(\lambda,s)=\det(\mathbb{M}), \quad  D_0(\lambda,s)=:\lambda^3-s^2R_{22}\lambda^2+(1+s^2)\lambda-s^4R_{22}.   \label{D0a}
 \eq

The following two lemmas are about the solutions to the equations
$D(\lambda, s)=0$ and $D_0(\lambda, s)=0$.

\begin{lem}[\cite{Li2}]
\label{eigen_2}
There exist two small constants $r_0>0$ and $r_1>0$ so that the equation $D(\lambda,s)=0$ admits  $C^\infty$ solution $\lambda_j(s)$ $(j=-1,0,1)$ for $(s,\lambda_j)\in
[-r_0,r_0]\times B_{r_1}(j\i)$ that satisfy
 \bma
 \lambda_j(0)=&j\i,\quad \lambda'_j(0)=0, \label{T_5b} \\
 \lambda''_{\pm1}(0)
 =&(L(L+\i \P_1)^{-1} \P_1(v_1 \chi_1),(L+\i \P_1)^{-1} \P_1(v_1\chi_1))
\nnm\\
& \pm \i(\|(L+\i \P_1)^{-1} \P_1(v_1\chi_1)\|^2+\frac53),\label{T_5a}
\\
\lambda''_{0}(0)=&2(L^{-1} \P_1(v_1\chi_4),v_1\chi_4).\label{T_6a}
 \ema
Moreover, $\lambda_j(s)$ is an even function and satisfies
 \bq
 \overline{\lambda_j(s)}=\lambda_{-j}(-s)
 =\lambda_{-j}(s)\quad \text{for}\quad j=0,\pm 1.\label{L_8a}
 \eq
In particular, $\lambda_0(s)$ is a real function.
\end{lem}

\begin{lem}
\label{eigen_1}
There exist two small constants $r_0>0$ and $r_1>0$ so that the equation $D_0(\lambda,s)=0$ admits $C^\infty$ solution $\lambda_j(s)$ $(j=-1,0,1)$ for $(s,\lambda_j)\in
[-r_0,r_0]\times B_{r_1}(j\i)$ that satisfy
 \bma
 \lambda_j(0)=&j\i,\quad \lambda'_j(0)=0, \quad \lambda_0''(0)=\lambda_0'''(0)=0,\label{1T_5a} \\
 \lambda''_{\pm1}(0)
 =&(L(L+\i \P_1)^{-1} \P_1(v_1 \chi_2),(L+\i \P_1)^{-1} \P_1(v_1\chi_2))
\nnm\\
& \pm \i(\|(L+\i \P_1)^{-1} \P_1(v_1\chi_2)\|^2+1),\label{1T_5}
\\
\lambda^{(4)}_{0}(0)=&24(L^{-1} \P_1(v_1\chi_2),v_1\chi_2).\label{1T_6}
 \ema
Moreover, $\lambda_j(s)$ is an even function and satisfies
 \bq
 \overline{\lambda_j(s)}=\lambda_{-j}(-s)
 =\lambda_{-j}(s)\quad \text{for}\quad j=0,\pm 1.\label{1L_8}
 \eq
In particular, $\lambda_0(s)$ is a real function.
\end{lem}
\begin{proof}
Since $D_0(\lambda,s)=0$ has three roots of the form $(\lambda,s)=(j\i,0)$ for $j=-1,0,1$, with $\lambda=j\i$ being the solution to $D_0(\lambda,0)=\lambda(\lambda^2+1)$, and
  \bq
  {\partial_s}D_0(j\i,0)=0,     
 \quad
  {\partial_\lambda}D_0(j\i,0)=1-3j^2\neq0,  \label{1lamd1}
 \eq
the implicit function theorem implies that there exists
small constants $r_0,r_1>0$ and a unique $C^\infty$ function $\lambda_j(s)$: $[-r_0,r_0]\to B_{r_1}(j\i)$ such that $D_0(\lambda_j(s),s)=0$ for $s\in [-r_0,r_0]$, and in particular
$$
 \lambda_j(0)=j \i\quad {\rm and}\quad
 \lambda_j'(0)=-\mbox{$\frac{{\partial_s}D_0(j\i,0)}{{\partial_\lambda}D_0(j\i,0)}$}=0,
 \quad j=0,\pm1.                 
$$
Direct computation gives
\bmas
&{\partial^2_s}D_0(j\i,0) =2j^2((L-j\i \P_1)^{-1} \P_1(v_1\chi_2),v_1\chi_2)+2j\i,\\
&{\partial^2_s}D_0(0,0)={\partial^3_s}D_0(0,0)=0, \quad  {\partial^4_s}D_0(0,0)=24((L-j\i \P_1)^{-1} \P_1(v_1\chi_2),v_1\chi_2),
\emas
which together with \eqref{1lamd1}  yields
 \bq
 \left\{\begin{aligned}
  \lambda_0''(0)
  &= \lambda_0'''(0)=0,\quad \lambda_0^{(4)}(0)=-\mbox{$\frac{\partial_s^4D_0(0,0)}{\partial_\lambda D_0(0,0)}$}
   =24(L^{-1} \P_1(v_1\chi_2),v_1\chi_2),
   \\
 \lambda_{\pm1}''(0)
 &=-\mbox{$\frac{\partial_s^2 D_0(\pm \i,0)}{\partial_\lambda D_0(\pm \i,0)}$}
  =((L\mp \i \P_1)^{-1} \P_1(v_1\chi_2),v_1\chi_2)   \pm  \i.
 \end{aligned}\right.                     \label{lamd3a}
 \eq

Finally, since
$D_0(\lambda,s)=D_0(\lambda,-s)$,
$\overline{D_0(\lambda,s)}=D_0(\overline{\lambda},-s)$, we can obtain
\eqref{1L_8} by using the fact that $\lambda_{\pm1}(s)={\pm}\i+O(s^2)$ and $\lambda_0(s)=O(s^4)$ as
$s\rightarrow0.$
\end{proof}

With the help of Lemmas \ref{eigen_2}--\ref{eigen_1}, we are able to
construct the eigenvector $\Psi_j(s,\omega)$ corresponding to the
eigenvalue $\lambda_j$ at the low frequency.
Indeed, we have
\begin{thm}\label{eigen_3a}
There exists a constant $r_0>0$ so that the spectrum $\lambda\in\sigma(\hat{\AA}_3(\xi))\subset\mathbb{C}$ for $\xi=s\omega$ with $|s|\leq r_0$ and $\omega\in \mathbb{S}^2$ consists of nine points $\{\lambda_j(s),\ -1\le j\le 7\}$ in the domain $\mathrm{Re}\lambda>-\mu /2$. The spectrum $\lambda_j(s)$ and the corresponding eigenvector $ \Psi_j(s,\omega)=(\psi_j,X_j,Y_j)(s,\omega)$ are $C^\infty$ functions of $s$ for $|s|\leq r_0$. In particular, the eigenvalues admit the following asymptotic expansion for $|s|\leq r_0$
 \be                                   \label{1specr0}
 \left\{\bln
 \lambda_{\pm1}(s)=&\pm \i+(-a_1\pm\i b_1)s^2+o(s^2),\quad  \overline{\lambda_1}=\lambda_{-1},\\
 \lambda_{0}(s) =& -a_0s^2+o(s^2),\\
 \lambda_{2}(s) =& \lambda_{3}(s) =-\i+(-a_2-\i b_2)s^2+o(s^2),\quad \overline{\lambda_{2}}=\lambda_{4},\\
 \lambda_{4}(s) =& \lambda_{5}(s) =\i+(-a_2+\i b_2)s^2+o(s^2),\\
  \lambda_{6}(s) =& \lambda_{7}(s) =-a_3s^4+o(s^4),
 \eln\right.
 \ee
where $a_j>0$ $(0\le j\le3)$ and $b_j>0$ $(1\le j\le2)$ are given by
\bq\left\{\bln
a_0&=-(L^{-1} \P_1(v_1\chi_4),v_1\chi_4),\quad a_3=-(L^{-1} \P_1(v_1\chi_2),v_1\chi_2),\\
a_1&=-\mbox{$\frac12$}(L(L+\i \P_1)^{-1} \P_1(v_1 \chi_1),(L+\i \P_1)^{-1} \P_1(v_1\chi_1)),\\
a_2&=-\mbox{$\frac12$}(L(L+\i \P_1)^{-1} \P_1(v_1 \chi_2),(L+\i \P_1)^{-1} \P_1(v_1\chi_2)),\\
b_1&=\mbox{$\frac12$}(\|(L+\i \P_1)^{-1} \P_1(v_1\chi_1)\|^2+\mbox{$\frac53$}),\quad
b_2= \mbox{$\frac12$}(\|(L+\i \P_1)^{-1} \P_1(v_1\chi_2)\|^2+1).
\eln\right.
\eq
The eigenvectors $  \Psi_j=(\psi_j,X_j,Y_j)$ are   orthogonal to each other and satisfy
 \be
 \left\{\bln
 &( \Psi_i(s,\omega), \Psi^*_j(s,\omega))_\xi=(\psi_i,\overline{\psi_j})_\xi-(X_i,\overline{X_j})-(Y_i,\overline{Y_j})=\delta_{ij},
  \quad  -1\le i,j\le7,                                  \label{1eigfr0}
 \\
&(\psi_j,X_j,Y_j)(s,\omega)
 =\mbox{$\sum_{n=0}^2$} (\psi_{j,n},X_{j,n},Y_{j,n})(\omega)s^n +o(s^2), \quad |s|\leq r_0,
 \eln\right.
 \ee
where $   \Psi^*_j=(\overline{\psi_j},-\overline{X_j},-\overline{Y_j})$, and the coefficients $(\psi_{j,n},X_{j,n},Y_{j,n})$  are given by
 \bq
  \left\{\bln                      \label{1eigf1}
 &\psi_{0,0}=\chi_4,\quad
  \psi_{0,1}=\i L^{-1} \P_1(v\cdot\omega)\chi_4,\quad
  (\psi_{0,2},\sqrt{M})=-\mbox{$\sqrt{\frac23}$},\quad X_0=Y_0\equiv0,
 \\
&\psi_{\pm1,0}=\mbox{$\frac{\sqrt2}2$}(v\cdot\omega)\sqrt{M},\quad
 (\psi_{\pm1,2},\sqrt{M})=0,\quad X_{\pm1}=Y_{\pm1}\equiv0,
 \\
 &\psi_{\pm1,1}=\mp\mbox{$\frac{\sqrt2}2$}\sqrt{M}
  \mp \mbox{$\frac{\sqrt{3}}{3}$}\chi_4+\mbox{$\frac{\sqrt2}2$}\i(L\mp \i \P_1)^{-1} \P_1(v\cdot\omega)^2\sqrt{M},
 \\
&\psi_{j,0}=\mbox{$\frac{\sqrt2}2$}(v\cdot W^j)\sqrt{M},\quad (\psi_{j,n},\sqrt M)=(\psi_{j,n},\chi_4)=0\,\,\, (n\geq0),\\
  & \P_1\psi_{j,1}=\i \mbox{$\frac{\sqrt2}2$}L^{-1} \P_1[(v\cdot\omega)(v\cdot W^j)\sqrt{M}],\quad j=2,3,4,5,
  \\
  &X_{j,0}=-\i \mbox{$\frac{\sqrt2}2$}\omega\times W^j,\,\, Y_{j,0}=0,\,\, j=2,3,\quad X_{j,0}=\i \mbox{$\frac{\sqrt2}2$}\omega\times W^j,\,\,  Y_{j,0}=0,\,\, j=4,5,
  \\
  &\psi_{j,0}=0,\quad (\psi_{j,n},\sqrt M)=(\psi_{j,n},\chi_4)=0\,\,\, (n\geq0),\\
  & \psi_{j,1}= (v\cdot W^j)\sqrt{M},\quad X_{j,0}=X_{j,1}=X_{j,2}=0,\quad Y_{j,0}=\i W^j,\quad j=6,7,
  \eln\right.
  \eq
Here, $W^j$ $(j=2,3,4,5,6,7)$ are normal vectors satisfying
$W^j\cdot\omega=0,\ W^1\cdot W^2=0$, $W^2=W^4=W^6,W^3=W^5=W^7$.
\end{thm}

\begin{proof}
The eigenvalues $\lambda_j(s)$ and the eigenvectors $\Psi_j(s,\omega)=(\psi_j,X_j,Y_j)(s,\omega)$, $-1\le j\le 7$, can be constructed as follows. For $j=2,3,4,5,6,7$, we take $\lambda_2=\lambda_3=\lambda_{-1}(s),\,\lambda_4=\lambda_5=\lambda_1(s),\,\lambda_6=\lambda_7=\lambda_0(s)$ to be the solution of the equation  $D_0(\lambda,s)=0$ defined in Lemma \ref{eigen_1}, and
choose $W_0=W_4=0$, and $W=W^j$ to be the linearly independent vector
so that  $W^j\cdot\omega=0$, $W^2\cdot W^3=W^4\cdot W^5=W^6\cdot W^7=0$ and $W^2=W^4=W^6,W^3=W^5=W^7$. And the corresponding eigenvectors $ \Psi_j(s,\omega)=(\psi_j,X_j,Y_j)(s,\omega)$ are defined by
 \bq          \nnm 
 \left\{\bln
  &\psi_j(s,\omega)
 =(W^j\cdot v)\sqrt{M}
   +\i s[L-\lambda_j \P_1-\i s \P_1(v\cdot\omega) \P_1]^{-1}
         \P_1[(v\cdot\omega)(W^j\cdot v)\sqrt{M}],\\
  &X_j(s,\omega)=(\lambda_j-s^2R_{22}(\lambda_j,s))(\omega\times W^j ),\quad Y_j(s,\omega)= \frac{\i s}{\lambda_j}(\lambda_j-s^2R_{22}(\lambda_j,s))W^j,
  \eln\right.
\eq
which satisfy the orthonormal relation: $(\Psi_2,\Psi_3^*)_\xi=(\Psi_4,\Psi_5^*)_\xi=(\Psi_6,\Psi_7^*)_\xi=0.$

For $j=-1,0,1$, we choose
$\lambda_j=\lambda_j(s)$ to be a solution of $D(\lambda,s)=0$ given by Lemma \ref{eigen_2}, and choose $X=Y=0$, and denote by $\{a_j,b_j,d_j\}=:\{W^j_0,\, (W\cdot\omega)^j,\, W^j_4\}$  a solution of system \eqref{1A_9}, \eqref{1A_11}, and \eqref{1A_12} for $\lambda=\lambda_j(s)$. Then we define
$  \Psi_j(s,\omega)=(\psi_j(s,\omega),0,0)$ ($j=-1,0,1$) where
 \bq
 \left\{\bln \psi_j(s,\omega)&= \P_0\psi_j(s,\omega)+ \P_1\psi_j(s,\omega),\\
 \P_0\psi_j(s,\omega)&=a_j(s)\chi_0+b_j(s)(v\cdot\omega)\sqrt{M}+
d_j(s)\chi_4,\\
 \P_1\psi_j(s,\omega)&=\i s[L-\lambda_j
 \P_1-\i s \P_1(v\cdot\omega) \P_1]^{-1} \P_1[(v\cdot\omega) \P_0\psi_j(s,\omega)].
\eln\right.     \nnm 
\eq
Then following the similar argument as Theorem \ref{eigen_3}, we can obtain the expansion of $\Psi_j(s,\omega)$ in \eqref{1eigfr0} and \eqref{1eigf1}. Hence, we omit the detail for brevity.
\end{proof}

\subsection{Asymptotics in high frequency }

The structure of the spectrum for the
one-species VMB in high frequency region is similar to the two-species case
so that we only sketch the key points here.
Recalling
the eigenvalue problem
 \bmas
  \lambda f
  &=B_2(\xi)f-v\sqrt M\cdot(\omega\times X),
  \\
  \lambda X&=-\omega\times (f,v\sqrt M)+\i\xi\times Y,
  \\
  \lambda Y&=-\i\xi\times X,\quad |\xi|\ne0.
\emas
Similar to two species case,
 we obtain
\bq (\lambda^2-((B_2(|\xi|e_1)-\lambda)^{-1}\chi_2,\chi_2)\lambda+|\xi|^2)X=0, \quad |\xi|>R_0.\eq
Denote
\bq D(\lambda,s)=\lambda^2-((B_2(se_1)-\lambda)^{-1}\chi_2,\chi_2)\lambda+s^2,\quad s>R_0.\label{D3}\eq

As Section 2.3, we can obtain
\begin{thm}\label{eigen_4a}
There exists a constant $r_1>0$ so that the spectrum $\sigma(A_2(\xi))\subset\mathbb{C}$ for $\xi=s\omega$ with $s=|\xi|> r_1$ and $\omega\in \mathbb{S}^2$ consists of four points $\{\beta_j(s),\ j=1,2,3,4\}$ in the domain $\mathrm{Re}\beta>-\mu/2$. In particular, the eigenvalues satisfy
 \bgr
 \beta_1(s) = \beta_2(s) =-\i s+O(s^{-1/2}),\label{1specr1}\\
 \beta_3(s) = \beta_4(s) =\i s+O(s^{-1/2}),\label{1specr2}\\
c_1\frac1s\le -{\rm Re}\beta_{j}(s)\le c_2\frac1s.
\egr
The eigenvectors $ \Phi_j=(\phi_j,X_j,Y_j)$ are   orthogonal to each other and satisfy
 \be
 ( \Phi_i(s,\omega), \Phi^*_j(s,\omega))=(\psi_i,\overline{\psi_j})-(X_i,\overline{X_j})-(Y_i,\overline{Y_j})=\delta_{ij},
  \quad  1\le i\ne j\le 4,                                  \label{1eigfr2}
 \ee
where $  \Phi^*_j=(\overline{\phi_j},-\overline{X_j},-\overline{Y_j})$, and
 \be\label{1eigfr3}
 \left\{\bln
 &\|\phi_j(s,\omega)\|=O(\mbox{$\frac1{\sqrt s}$}),\quad (\phi_j(s,\omega),\chi_0)=(\phi_j(s,\omega),\chi_4)=0, \\
 &X_j(s,\omega)=O(1)\i (\omega\times W^j),\quad Y_j(s,\omega)=O(1)\i  W^j.
 \eln\right.
  \ee
 Here, $W^j$ $(j=1,2,3,4)$ are normal vectors satisfying
$W^j\cdot\omega=0,\ W^1\cdot W^2=0$,  $W^1=W^3, W^2=W^4$.
\end{thm}

\section{The linearized system}
\label{behavior-linear}

In this section, we consider the Cauchy problems \eqref{LVMB1a} and \eqref{LVMB} for the linearized Vlasov-Maxwell-Boltzmann equations in two and one-species
 and give the optimal time decay rates of
the solution based on  on spectrum structures obtained in the previous
sections.

\subsection{Semigroup for two-species}
\setcounter{equation}{0}

Before giving the theorem on the semigroup, we firstly prepare some
lemmas on its properties.

\begin{lem}[\cite{Li3}]
 \label{SG_2}
The operator $Q(\xi)=L_1-\i  P_r(v\cdot\xi) P_r$ generates a strongly continuous contraction
semigroup on $N_1^\bot$, which satisfies for any $t>0$ and $f\in N_1^\bot\cap L^2(\R^3_v)$ that
  \bq
    \|e^{tQ(\xi)}f\|\leq e^{-\mu t}\|f\|. \label{decay_1}
 \eq
In addition, for any $x>-\mu $ and $f\in N_1^\bot\cap L^2(\R^3_v)$ it holds
\bq
 \int^{+\infty}_{-\infty}\|[(x+\i y) P_r-Q(\xi)]^{-1}f\|^2dy
 \leq
        \pi(x+\mu )^{-1}\|f\|^2.\label{S_4a}
\eq
\end{lem}

\begin{lem}
 \label{SG_3}
The operator $B_3(\xi)$ generates a strongly continuous unitary
semigroup on $\mathbb{C}^6$, which satisfies for any $t>0$ and $U\in\mathbb{C}^6$ that
  \bq
    |e^{tB_3(\xi)}U|=|U|. \label{decay_2}
 \eq
In addition, for any $x\ne0$ and $U\in\mathbb{C}^6$ it holds
\bq
 \int^{+\infty}_{-\infty}|[(x+\i y) -B_3(\xi)]^{-1}U|^2dy
 \leq
        \pi|x|^{-1}|U|^2.\label{S_4b}
\eq
\end{lem}
\begin{proof}
Since $\i B_3(\xi)$ is a self-adjoint operator on $\mathbb{C}^6$ satisfying \eqref{b_1a}, we can prove \eqref{decay_2} and \eqref{S_4b} by applying a similar argument as the one for  Lemma 3.2 in \cite{Li2}.
\end{proof}

\begin{lem}\label{F_1}Let $r_0>0$ and $b_2>0$ be given in Lemma \ref{eigen_3}. Let $\alpha=\alpha(r_0,r_1)>0$ with $r_1>r_0$ and $\alpha(r_0,r_1)$ defined in Lemma \ref{LP01}. Then
\bma
\sup_{0<|\xi|<r_0,y\in \R}\|(I-G_4(\xi)(-\frac{b_2}2+\i y-G_3(\xi))^{-1})^{-1}\|_\xi&\le C,\label{S_9}\\
\sup_{r_0<|\xi|<r_1,y\in \R}\|(I-G_2(\xi)(-\frac{\alpha}2+\i y-G_1(\xi))^{-1})^{-1}\|&\le C.\label{S_9a}
\ema
\end{lem}
\begin{proof}
Let $\lambda=x+\i y$. We prove \eqref{S_9} first.
By \eqref{S_2} and \eqref{S_5}, there exists $R>0$ large enough  such that if ${\rm Re}\lambda\ge -\mu/2$, $|{\rm Im}\lambda|\ge R$ and $|\xi|\le r_0$, then \eqref{bound_1} holds. This yields
$$\|(I-G_4(\xi)(\lambda-G_3(\xi))^{-1})^{-1}\|_\xi\le 2.$$
Thus, it remains to prove \eqref{S_9} for $|y|\le R$. We will prove it by contradiction. Indeed, if \eqref{S_9} does not hold for $|y|\le R$, namely, there are subsequences $\{\xi_n\}$, $\{\lambda_n=b_2/2+\i y_n\}$ with $|\xi_n|\le r_0$, $|y_n|\le R$, and $U_n=(f_n,E^1_n,B^1_n),V_n=(g_n,E^2_n,B^2_n)$ with $\|U_n\|_{\xi_n}\to 0$ $(n\to \infty)$, $\|V_n\|_{\xi_n}=1$ such that
$$(I-G_4(\xi)(\lambda-G_3(\xi))^{-1})^{-1}U_n=V_n.$$
Let
$$(a_n,b_n)^T=(\lambda_n -B_1(\xi_n))^{-1}(E^2_n,B^2_n)^T \Longleftrightarrow (E^2_n,B^2_n)^T=\lambda_n (a_n,b_n)^T-(\i\xi_n\times b_n,-\i\xi_n\times a_n)^T.$$
Then
\bma
P_{\rm d}f_n=&P_{\rm d}g_n+\i P_{\rm d}(v\cdot\xi_n)P_r(\lambda_n P_r-Q(\xi_n))^{-1}P_rg_n,\label{m_2}\\
P_r f_n=&P_r g_n+\i \lambda^{-1}_n P_r(v\cdot\xi_n)(1+\frac1{|\xi_n|^2})P_{\rm d}g_n-v\sqrt{M}\cdot(\omega_n\times a_n),\label{m}\\
E^1_n=&\lambda_n a_n-\i\xi_n\times b_n-\omega_n\times ((\lambda_n P_r-Q(\xi_n))^{-1}P_rg_n,v\sqrt M),\label{m_5}\\
B^1_n=&\lambda_n b_n+\i\xi_n\times a_n.\label{m_6}
\ema
Substituting \eqref{m} into \eqref{m_2} and \eqref{m_5}, we obtain
\bma
P_{\rm d}f_n=&P_{\rm d}g_n+\i P_{\rm d}(v\cdot\xi_n)P_r(\lambda_n P_r-Q(\xi_n))^{-1}P_rf_n \nnm\\
&+\lambda_n^{-1}P_{\rm d} (v\cdot\xi_n)P_r(\lambda_n P_r-Q(\xi_n))^{-1} P_r(v\cdot\xi_n)(1+\frac1{|\xi_n|^2})P_{\rm d}g_n,\label{m_4}\\
E^1_n=&\lambda_n a_n-\i\xi_n\times b_n-\omega_n\times ((\lambda_n P_r-Q(\xi_n))^{-1}P_rf_n,v\sqrt M)\nnm\\
&-\omega_n\times ((\lambda_n P_r-Q(\xi_n))^{-1}v\sqrt{M}\cdot(\omega_n\times a_n),v\sqrt M).\label{m_7}
\ema
Since $\|f_n\|_{\xi_n}+|E^1_n|+|B^1_n|\to 0$ as $n\to \infty$, it follows from \eqref{m_4}, \eqref{m_7} and \eqref{m_6} that
\bmas
&\lim_{n \to \infty}\sqrt{1+\frac1{|\xi_n|^2}}\frac{|C_n|}{|\lambda_n|}|\lambda_n+(|\xi_n|^2+1)((\lambda_n P_r-Q(|\xi_n|e_1))^{-1}\chi_1,\chi_1)|\to0,\\
&\lim_{n \to \infty}|\lambda_n a_n-\i\xi_n\times b_n- ((\lambda_n P_r-Q(|\xi_n|e_1))^{-1}\chi_2,\chi_2)a_n|= 0,\\
&\lim_{n \to \infty}|\lambda_n b_n+\i\xi_n\times a_n|= 0,
\emas
where $C_n=(g_n,\sqrt M)$.
Since $\sqrt{1+\frac1{|\xi_n|^2}}|C_n|\le 1,|\xi_n|\le r_0, |y_n|\le R$ and $|(a_n.b_n)|\le 2b_2^{-1}|(E^2_n,B^2_n)|\le 2b_2^{-1}$, there is a subsequence $\{(\xi_{n_j},\lambda_{n_j},C_{n_j})\}$ such that $\sqrt{1+\frac1{|\xi_{n_j}|^2}}C_{n_j}\to A_0,a_{n_j}\to a_0,b_{n_j}\to b_0,\xi_{n_j}\to\xi_0,\lambda_{n_j}\to \lambda_0=b_2/2+\i y\ne0$.
Thus
\bma
&\frac{|A_0|}{|\lambda_0|}|\lambda_0+(|\xi_0|^2+1)((\lambda_0 P_r-Q(|\xi_0|e_1))^{-1}\chi_1,\chi_1)|=0,\label{m_3}\\
&\lambda_0 a_0-\i\xi_0\times b_0+ ((\lambda_0 P_r-Q(|\xi_0|e_1))^{-1}\chi_2,\chi_2)a_0= 0,\label{m_3a}\\
&\lambda_0 b_0+\i\xi_0\times a_0= 0.\label{m_3b}
\ema

It is straightforward to verify  that  $(A_0,a_0,b_0)\ne0$. Indeed, otherwise,  we have $\lim\limits_{j\to\infty}(\mbox{$\sqrt{1+\frac1{|\xi_{n_j}|^2}}$}C_{n_j},a_{n_j},b_{n_j})=0.$
This and \eqref{m} lead to $\lim\limits_{j\to\infty}\|P_r g_{n_j}\|=0$ and $\lim\limits_{j\to\infty}(E^2_{n_j},B^2_{n_j})=0.$
 Thus $\lim\limits_{j\to\infty}\|V_{n_j}\|_{\xi_{n_j}}=0$, which  contradicts to $\|V_{n}\|_{\xi_n}=1$. Therefore,  \eqref{m_3}, \eqref{m_3a} and \eqref{m_3b}  imply that $\lambda_0$ is an eigenvalue of $\hat{B}(\xi_0)$ with ${\rm Re}\lambda_0=-b_2/2$ and $|\xi_0|\le r_0$, which contradicts to Theorem \ref{eigen_3} since we can assume ${\rm Re}\lambda_j(s)\ne -b_2/2$ by taking a smaller $r_0$ if necessary.

By an  argument similar to the one
for Lemma \ref{LP01}, we can prove \eqref{S_9a}.  And this completes the proof of the lemma.
\end{proof}

With the help of Lemmas~\ref{LP03}--\ref{spectrum3} and Lemmas~\ref{SG_2}--\ref{F_1}, we have a decomposition of the semigroup $S(t,\xi)=e^{t\hat{\AA}_1(\xi)}$.

\begin{thm}\label{semigroup1}
The semigroup $S(t,\xi)=e^{t\hat{\AA}_1(\xi)}$ with $\xi=s\omega\in \R^3$ and $s=|\xi|\neq0$  has the following decomposition
 \be
 S(t,\xi)U=S_1(t,\xi)U+S_2(t,\xi)U+S_3(t,\xi)U,
     \quad U\in L^2_\xi(\R^3_v)\times \mathbb{C}^3_\xi\times \mathbb{C}^3_\xi, \ \ t>0, \label{E_3a}
 \ee
 where
 \bma
 S_1(t,\xi)U&=\sum^2_{j=1}e^{t\lambda_j(s)}
              (U, \Psi^*_j(s,\omega)\,) \Psi_j(s,\omega)
               1_{\{|\xi|\leq r_0\}}, \label{E_5a}\\
 S_2(t,\xi)U&=\sum^4_{j=1} e^{t\beta_j(s)}
              (U,\Phi^*_j(s,\omega)\,) \Phi_j(s,\omega)
               1_{\{|\xi|\ge r_1\}},   \label{E_5b}
 \ema
with $(\lambda_j(s),\Psi_j(s,\omega))$ and $(\beta_j(s),\Phi_j(s,\omega))$ being the eigenvalue and eigenvector of the operator $\hat{\AA}_1(\xi)$ given by Theorem~\ref{eigen_3} and Theorem~\ref{eigen_4} for $|\xi|\le r_0$ and $|\xi|>r_1$ respectively,
and $S_3(t,\xi)=: S(t,\xi)-S_1(t,\xi)-S_2(t,\xi)$ satisfies that there
exists a constant $\kappa_0>0$ independent of $\xi$ such that
 \bq
 \|S_3(t,\xi)U\|_\xi\leq Ce^{-\kappa_0t}\|U\|_\xi,\quad t\ge0.\label{B_3a}
 \eq
\end{thm}
\begin{proof}
Since $D(\hat{B}_1(\xi)^2)$ is dense in $L^2_\xi(\R^3_v)$, by Theorem 2.7
in \cite{Pazy}, it is sufficient to prove the above decomposition for $U=(f,E,B)\in D(\hat{B}_1(\xi)^2)\times \mathbb{C}^3_\xi\times \mathbb{C}^3_\xi$. By Corollary 7.5 in \cite{Pazy},
the semigroup $e^{t\hat{\AA}_1(\xi)}$ can be represented by
 \bq
  e^{t\hat{\AA}_1(\xi)}U
  =\frac1{2\pi \i}\int^{\kappa+\i\infty}_{\kappa-\i\infty}
   e^{\lambda t}(\lambda-\hat{\AA}_1(\xi))^{-1}Ud\lambda,
   \quad  U\in D(\hat{B}_1(\xi)^2)\times \mathbb{C}^3_\xi\times \mathbb{C}^3_\xi,\,\, \kappa>0.          \label{V_3c}
 \eq
It remains to analyze the resolvent $(\lambda-\hat{\AA}_1(\xi))^{-1}$ for $\xi\in\R^3$ in order to obtain the decomposition \eqref{E_3a} for the semigroup $e^{t\hat{\AA}_1(\xi)}$.

By \eqref{S_8}, we rewrite  $(\lambda-\hat{\AA}_1(\xi))^{-1}$ for $|\xi|\le r_0$ as
 \bq
  (\lambda-\hat{\AA}_1(\xi))^{-1}
 =(\lambda-G_3(\xi))^{-1}+Z_1(\lambda,\xi)\label{V_1c},
 \eq
with
 \bmas
 Z_1(\lambda,\xi)
 &=(\lambda-G_3(\xi))^{-1}
   [I-Y_1(\lambda,\xi)]^{-1} Y_1(\lambda,\xi),  
\\
 Y_1(\lambda,\xi)
 &= G_4(\xi)(\lambda-G_3(\xi))^{-1}. 
 \emas
Substituting \eqref{V_1c} into \eqref{V_3c}, we have the following decomposition of the semigroup $e^{t\hat{\AA}_1(\xi)}$
 \be
 e^{t\hat{\AA}_1(\xi)}U
 =(e^{tQ(\xi)}P_rf,0,0)
  -\frac1{2\pi \i}\int^{\kappa+\i\infty}_{\kappa-\i\infty}
    e^{\lambda t}Z_2(\lambda,\xi)Ud\lambda,  \quad |\xi|\le r_0,  \label{V_3a}
  \ee
with
$$
Z_2(\lambda,\xi)=Z_1(\lambda,\xi)-Y_2(\lambda,\xi),\quad {\rm with}\quad
Y_2(\lambda,\xi)=\left(\ba \lambda^{-1}P_{\rm d} & 0 \\  0 & (\lambda -B_3(\xi))^{-1} \ea\right).
$$

To estimate the last term on the right hand side of \eqref{V_3a}, let us denote
 \bq
 X_{\kappa,N}=\frac1{2\pi \i}\int^{N}_{-N}
   e^{(\kappa+\i y)t}Z_2(\kappa+iy,\xi)U1_{|\xi|\le r_0}dy,    \label{UsN}
 \eq
where the constant $N>0$ is chosen large enough so that $N>y_1$ with $y_1$ defined in Lemma \ref{spectrum3}. Since $Z_2(\lambda,\xi)$ is analytic in the domain ${\rm Re}\lambda>-b_2/2$ with only finite singularities at
$\lambda=\lambda_j(s)\in \sigma(\hat{\AA}_1(\xi))$ for $j=1,2$, we can shift the integration
\eqref{UsN} from the line ${\rm Re}\lambda=\kappa>0$ to Re
$\lambda=-b_2/2$. Then
 \bq
 X_{\kappa,N}
 =X_{-\frac{b_2 }2,N}+H_N+2\pi\i\sum_{j=1}^2{\rm Res}
  \lt\{e^{\lambda t}Z_2(\lambda,\xi)U;\lambda_j(s)\rt\}1_{|\xi|\le r_0},    \label{UsN2}
 \eq
where Res$\{f(\lambda);\lambda_j\}$ is the residue of $f$ at $\lambda=\lambda_j$ and
 \bmas
  H_N=\frac1{2\pi \i}\(\int^{\kappa+\i N}_{-\frac{b_2}2+\i N}
      -\int^{\kappa-\i N}_{-\frac{b_2 }2-\i N}\)
        e^{\lambda t}Z_2(\lambda,\xi)U1_{|\xi|\le r_0}d\lambda.
 \emas
The right hand side of \eqref{UsN2} is estimated as follows.
By Lemma \ref{LP}, we have
 \be
 \|H_N\|_\xi\rightarrow0, \quad\mbox{as}\quad N\rightarrow\infty.  \label{UsN2a}
 \ee

By Cauchy Theorem, we obtain
 \bma
 \lim_{N\to\infty}\left|\int^{-\frac{b_2}2+\i N}_{-\frac{b_2}2-\i N}
 e^{\lambda t}\lambda^{-1}d\lambda\right|
  =\lim_{N\to\infty}\left\|\int^{-\frac{b_2}2+\i N}_{-\frac{b_2}2-\i N}
 e^{\lambda t}(\lambda-B_3(\xi))^{-1}d\lambda\right\|
   =0,\label{A_1}
  \ema
which leads to
$$
 \int^{-\frac{b_2}2+\i\infty}_{-\frac{b_2}2-\i\infty}
      e^{\lambda t}Y_2(\lambda,\xi)d\lambda
 =\lim_{N\to\infty}\int^{-\frac{b_2}2+\i N}_{-\frac{b_2}2-\i N}
      e^{\lambda t}Y_2(\lambda,\xi)d\lambda=0.
$$
Thus
 \be
 \lim_{N\to\infty} X_{-\frac{b_2}2,N}(t)
 =X_{-\frac{b_2}2,\infty}(t)
 =:\int^{-\frac{b_2}2+\i \infty}_{-\frac{b_2}2-\i \infty}
  e^{\lambda t}Z_1(\lambda,\xi)Ud\lambda.            \label{UsN3}
\ee

By Lemma \ref{F_1}, it holds that
$
\sup\limits_{|\xi|\le r_0, y\in\R}
\|[I-Y_1(-\frac{b_2 }2+\i y,\xi)]^{-1}\|_\xi\le 2.
$
Thus, we have for any $U,V\in L^2_\xi(\R^3_v)\times \mathbb{C}^3_\xi\times \mathbb{C}^3_\xi$,
  \bmas
 |(X_{-\frac{b_2 }2,\infty}(t)U,V)_\xi|
 \leq
 C e^{-\frac{b_2 t}2}\int^{+\infty}_{-\infty}
   \|[\lambda-G_3(\xi)]^{-1}U\|_\xi\|[\overline{\lambda}-G_3(-\xi)]^{-1}V\|_\xi dy,\quad \lambda=-\frac{b_2}2+\i y.
 \emas
By \eqref{S_4a} and \eqref{S_4b}, we have
 \bmas
 &\int^{+\infty}_{-\infty}
   \|[\lambda-G_3(\xi)]^{-1}U\|^2_\xi dy\\
=&\int^{+\infty}_{-\infty}
   \|[\lambda-Q(\xi)]^{-1}P_rf\|^2+|\lambda|^{-1}\|P_{\rm d}f\|^2_\xi dy+\int^{+\infty}_{-\infty}
   \|[\lambda-B_3(\xi)]^{-1}(E,B)^{T}\|^2dy\\
\le &C(\|f\|^2_\xi+|(E,B)|^2),\quad \lambda=-\frac{b_2}2+\i y,
 \emas
which yields
 $
 |(X_{-\frac{b_2}2,\infty}(t)U,V)_\xi|
  \leq
 C e^{-\frac{b_2 t}2}\|U\|_\xi\|V\|_\xi,
 $
and
 \bq
\|X_{-\frac{b_2}2,\infty}(t)\|_\xi
   \le C  e^{-\frac{b_2 t}2}. \label{UsN4}
 \eq

Since $\lambda_j(s)\in \rho(Q(\xi))$  and $Z_2(\lambda,\xi)=(\lambda-\hat{\AA}_1(\xi))^{-1}-\left(\ba (\lambda P_r-Q(\xi))^{-1}P_r & 0 \\  0 & 0 \ea\right)$,
 by a similar argument as Theorem 3.4 in \cite{Li2}, we can prove 
 \bma
 {\rm Res}\{e^{\lambda t}Z_2(\lambda,\xi)U;\lambda_j(s)\}
  ={\rm Res}\{e^{\lambda t}(\lambda-\hat{\AA}_1(\xi))^{-1}U;\lambda_j(s)\}
  =e^{\lambda_j(s)t}
    (U,\Psi_j^*(s,\omega))\Psi_j(s,\omega).\label{projection1}
 \ema

Therefore, we conclude from \eqref{V_3a}--\eqref{projection1} that
 \bq
   e^{t\hat{\AA}_1(\xi)}U
 = (e^{tQ(\xi)}P_rf,0,0)
  +X_{-\frac{b_2}2,\infty}(t)
  +\sum^2_{j=1}e^{t\lambda_j(s)}
   (U,\Psi_j^*(s,\omega))\Psi_j(s,\omega)1_{\{|\xi|\leq r_0\}},\quad |\xi|\leq r_0. \label{low}
 \eq

By \eqref{E_6}, we have for $|\xi|> r_0$ that
 \bq
(\lambda-\hat{\AA}_1(\xi))^{-1}=(\lambda-G_1(\xi))^{-1}+Z_3(\lambda,\xi),\label{V_2a}
 \eq
with the operator $ Z_3(\lambda,\xi)$ defined by
 \bmas
 Z_3(\lambda,\xi)
 &=(\lambda-G_1(\xi))^{-1}[I-Y_3(\lambda,\xi)]^{-1}Y_3(\lambda,\xi),
 \\
 Y_3(\lambda,\xi)
  &=:G_2(\xi)(\lambda-G_1(\xi))^{-1}.
 \emas
Substituting \eqref{V_2a} into \eqref{V_3c} yields
 \bq
 e^{t\hat{\AA}_1(\xi)}U
 =( e^{tc(\xi)}f,0,0)
  +\frac1{2\pi \i}\int^{\kappa+\i\infty}_{\kappa-\i\infty}
    e^{\lambda t}Z_4(\lambda,\xi)Ud\lambda,\quad |\xi|> r_1.   \label{V_3b}
 \eq
Here,
$$
Z_4(\lambda,\xi)=Z_3(\lambda,\xi)-Y_4(\lambda,\xi),\quad {\rm with}\quad
Y_4(\lambda,\xi)=\left(\ba 0 & 0 \\  0 & (\lambda -B_3(\xi))^{-1} \ea\right).
$$
Similarly, in order to estimate the last term on the right hand side of \eqref{V_3b},
let us denote
 \bq
Y_{\kappa,N}
 =\frac1{2\pi \i}\int^{N}_{-N}e^{(\kappa+\i y)t}Z_4(\kappa+\i y,\xi)U1_{\{|\xi|> r_0\}}dy,    \label{VsN}
 \eq
where the constant $N>y_1$ for $|\xi|\le r_1$ and $N>2|\xi|$ for $|\xi|\ge r_1$ with $y_1,r_1$ defined in Lemma \ref{spectrum3}.
Since the operator $Z_4(\lambda,\xi)$ is analytic in the domain ${\rm Re}\lambda\ge -\kappa_0=:-\alpha(r_0,r_1)/2>0$ and  $r_0<|\xi|< r_1$ with the constant $\alpha(r_0,r_1)>0$ defined in Lemma \ref{LP01}, and is analytic
except only finite singularities at
$\lambda=\lambda_j(s)\in \sigma(\hat{\AA}_1(\xi))$ for $j=1,2,3,4$, in the domain ${\rm Re}\lambda\ge -\mu/2$ and $|\xi|\ge r_1$, we can  shift the integration of \eqref{VsN} for $r_0<|\xi|<r_1$ from the line ${\rm Re}\lambda=\kappa>0$ to $\mbox{Re}\lambda=-\kappa_0$, and shift the integration of \eqref{VsN} for $|\xi|\ge r_1$ from the line ${\rm Re}\lambda=\kappa>0$ to $\mbox{Re}\lambda=-\mu/2$ to obtain
\bma
 &Y_{\kappa,N}1_{\{r_0<|\xi|< r_1\}}=Y_{-\kappa_0,N}1_{\{r_0<|\xi|< r_1\}}+I_N,   \label{VsN1}\\
  &Y_{\kappa,N}1_{\{|\xi|\ge r_1\}}=Y_{-\frac{\mu}2,N}1_{\{|\xi|\ge r_1\}}+J_N+2\pi\i\sum^4_{j=1}{\rm Res}
  \lt\{e^{\lambda t}Z_4(\lambda,\xi)U;\beta_j(s)\rt\}1_{\{|\xi|\ge r_1\}},
 \ema
with
\bmas
 I_N&=\frac1{2\pi \i}\(\int^{\kappa+\i N}_{-\kappa_0+\i N}-\int^{\kappa-\i N}_{-\kappa_0-\i N}\)
     e^{\lambda t}Z_4(\lambda,\xi)U1_{\{r_0<|\xi|< r_1\}}d\lambda,\\
 J_N&=\frac1{2\pi \i}\(\int^{\kappa+\i N}_{-\frac{\mu}2+\i N}-\int^{\kappa-\i N}_{-\frac{\mu}2-\i N}\)
     e^{\lambda t}Z_4(\lambda,\xi)U1_{\{|\xi|\ge r_1\}}d\lambda.
\emas
By Lemma~\ref{LP03} and Lemma \ref{F_1}, it is straightforward to verify
 \bgr
 \|I_N\|\rightarrow0,\quad \|J_N\|\rightarrow0 \ \ \mbox{as}\ \  N\rightarrow\infty,    \label{VsNa}\\
 \sup_{r_0<|\xi|<r_1, y\in\R} \|[I-Y_4(-\frac{b_2}2+\i y,\xi)]^{-1}\| \le C,\quad \sup_{|\xi|>r_1, y\in\R} \|[I-Y_4(-\frac{\mu}2+\i y,\xi)]^{-1}\| \le 2.
  \egr
By \eqref{VsN} and \eqref{A_1},  it holds that
 \bmas
  &Y_{-\kappa_0,\infty}(t)
 =\lim_{N\to \infty}Y_{-\kappa_0,N}(t)
 =\int^{+\infty}_{-\infty}e^{(-\kappa_0+\i y)t}Z_3(-\kappa_0+\i y,\xi)Udy,
 \\
  &Y_{-\frac{\mu}2,\infty}(t)
 =\lim_{N\to \infty}Y_{-\frac{\mu}2,N}(t)
 =\int^{+\infty}_{-\infty}e^{(-\frac{\mu}2+\i y)t}Z_3(-\frac{\mu}2+\i y,\xi)Udy.
 \emas
Then we have for any $U,V\in L^2(\R^3_v)\times \mathbb{C}^3_\xi\times \mathbb{C}^3_\xi$,
 \bma
 |(Y_{-\kappa_0,\infty}(t)U1_{\{r_0<|\xi|< r_1\}},V)|
&\leq
C e^{-\kappa_0 t}\int^{+\infty}_{-\infty}\|(\lambda-G_1(\xi))^{-1}U\|\|(\overline{\lambda}-G_1(-\xi))^{-1}V\|dy
\nnm\\
&\leq
 C(\nu_0-\kappa_0)^{-1} e^{-\frac{\mu}2 t}\|U\|\|V\|, \quad \lambda=-\kappa_0+\i y, \label{VsNb1}
 \\
|(Y_{-\frac{\mu}2,\infty}(t)U1_{\{|\xi|\ge r_1\}},V)|
&\leq
C e^{-\frac{\mu}2 t}\int^{+\infty}_{-\infty}\|(\lambda-G_1(\xi))^{-1}U\|\|(\overline{\lambda}-G_1(-\xi))^{-1}V\|dy
\nnm\\
&\leq
 C(\nu_0-\frac{\mu}2)^{-1} e^{-\frac{\mu}2 t}\|U\|\|V\|, \quad \lambda=-\frac{\mu}2+\i y,  \label{VsNb1a}
 \ema
where we have used \eqref{S_4b} and the fact (cf. Lemma 2.2.13 of \cite{Ukai3})
that
 \bq
\int^{+\infty}_{-\infty}\|(x+\i y-c(\xi))^{-1}f\|^2dy\leq \pi(x+\nu_0)^{-1}\|f\|^2,\quad x>-\nu_0. \nnm
 \eq
From \eqref{VsNb1}, \eqref{VsNb1a} and the fact that $\|f\|^2\leq\|f\|_\xi^2\leq(1+r_0^{-2})\|f\|^2$ for $|\xi|> r_0$, we have
  \bq
  \|Y_{-\kappa_0,\infty}(t)1_{\{r_0<|\xi|< r_1\}}\|_\xi\le C e^{-\kappa_0t},\quad \|Y_{-\frac{\mu}2,\infty}(t)1_{\{|\xi|\ge r_1\}}\|_\xi
   \leq C e^{-\frac{\mu}2t}. \label{VsN2}
 \eq

By $\lambda_j(s)\in \rho(c(\xi))$ and $Z_4(\lambda,\xi)=(\lambda-\hat{\AA}_1(\xi))^{-1}-\left(\ba (\lambda -c(\xi))^{-1} & 0 \\  0 & 0 \ea\right)$, we can prove 
 \be
 {\rm Res}\{e^{\lambda t}Z_4(\lambda,\xi)U;\beta_j(s)\}
  ={\rm Res}\{e^{\lambda t}(\lambda-\hat{\AA}_1(\xi))^{-1}U;\beta_j(s)\}
  =e^{\beta_j(s)t}
    (U,\Phi_j^*(s,\omega))\Phi_j(s,\omega).   \label{projection}
\ee

Therefore, we conclude from \eqref{V_3b}--\eqref{projection} that
 \bma
   e^{t\hat{\AA}_1(\xi)}U
 =&(e^{tc(\xi)}f,0,0) +Y_{-\kappa_0,\infty}(t)1_{\{r_0<|\xi|< r_1\}}+Y_{-\frac{\mu}2,\infty}(t)1_{\{|\xi|\ge r_1\}}
 \nnm\\
 &+\sum^4_{j=1}e^{t\beta_j(s)}
   (U,\Phi_j^*(s,\omega))\Phi_j(s,\omega)1_{\{|\xi|> r_1\}}, \quad |\xi|> r_0.  \label{high}
 \ema

Combining \eqref{low} and \eqref{high} gives  \eqref{E_3a} with $S_1(t,\xi)f$, $S_2(t,\xi)f$ and $ S_3(t,\xi)f$ defined by
 \bmas
 S_1(t,\xi)U&=\sum^2_{j=1}e^{t\lambda_j(s)}
              (U,\Psi_j^*(s,\omega))\Psi_j(s,\omega)1_{\{|\xi|\leq r_0\}}, \\
 S_2(t,\xi)U&=\sum^4_{j=1}e^{t\beta_j(s)}
              (U,\Phi_j^*(s,\omega))\Phi_j(s,\omega)1_{\{|\xi|\ge r_1\}}, \\
 S_3(t,\xi)U&=(e^{tQ(\xi)}P_rf,0,0)1_{\{|\xi|\leq r_0\}}+X_{-\frac{b_2}2,\infty}(t)1_{\{|\xi|\leq r_0\}}\\
            &\quad       +(e^{tc(\xi)}f,0,0)1_{\{|\xi|\ge r_1\}}+Y_{-\kappa_0,\infty}(t)1_{\{r_0<|\xi|< r_1\}}+Y_{-\frac{\mu}2,\infty}(t)1_{\{|\xi|\ge r_1\}}.
\emas
In particular, $S_3(t,\xi)U$  satisfies \eqref{B_3a} because of \eqref{decay_1}, \eqref{UsN4},  \eqref{VsN2} and the estimate $\| e^{tc(\xi)}1_{\{|\xi|> r_0\}}\|_\xi \le  Ce^{-\nu_0t}$ coming from \eqref{Cxi} and \eqref{nuv}. This completes the proof of the theorem.
\end{proof}

\subsection{Optimal convergence rates for two-species}

Based on the decomposition of the semigroup given in the previous subsection,
we now study the optimal convergence rates of the solution
of the linearized system to the equilibrium.

Set $U=(f,E,B)$ with $f=f(x,v)$, $E=E(x)$ and $B=B(x)$ and denote $ D^l=\{U\in L^2(\R^3_{x}\times \R^3_{v})\times L^2(\R^3_{x})\times L^2(\R^3_{x}) \,|\, \|U\|_{D^l}<\infty\,\}\, (Z^2=D^0)$
with the norm $\|\cdot\|_{D^l}$ defined by
 \bmas
 \|U\|_{D^l}
 &=\(\intr (1+|\xi|^2)^l \|\hat{U}\|^2 d\xi \)^{1/2}\\
 &=\(\intr (1+|\xi|^2)^l
     \(\intr |\hat{f}|^2dv+ |\hat E|^2+| \hat B|^2\)d\xi
    \)^{1/2},
 \emas
where $\hat{f}=\hat{f}(\xi,v)$, $\hat{E}=\hat{E}(\xi)$ and $\hat{B}=\hat{B}(\xi)$.

For any $U_0=(f_0,E_0,B_0)\in L^2(\R^3_v;H^l(\R^3_{x}))\times H^l(\R^3_{x})\times H^l(\R^3_{x})$,
set
 \bq
 e^{t\hat \AA_0(\xi)}\hat U_0=( (e^{t\hat{\AA}_1(\xi)}\hat V_0)_1 ,-\frac{\i\xi}{|\xi|^2}((e^{t\hat{\AA}_1(\xi)}\hat V_0)_1,\sqrt M)-\frac{\xi}{|\xi|}\times(e^{t\hat{\AA}_1(\xi)}\hat V_0)_2,-\frac{\xi}{|\xi|}\times(e^{t\hat{\AA}_1(\xi)}\hat V_0)_3),
\label{solution1}
  \eq
  with
\bgrs
e^{t\hat{\AA}_1(\xi)}\hat V_0
 =( (e^{t\hat{\AA}_1(\xi)}\hat V_0)_1,(e^{t\hat{\AA}_1(\xi)}\hat V_0)_2,(e^{t\hat{\AA}_1(\xi)}\hat V_0)_3)\in L^2_\xi(\R^3_v)\times \mathbb{C}^3_\xi\times \mathbb{C}^3_\xi,\\
\hat V_0=(\hat f_0,\frac{\xi}{|\xi|}\times \hat E_0,\frac{\xi}{|\xi|}\times \hat B_0).
\egrs

Then $e^{t\AA_0}U_0$ is the solution of the system \eqref{LVMB1a}. By Lemma \ref{SG_1}, it holds that
 $$
 \|e^{t \AA_0} U_0\|_{D^l}=\intr (1+|\xi|^2)^l\|e^{t\hat{\AA}_1(\xi)}\hat V_0\|^2_\xi d\xi\le \intr (1+|\xi|^2)^l\|\hat V_0\|^2_\xi d\xi
=\|U_0\|_{D^l}.
 $$
This means that the linear operator $\AA_0$ generates a strongly continuous contraction semigroup $e^{t\AA_0}$ on $D^l$, and therefore, $U(t)=e^{t\AA_0}U_0$ is a global solution to \eqref{LVMB1a} for the linearized Vlasov-Maxwell-Boltzmann equations
with initial data $U_0\in D^l$.

First of all, we have the upper bounds of the time decay rates given in
\begin{thm}
 \label{rate1}
 Let $(f_2(t),E(t),B(t))$ be a solution of the system \eqref{LVMB1a}. If the initial data $U_0=(f_0,E_0,B_0)\in L^2(\R^3_{v};H^l(\R^3_{x}) \cap L^1(\R^3_{x}))\times H^l(\R^3_x)\cap L^1(\R^3_{x})\times H^l(\R^3_x)\cap L^1(\R^3_{x})$, then it holds for any $\alpha,\alpha'\in\N^3$ with $\alpha'\le \alpha$ that
 \bma
\|\da_x P_{\rm d}f(t)\|_{L^2_{x,v}}
\leq
&Ce^{-\kappa_0t} \|\da_x U_0\|_{Z^2},\label{D_2z}
\\
\|\da_x P_rf(t)\|_{L^2_{x,v}}
\leq
&C(1+t)^{-(\frac54+\frac{k}2)}
  (\|\da_x U_0\|_{Z^2}+\|\dx^{\alpha'}U_0\|_{Z^1})+C(1+t)^{-(m+\frac12)}\|\Tdx^{m}\da_x U_0\|_{Z^2}, \label{D_3z}
\\
\|\da_x E(t)\|_{L^2_x}
\leq
&C(1+t)^{-(\frac54+\frac{k}2)}
  (\|\da_x U_0\|_{Z^2}+\|\dx^{\alpha'}U_0\|_{Z^1})+C(1+t)^{-m}\|\Tdx^{m}\da_x U_0\|_{Z^2}, \label{D_4z}
\\
\|\dxa B(t)\|_{L^2_x}
\leq
&C(1+t)^{-(\frac34+\frac{k}2)}
  (\|\da_x U_0\|_{Z^2}+\|\dx^{\alpha'}U_0\|_{Z^1})+C(1+t)^{-m}\|\Tdx^{m}\da_x U_0\|_{Z^2},\label{D_0z}
 \ema
 where $k=|\alpha-\alpha'|$ and $m\ge 0$.
\end{thm}
 \begin{proof}
By  \eqref{solution1} and Theorem \ref{semigroup1}, we have for $\omega=\xi/|\xi|$ that
\bma
(\hat f_2(t), \omega\times\hat E(t),\omega\times\hat B(t))&=e^{t\hat{\AA}_1(\xi)}\hat V_0=S_1(t,\xi)\hat V_0+S_2(t,\xi)\hat V_0+S_3(t,\xi)\hat V_0\nnm\\
&=\sum^3_{k=1}(h_k(t),H_k(t),J_k(t)),
\ema
where
$
S_k(t,\xi)\hat V_0=(h_k(t),H_k(t),J_k(t))\in L^2_\xi(\R^3_v)\times \mathbb{C}^3_\xi\times \mathbb{C}^3_\xi
,\ k=1,2,3,$ and $
\hat V_0=(\hat f_0,\omega\times \hat E_0,\omega\times \hat B_0).
$ 
Then
 \bma
   \|\dxa P_rf_2(t)\|_{L^2_{x,v}}
 &=\|\xi^\alpha P_r\hat f_2(t)\|_{L^2_{\xi,v}}
 \nnm\\
&\le
  \|\xi^\alpha P_rh_1(t)\|_{L^2_{\xi,v}}
 +\|\xi^\alpha P_rh_2(t)\|_{L^2_{\xi,v}}+\|\xi^\alpha P_rh_3(t)\|_{L^2_{\xi,v}}, \label{D1a}
 \\
 \|\dxa E(t)\|_{L^2_x}&=\|\xi^\alpha \hat E(t)\|_{L^2_\xi}=(\|\xi^\alpha|\xi|^{-1}(\hat f(t),\chi_0)\|^2_{L^2_\xi}+\|\xi^\alpha  (\omega\times \hat E)(t)\|^2_{L^2_\xi})^{1/2}\nnm\\
 &\le
   \|\frac{\xi^\alpha}{|\xi|}(h_1(t),\chi_0)\|_{L^2_\xi}
 +\|\frac{\xi^\alpha}{|\xi|} (h_2(t),\chi_0)\|_{L^2_\xi}+\|\frac{\xi^\alpha}{|\xi|} (h_3(t),\chi_0)\|_{L^2_\xi}\nnm\\
 &\quad+\|\xi^\alpha H_1(t)\|_{L^2_\xi}
 +\|\xi^\alpha H_2(t)\|_{L^2_\xi}+\|\xi^\alpha  H_3(t)\|_{L^2_\xi}, \label{D1b}
 \\
 \|\dxa B(t)\|_{L^2_x}&=\|\xi^\alpha \hat B(t)\|_{L^2_\xi}=\|\xi^\alpha  (\omega\times \hat B)(t)\|_{L^2_\xi}\nnm\\
 &\le
  \|\xi^\alpha J_1(t)\|_{L^2_\xi}
 +\|\xi^\alpha J_2(t)\|_{L^2_\xi}+\|\xi^\alpha  J_3(t)\|_{L^2_\xi}. \label{D1c}
 \ema
 By \eqref{B_3a} and $|\hat E_0|^2=\frac1{|\xi|^2}|(\hat f_0,\sqrt M)|^2+|\omega\times E_0|^2$, $|\hat B_0|^2=|\omega\times B_0|^2$,
we can estimate the terms on the right hand side of \eqref{D1a}--\eqref{D1c} as follows:
 \bma
&\intr (\xi^{\alpha})^2(\| h_3(t)\|^2_{L^2_v}
 +\frac1{|\xi|^2}|(h_3(t),\sqrt M)|^2 +|H_3(t)|^2+|J_3(t)|^2)d\xi
 \nnm\\
\leq
& C\intr e^{-2\kappa_0 t}(\xi^{\alpha})^2
  (\|\hat f_0\|^2_{L^2_v}+\frac1{|\xi|^2}|(\hat{f}_0,\sqrt{M})|^2+|\omega\times \hat E_0|^2+|\omega\times \hat B_0|^2)d\xi
 \nnm\\
\leq
& C e^{-2\kappa_0 t}(\|\da_x f_0\|^2_{L^2_{x,v}}+\|\dxa E_0\|_{L^2_x}+\|\dxa B_0\|_{L^2_x}).\label{D_1}
 \ema

In the low frequency region, by \eqref{E_5a}, we have
\bmas
  S_1(t,\xi)\hat{V}_0&=\sum^2_{j=1}e^{t\lambda_j(|\xi|)}\big\{[(\hat{f}_0,\overline{\psi_{j,0}}\,) -(\omega\times\hat{E}_0,\overline{X_{j,0}})-(\omega\times\hat{B}_0,\overline{Y_{j,0}})] (\psi_{j,0},X_{j,0},Y_{j,0})\\
  &\quad+|\xi|((T_j(\xi)\hat{V}_0)_1,(T_j(\xi)\hat{V}_0)_2,(T_j(\xi)\hat{V}_0)_3)\big\}1_{\{|\xi|\leq r_0\}},
\emas
where $T_j(\xi), \ j=1,2,$ are the  linear operators with the norm $\|T_j(\xi)\|$ being uniformly bounded for $|\xi|\leq r_0$.

By \eqref{eigfr0} and \eqref{eigf1}, we have
\bma
 (h_1(t),\sqrt M)=&0,
 \quad
P_rh_1(t)
=|\xi|\sum_{j=1}^2e^{\lambda_j(|\xi|)t}(T_j(\xi)\hat{V}_0)_1,\label{F_2}
  \\
 H_1(t)
 =&|\xi|\sum_{j=1}^2e^{\lambda_j(|\xi|)t}(T_j(\xi)\hat{V}_0)_2,\label{F_3}
 \\
J_1(t)
 =&\sum_{j=1}^2 e^{\lambda_j(|\xi|)t}(\omega\times\hat B_0,W^j)W^j+|\xi|\sum_{j=1}^2e^{\lambda_j(|\xi|)t}(T_j(\xi)\hat{V}_0)_3, \label{F_4}
 \ema
where
$W^j, \ j=1,2,$ is given by \eqref{eigf1}.
Since
 \bq
 \mathrm{Re}\lambda_j(|\xi|)
 =a_j|\xi|^2(1+O(|\xi|))
 \le - \eta_1 |\xi|^2,\quad |\xi|\leq r_0,     \label{ee}
 \eq
where $\eta_1>0$ denotes a generic constant that will also be used later, we obtain by \eqref{F_2}--\eqref{F_4} that
 \bma
 \|\xi^\alpha P_rh_1(t)\|^2_{L^2_{\xi,v}}
 &\leq
 C\int_{|\xi|\leq r_0}(\xi^\alpha)^{2}|\xi|^2e^{-2\eta_1 |\xi|^2t}(\|\hat{f}_0\|^2_{L^2_v}+|\hat E_0|^2+|\hat B_0|^2)d\xi\nnm\\
&\leq
  C(1+t)^{-(5/2+k)}(\|\dx^{\alpha'}f_0\|^2_{L^{2,1}}+\|\dx^{\alpha'}E_0\|^2_{L^1_x}+\|\dx^{\alpha'}B_0\|^2_{L^1_x}),\label{V_6}
  \\
  \|\xi^\alpha H_1(t)\|^2_{L^2_\xi}
 &\leq
 C\int_{|\xi|\leq r_0}(\xi^\alpha)^{2}|\xi|^2e^{-2\eta_1 |\xi|^2t}(\|\hat{f}_0\|^2_{L^2_v}+|\hat E_0|^2+|\hat B_0|^2)d\xi\nnm\\
&\leq
  C(1+t)^{-(5/2+k)}(\|\dx^{\alpha'}f_0\|^2_{L^{2,1}}+\|\dx^{\alpha'}E_0\|^2_{L^1_x}+\|\dx^{\alpha'}B_0\|^2_{L^1_x}),\label{V_6a}
\\
 \|\xi^\alpha J_1(t)\|^2_{L^2_\xi}
&\leq
 C\int_{|\xi|\leq r_0}(\xi^\alpha)^{2}e^{-2\eta_1 |\xi|^2t}(|\xi|^2\|\hat{f}_0\|^2_{L^2_v}+|\xi|^2|\hat E_0|^2+|\hat B_0|^2)d\xi\nnm\\
&\leq
  C(1+t)^{-(3/2+k)}(\|\dx^{\alpha'}f_0\|^2_{L^{2,1}}+\|\dx^{\alpha'}E_0\|^2_{L^1_x}+\|\dx^{\alpha'}B_0\|^2_{L^1_x}),\label{V_5}
\ema
with $k=|\alpha-\alpha'|$.

In the high frequency region, by \eqref{E_5b}, we have
\bma
  S_2(t,\xi)\hat{V}_0&=\sum^4_{j=1}e^{t\beta_j(|\xi|)}[(\hat{f}_0,\overline{\phi_j}\,)-(\omega\times\hat{E}_0,\overline{X_j})-(\omega\times\hat{B}_0,\overline{Y_j})] (\phi_j,X_j,Y_j)(s,\omega)1_{\{|\xi|\ge r_1\}}, \label{high1}
\ema and in particular $(h_2(t),\sqrt M)=0.$
Since
 \bq
 \mathrm{Re}\beta_j(|\xi|) \le -c_1 |\xi|^{-1},\quad |\xi|\ge r_1,     \label{ee1}
 \eq
we obtain by \eqref{high1} and \eqref{eigfr3} that
 \bma
 \|\xi^\alpha P_rh_2(t)\|^2_{L^2_{\xi,v}}
 &\leq
 C\int_{|\xi|\ge r_1}(\xi^\alpha)^{2}\frac1{|\xi|}e^{-\frac{2c_1 t}{|\xi|}}(\|\hat{f}_0\|^2_{L^2_v}+|\hat E_0|^2+|\hat B_0|^2)d\xi\nnm\\
 &\leq
  C\sup_{|\xi|\ge r_1}\frac1{|\xi|^{2m+1}}e^{-\frac{2c_1 t}{|\xi|}}\int_{|\xi|\ge r_1}(\xi^\alpha)^{2}|\xi|^{2m}(\|\hat{f}_0\|^2_{L^2_v}+|\hat E_0|^2+|\hat B_0|^2)d\xi,\label{V_4a}
  \\
 \|\xi^\alpha H_2(t)\|^2_{L^2_\xi}
&\leq
 C\int_{|\xi|\ge r_1}(\xi^\alpha)^{2}e^{-\frac{2c_1 t}{|\xi|}}(\|\hat{f}_0\|^2_{L^2_v}+|\hat E_0|^2+|\hat B_0|^2)d\xi\nnm\\
&\leq
  C\sup_{|\xi|\ge r_1}\frac1{|\xi|^{2m}}e^{-\frac{2c_1 t}{|\xi|}}\int_{|\xi|\ge r_1}(\xi^\alpha)^{2}|\xi|^{2m}(\|\hat{f}_0\|^2_{L^2_v}+|\hat E_0|^2+|\hat B_0|^2)d\xi,\label{V_5a}
\\
\|\xi^\alpha J_2(t)\|^2_{L^2_\xi}
&\leq
 C\int_{|\xi|\ge r_1}(\xi^\alpha)^{2}e^{-\frac{2c_1 t}{|\xi|}}(\|\hat{f}_0\|^2_{L^2_v}+|\hat E_0|^2+|\hat B_0|^2)d\xi\nnm\\
&\leq
  C\sup_{|\xi|\ge r_1}\frac1{|\xi|^{2m}}e^{-\frac{2c_1 t}{|\xi|}}\int_{|\xi|\ge r_1}(\xi^\alpha)^{2}|\xi|^{2m}(\|\hat{f}_0\|^2_{L^2_v}+|\hat E_0|^2+|\hat B_0|^2)d\xi.\label{V_6b}
\ema
Since
$$\sup_{|\xi|\ge r_1}\frac1{|\xi|^{2m}}e^{-\frac{2c_1 t}{|\xi|}}\le C(1+t)^{-2m},$$
it follows from \eqref{V_4a}--\eqref{V_6b} that
\bma
 \|\xi^\alpha P_rh_2(t)\|^2_{L^2_{\xi,v}}&\le
 C(1+t)^{-(2m+1)}(\|\Tdx^m\dxa f_0\|^2_{L^{2}_{x,v}}+\|\Tdx^m\dxa E_0\|^2_{L^2_x}+\|\Tdx^m\dxa B_0\|^2_{L^2_x}),\label{V_4b}
\\
 \|\xi^\alpha H_2(t)\|^2_{L^2_\xi}&\le
C(1+t)^{-2m}(\|\Tdx^m\dxa f_0\|^2_{L^{2}_{x,v}}+\|\Tdx^m\dxa E_0\|^2_{L^2_x}+\|\Tdx^m\dxa B_0\|^2_{L^2_x}),\label{V_4c}
\\
  \|\xi^\alpha J_2(t)\|^2_{L^2_\xi}&\le
C(1+t)^{-2m}(\|\Tdx^m\dxa f_0\|^2_{L^{2}_{x,v}}+\|\Tdx^m\dxa E_0\|^2_{L^2_x}+\|\Tdx^m\dxa B_0\|^2_{L^2_x}).\label{V_4d}
\ema
The combination of \eqref{D1a}--\eqref{D_1}, \eqref{V_6}--\eqref{V_5} and \eqref{V_4b}--\eqref{V_4d} leads to \eqref{D_2z}--\eqref{D_0z}. And this completes the proof of the theorem.
\end{proof}

In fact, the above time decay rates are optimal as shown in

\begin{thm}\label{rate2}
 Let $(f_2(t),E(t),B(t))$ be a solution of the system \eqref{LVMB1a}. Assume that the initial data $U_0=(f_0,E_0,B_0)\in L^2(\R^3_{v};H^l(\R^3_{x}) \cap L^1(\R^3_{x}))\times H^l(\R^3_x)\cap L^1(\R^3_{x})\times H^l(\R^3_x)\cap L^1(\R^3_{x})$ for $l\ge 2$ and $\hat U_0=(\hat f_0,\hat E_0,\hat B_0)$  satisfies that
$\inf_{|\xi|\le r_0}|\frac{\xi}{|\xi|}\times\hat{B}_0|\geq d_0>0$, then
  \bgr
 C_1(1+t)^{-\frac54}
  \leq\| P_rf_2(t)\|_{L^2_{x,v}}\leq C_2(1+t)^{-\frac54},\label{H_1bz}
\\
 C_1(1+t)^{-\frac54}
  \leq\|E(t)\|_{L^2_x}\leq C_2(1+t)^{-\frac54},\label{H_2bz}
\\
 C_1(1+t)^{-\frac34}
  \leq\|B(t)\|_{L^2_x}\leq C_2(1+t)^{-\frac34}, \label{H_3z}
 \egr
for $t>0$  large enough with $C_2\ge C_1>0$ as two generic constants.
\end{thm}
\begin{proof}
By Theorem \ref{rate1}, we only need to show the lower bounds of the time decay rates for the solution $(f(t),E(t),B(t))$  under the assumptions of Theorem \ref{rate2}.  Indeed, in terms of Theorem~\ref{rate1}, we have
 \bma
   \| P_r f_2(t)\|_{L^2_{x,v}}
&\ge
  \|P_rh_1(t)\|_{L^2_{\xi,v}}
 -\| P_rh_2(t)\|_{L^2_{\xi,v}}-\| P_rh_3(t)\|_{L^2_{\xi,v}}\nnm\\
 &\ge
  \|P_rh_1(t)\|_{L^2_\xi}-C(1+t)^{-2}-Ce^{-Ct}, \label{H_4}
 \\
 \| E(t)\|_{L^2_x}
 &\ge
   \frac{\sqrt2}2(\| \frac1{|\xi|}(h_1(t),\chi_0)\|_{L^2_\xi}
 -\|\frac1{|\xi|}(h_2(t),\chi_0)\|_{L^2_\xi}-\|\frac1{|\xi|} (h_3(t),\chi_0)\|_{L^2_\xi})\nnm\\
 &\quad+\frac{\sqrt2}2(\|H_1(t)\|_{L^2_\xi}
 -\| H_2(t)\|_{L^2_\xi}-\|H_3(t)\|_{L^2_\xi})\nnm\\
 &\ge
  \frac{\sqrt2}2(\| \frac1{|\xi|}(h_1(t),\chi_0)\|_{L^2_\xi}+\|H_1(t)\|_{L^2_\xi})-C(1+t)^{-2}-Ce^{-Ct}, \label{H_4a}
 \\
 \|B(t)\|_{L^2_x}
 &\ge
  \| J_1(t)\|_{L^2_\xi}
 -\| J_2(t)\|_{L^2_\xi}-\| J_3(t)\|_{L^2_\xi}\nnm\\
 &\ge
  \| J_1(t)\|_{L^2_\xi}-C(1+t)^{-2}-Ce^{-Ct}, \label{Q_4}
 \ema
where we have used \eqref{D_1} and \eqref{V_4b}--\eqref{V_4d} for $|\alpha|=0$.

By \eqref{F_2} and $\lambda_1(|\xi|)=\lambda_2(|\xi|)$, we have
\bma
P_rh_1(t)=\i a_1|\xi|e^{\lambda_1(|\xi|)t}\sum_{j=1,2}(\omega\times\hat B_0,W^j)L^{-1}_1(v\cdot W^j)\chi_0+|\xi|^2T_3(t,\xi)\hat U_0,\label{x}
\ema
where $T_3(t,\xi)\hat U_0$ is a linear operator satisfying $\|T_3(t,\xi)\hat U_0\|^2\le Ce^{-2\eta_1 |\xi|^2t}\|\hat U_0\|^2$. Since the terms $L^{-1}_1(v\cdot W^1)\sqrt M$ and $L^{-1}_1(v\cdot W^2)\sqrt M$ are orthogonal, 
it follows from \eqref{x} that
 \bma
\|P_rh_1(t)\|^2_{L^2_v}
\geq
   \frac12|\xi|^2a_1^2\|L^{-1}_1\chi_1\|^2_{L^2_v}e^{2\mathrm{Re}\lambda_1(|\xi|)t}|\omega\times\hat B_0|^2
  -C|\xi|^4e^{-2\eta_1 |\xi|^2t}\|\hat U_0\|^2.\label{B_1b}
 \ema
Since
$$\mathrm{Re}\lambda_j(|\xi|)=a_j|\xi|^2(1+O(|\xi|))\ge -\eta_2 |\xi|^2, \quad |\xi|\leq r_0,
$$
for some constant $\eta_2>0$, we obtain by \eqref{B_1b} that \bma
\|P_rh_1(t)\|^2_{L^2_{\xi,v}}
\geq&
 \frac12d_0^2a_1^2\|L^{-1}_1\chi_1\|^2_{L^2_v}\int_{|\xi|\leq r_0} |\xi|^2e^{-2\eta_2 |\xi|^2t}d\xi -C\int_{|\xi|\leq r_0} e^{-2\eta_1 |\xi|^2t} |\xi|^4\|\hat{U}_0\|^2d\xi\nnm\\
=:&\frac12d_0^2a_1^2\|L^{-1}_1\chi_1\|^2_{L^2_v}I_1- C\|U_0\|^2_{Z^1}(1+t)^{-7/2}.    \label{E_7a}
 \ema
Moreover,  for time $t \geq t_0=\frac{1}{r_0^2}$, we have
 \bma
 I_1= \int_{|\xi|\leq r_0}|\xi|^2e^{-2\eta_2 |\xi|^2t} d\xi = 4\pi t^{-5/2}\int^{ r_0\sqrt{t}}_0
   e^{-2\eta_2 r^2}r^4 dr \geq C_0 (1+t)^{-5/2},\label{E_9b}
 \ema
where $C_0> 0$ denotes a generic positive constant. We can  substitute \eqref{E_7a} and \eqref{E_9b} into  \eqref{H_4}  to obtain \eqref{H_1bz}.

By \eqref{F_3}, we have
$$
H_1(t)=\i a_1|\xi|e^{\lambda_1(|\xi|)t}\sum_{j=1,2}(\omega\times\hat B_0,W^j)(\omega\times W^j)+|\xi|^2T_4(t,\xi)\hat U_0, \quad (h_1(t),\sqrt M)=0,
 $$
where $T_4(t,\xi)$ is a linear operator satisfying $\|T_3(t,\xi)\hat U_0\|^2\le Ce^{-2\eta_1 |\xi|^2t}\|\hat U_0\|^2$. Then
 $$
 |H_1(t)|^2
 \geq
\frac12|\xi|^2a_1^2e^{2\mathrm{Re}\lambda_1(|\xi|)t}|\omega\times\hat B_0|^2
    -Ce^{-2\eta_1 |\xi|^2t}|\hat{U}_0|^2.
 $$
Similar to  \eqref{E_9b}, we get
 \bmas
 \|H_1(t)\|^2_{L^2_\xi}
\geq &
 \frac12d_0^2a_1^2\int_{|\xi|\le r_0}|\xi|^2
   e^{-2\eta_2|\xi|^2t}d\xi
-C\int_{|\xi|\leq r_0} |\xi|^4e^{-2\eta_1  |\xi|^2 t}
 \|\hat{U}_0\|^2d\xi
\nnm\\
\geq &  C_3(1+t)^{-5/2}- C(1+t)^{-7/2},
\emas
which together with \eqref{H_4a}  lead to \eqref{H_2bz} for $t>0$ being large enough.

By \eqref{F_4}, we have
 $$
J_1(t)=-e^{\lambda_1(|\xi|)t}\sum_{j=1,2}(\omega\times\hat B_0,W^j) W^j+|\xi|T_5(t,\xi)\hat U_0,
 $$
where $T_5(t,\xi)$ is a linear operator satisfying $\|T_5(t,\xi)\hat U_0\|^2\le Ce^{-2\eta_1 |\xi|^2t}\|\hat U_0\|^2$. Then
 \bmas
 \|J_1(t)\|^2_{L^2_\xi}
 &\geq
 \frac12d_0^2\int_{|\xi|\le r_0}e^{-2\eta_2 |\xi|^2t}d\xi
  -C\int_{|\xi|\leq r_0}|\xi|^2e^{-2\eta_1 |\xi|^2 t}
   \|\hat{U}_0\|^2d\xi
\\
 &\ge C_3(1+t)^{-3/2}  -C(1+t)^{-5/2}.
 \emas
This and \eqref{Q_4}  give  \eqref{H_3z} for $t>0$  being
large enough.
The proof is then completed.
\end{proof}

For  comparison, we include the known time decay rates of the global solution to the linearized Boltzmann equation~\eqref{VMB4a} as follows,
ref. \cite{Ukai1,Zhong2012Sci} and references therein.

\begin{thm}\label{time1b} Assume that $f_0\in L^2_v(H^N_x)\cap L^{2,q}$ for $N\ge 1$ and $q\in[1,2]$. Then there is a globally unique solution $f(x,v,t)=e^{t\BB_0}f_0(x,v)$ to  the linearized Boltzmann equation~\eqref{VMB4a}, which satisfies for any $\alpha,\alpha'\in\N^3$ with  $|\alpha|\le N$, $\alpha'\le \alpha$  and $k=|\alpha-\alpha'|$ that
 \bma
 \|(\da_x e^{t\BB_0}f_0,\chi_j)\|_{L^2_x}
&\leq C(1+t)^{-\frac 32\(\frac1q-\frac12\)-\frac k2}(\|\da_x
f_0\|_{L^2_{x,v}}+\|\dx^{\alpha'}f_0\|_{L^{2,q}}),\quad j=0,1,2,3,4, \label{V_1}
\\
\|  \P_1 (\da_x e^{t\BB_0}f_0)\|_{L^2_{x,v}}
&\leq
 C(1+t)^{-\frac 32\(\frac1q-\frac12\)-\frac{k+1}{2}}
  (\|\da_x f_0\|_{L^2_{x,v}}+\|\dx^{\alpha'}f_0\|_{L^{2,q}}). \label{V_2}
 \ema
In addition, assume that $f_0\in L^2_v(H^N_x)\cap L^{2,1}$ for $N\ge 1$ and there exist positive constants $d_0,d_1>0$ and a small constant $r_0>0$ so that the Fourier transform $\hat{f}_0(\xi,v)$ of the initial data $f_0$ satisfies that $\inf_{|\xi|\le r_0}|(\hat f_0,\sqrt M)|\ge d_0$, $\inf_{|\xi|\le r_0}|(\hat f_0,\chi_4)|\ge d_1\sup_{|\xi|\le r_0}|(\hat f_0,\sqrt M)|$ and $\sup_{|\xi|\le r_0}|(\hat f_0,v\sqrt M)|=0$. Then there exist
 two positive constants $C_2\ge C_1$ such that
the  global solution $f(x,v,t)=e^{t\BB_0}f_0(x,v)$  satisfies
\bma
  C_1(1+t)^{-\frac34}
 &\leq \|(e^{t\BB_0}f_0,\chi_j)\|_{L^2_x} \leq C_2(1+t)^{-\frac34},\quad j=0,1,2,3,4, \label{H_1}
\\
 C_1(1+t)^{-\frac54}
 &\leq \|  \P_1 (e^{t\BB_0}f_0)\|_{L^2_{x,v}}\leq C_2(1+t)^{-\frac54}, \label{H_2}
 \ema
for $t>0$ sufficiently large.
\end{thm}

\subsection{Corresponding results  for one-species}


Corresponding to the linearized  two-species VMB, the decomposition
and estimation on the semigroup together with the optimal convergence
rates can be obtained for the one-species VMB based on its
spectrum structure. However, in the following, we can see
that the estimates are very different from the case
of two-species  mainly due to the lack of
cancellation in the one-species model.

\begin{lem}[\cite{Li2}]
 \label{SG_2a}
The operator $B_5(\xi)=L-\i  \P_1(v\cdot\xi) \P_1$ generates a strongly continuous contraction
semigroup on $N_0^\bot$ for any fixed $|\xi|\neq0$, which satisfies for any $t>0$ and $f\in N_0^\bot\cap L^2(\R^3_v)$ that
  \bq
    \|e^{tB_5(\xi)}f\|\leq e^{-\mu t}\|f\|. \label{decay_1a}
 \eq
In addition, for any $x>-\mu $ and $f\in N_0^\bot\cap L^2(\R^3_v)$, it holds
\bq
 \int^{+\infty}_{-\infty}\|[(x+\i y) \P_1-B_5(\xi)]^{-1}f\|^2dy
 \leq
        \pi(x+\mu )^{-1}\|f\|^2.\label{1S_4a}
\eq
\end{lem}

\begin{lem}\label{SG_3a}
The operator $G_6(\xi)$ generates a strongly continuous unitary group on $N_0\times \mathbb{C}^3_\xi\times \mathbb{C}^3_\xi$ for any fixed $|\xi|\neq0$, which satisfies for $t>0$ and $U\in N_0\cap L^2_\xi(\R^3_v)\times \mathbb{C}^3_\xi\times \mathbb{C}^3_\xi$ that
  \bq
   \|e^{tG_6(\xi)}U\|_\xi= \|U\|_\xi.\label{decay_2a}
 \eq
In addition, for any $x\neq 0$ and $f\in N_0\cap L^2_\xi(\R^3_v)$, it holds
\bq
 \int^{+\infty}_{-\infty}\|[(x+\i y) P_A-G_6(\xi)]^{-1}U\|^2_\xi dy
 = \pi|x|^{-1}\|U\|^2_\xi.  \label{1S_6a}
 \eq
\end{lem}
\begin{proof}
Since the operator $\i G_6(\xi)$ is self-adjoint on $N_0\times \mathbb{C}^3_\xi\times \mathbb{C}^3_\xi$ with respect to the inner product $(\cdot,\cdot)_\xi$ defined by \eqref{symmetric}, we can prove \eqref{decay_2a} and \eqref{1S_6a} by applying a similar argument as Lemma 3.2 in \cite{Li2}.
\end{proof}

By a similar argument as for Lemma \ref{F_1}, we have

\begin{lem}\label{F_1a}Given any constant $r_0>0$. Let $\alpha=\alpha(r_0,r_1)>0$ with $r_1>r_0$ and $\alpha(r_0,r_1)$ defined in Lemma \ref{LP01a}. Then
\bma
\sup_{r_0<|\xi|<r_1,y\in \R}\|[I-G_5(\xi)(-\frac{\alpha}2+\i y-G_1(\xi))^{-1}]^{-1}\|&\le C.\label{1S_9a}
\ema
\end{lem}

With the help of Lemmas~\ref{LP01a}--\ref{spectruma} and Lemmas~\ref{SG_2a}--\ref{F_1a}, we have the decomposition of the semigroup $S(t,\xi)=e^{t\hat{\AA}_3(\xi)}$ as following, the detail of the proof is omitted for brevity.

\begin{thm}\label{semigroup}
The semigroup $S(t,\xi)=e^{t\hat{\AA}_3(\xi)}$ with $\xi=s\omega\in \R^3$ and $s=|\xi|\neq0$  has the following decomposition
 \be
 S(t,\xi)U=S_1(t,\xi)U+S_2(t,\xi)U+S_3(t,\xi)U,
     \quad U\in L^2_\xi(\R^3_v)\times \mathbb{C}^3_\xi\times \mathbb{C}^3_\xi, \ \ t>0, \label{E_3a1}
 \ee
 where
 \bma
 S_1(t,\xi)U&=\sum^7_{j=-1}e^{t\lambda_j(s)}
              (U, \Psi^*_j(s,\omega)\,)_\xi \Psi_j(s,\omega)
               1_{\{|\xi|\leq r_0\}}, \label{E_5a1}\\
 S_2(t,\xi)U&=\sum^4_{j=1} e^{t\beta_j(s)}
              (U,\Phi^*_j(s,\omega)\,) \Phi_j(s,\omega)
               1_{\{|\xi|\ge r_1\}},   \label{E_5b1}
 \ema
with $(\lambda_j(s),\Psi_j(s,\omega))$ and $(\beta_j(s),\Phi_j(s,\omega))$ being the eigenvalue and eigenvector of the operator $\hat{\AA}_3(\xi)$ given by Theorem~\ref{eigen_3a} and Theorem~\ref{eigen_4a} for $|\xi|\le r_0$ and $|\xi|>r_1$ respectively,
and $S_3(t,\xi)=: S(t,\xi)-S_1(t,\xi)-S_2(t,\xi)$ satisfies that there
exists a constant $\kappa_0>0$ independent of $\xi$ such that
 \bq
 \|S_3(t,\xi)U\|_\xi\leq Ce^{-\kappa_0t}\|U\|_\xi,\quad t\ge0.\label{B_3a1}
 \eq
\end{thm}

Based on the decomposition of the semigroup given in Theorem \ref{semigroup},
we now study the optimal convergence rates of the solutions
to the linearized system around an equilibrium.

For any $U_0=(f_0,E_0,B_0)\in L^2(\R^3_v;H^l(\R^3_{x}))\times H^l(\R^3_{x})\times H^l(\R^3_{x})$,
define
 \bq
  e^{t\hat{\AA}_2(\xi)}\hat U_0=((e^{t\hat{\AA}_3(\xi)}\hat V_0)_1,-\frac{\i\xi}{|\xi|^2}((e^{t\hat{\AA}_3(\xi)}\hat V_0)_1,\sqrt M)-\frac{\xi}{|\xi|}\times(e^{t\hat{\AA}_3(\xi)}\hat V_0)_2,-\frac{\xi}{|\xi|}\times(e^{t\hat{\AA}_3(\xi)}\hat V_0)_3), \label{solution}
  \eq
  with
\bgrs
e^{t\hat{\AA}_3(\xi)}\hat V_0=((e^{t\hat{\AA}_3(\xi)}\hat V_0)_1,(e^{t\hat{\AA}_3(\xi)}\hat V_0)_2,(e^{t\hat{\AA}_3(\xi)}\hat V_0)_3)\in L^2_\xi(\R^3_v)\times \mathbb{C}^3_\xi\times \mathbb{C}^3_\xi,\\
\hat V_0=(\hat f_0,\frac{\xi}{|\xi|}\times \hat E_0,\frac{\xi}{|\xi|}\times \hat B_0).
\egrs

Then $e^{t\AA_2}U_0$ is the solution of the system \eqref{LVMB}. By Lemma \ref{1SG_1}, it holds that
 $$
 \|e^{t \AA_2} U_0\|_{D^l}=\intr (1+|\xi|^2)^l\|e^{t\hat{\AA}_3(\xi)}\hat V_0\|^2_\xi d\xi\le \intr (1+|\xi|^2)^l\|\hat V_0\|^2_\xi d\xi
=\|U_0\|_{D^l}.
 $$
This implies that the linear operator $\AA_2$ generates a strongly continuous contraction semigroup $e^{t\AA_2}$ in $D^l$, and therefore, $U(t)=e^{t\AA_2}U_0$ is a global solution to the IVP~\eqref{LVMB} for the linearized
one-species Vlasov-Maxwell-Boltzmann equation for $U_0\in D^l$.

 \begin{proof}[\underline{Proof of Theorem \ref{time1a}}]
By  \eqref{solution} and Theorem \ref{semigroup}, we have for $\omega=\xi/|\xi|$ that
\bmas
(\hat f(t), \omega\times\hat E(t),\omega\times\hat B(t))&=e^{t\hat{\AA}_3(\xi)}\hat V_0=S_1(t,\xi)\hat V_0+S_2(t,\xi)\hat V_0+S_3(t,\xi)\hat V_0
\nnm\\
&=\sum^3_{k=1}(h_k(t),H_k(t),J_k(t)),
\emas
with
$\hat V_0=(\hat f_0,\omega\times \hat E_0,\omega\times \hat B_0).$
Note that
 \bma
   \|\dxa (f(t),\chi_j)\|_{L^2_x}
&\le
  \|\xi^\alpha (h_1(t),\chi_j)\|_{L^2_\xi}
 +\|\xi^\alpha (h_2(t),\chi_j)\|_{L^2_\xi}+\|\xi^\alpha (h_3(t),\chi_j)\|_{L^2_\xi}, \label{1D1a}
 \\
 \|\dxa E(t)\|_{L^2_x}
  &\le
    \|\frac{\xi^\alpha}{|\xi|}(h_1(t),\chi_0)\|_{L^2_\xi}
 +\|\frac{\xi^\alpha}{|\xi|} (h_2(t),\chi_0)\|_{L^2_\xi}+\|\frac{\xi^\alpha}{|\xi|} (h_3(t),\chi_0)\|_{L^2_\xi}\nnm\\
 &\quad+\|\xi^\alpha H_1(t)\|_{L^2_\xi}
 +\|\xi^\alpha H_2(t)\|_{L^2_\xi}+\|\xi^\alpha  H_3(t)\|_{L^2_\xi}, \label{1D1b}
 \\
 \|\dxa B(t)\|_{L^2_x}
 &\le
  \|\xi^\alpha J_1(t)\|_{L^2_\xi}
 +\|\xi^\alpha J_2(t)\|_{L^2_\xi}+\|\xi^\alpha  J_3(t)\|_{L^2_\xi}. \label{1D1c}
 \ema
 By \eqref{B_3a1}, 
we can estimate the last term on the right hand side of \eqref{1D1a}--\eqref{1D1c} as follows:
 \bma
&\intr (\xi^{\alpha})^2(\| h_3(t)\|^2_{L^2_v}
 +\frac1{|\xi|^2}|(h_3(t),\sqrt M)|^2 +|H_3(t)|^2+|J_3(t)|^2)d\xi
 \nnm\\
&\leq
 C e^{-2\kappa_0 t}(\|\da_x f_0\|^2_{L^2_{x,v}}+\|\dxa E_0\|_{L^2_x}+\|\dxa B_0\|_{L^2_x}).\label{1D_1}
 \ema

In the low frequency region, by \eqref{E_5a1}, we have
\bmas
S_1(t,\xi)\hat{V}_0=&\sum^7_{j=-1}e^{t\lambda_j(|\xi|)} [(\hat{f}_0,\overline{\psi_j}\,)_\xi-(\omega\times \hat E_0,\overline{X_j})-(\omega\times \hat B_0,\overline{Y_j})](\psi_j,X_j,Y_j) 1_{\{|\xi|\leq r_0\}},
\emas
From \eqref{1eigfr0} and \eqref{1eigf1},
 the macroscopic part and microscopic part of $ h_1(t)$ and $H_1(t),J_1(t)$ satisfy
 \bma
 (h_1(t),\sqrt{M})
=&\frac12\sum_{j=\pm1}e^{\lambda_j(|\xi|)t}\hat{n}_0
  + |\xi|\sum_{j=-1}^1e^{\lambda_j(|\xi|)t}((T_j(\xi)\hat{V}_0)_1,\sqrt{M}),\label{2F_2}
 \\
(h_1(t),v\sqrt{M})
=&\frac12\sum_{j=\pm1}e^{\lambda_j(|\xi|)t}\Big[(\hat{m}_0\cdot\omega)
  -\frac{j}{|\xi|}\hat n_0\Big]\omega+\frac12\sum_{j=2,3}e^{\lambda_j(|\xi|)t}[(\hat{m}_0\cdot W^j)+\i(\hat E_0,W^j)]W^j
  \nnm\\
 &+\frac12\sum_{j=4,5}e^{\lambda_j(|\xi|)t}[(\hat{m}_0\cdot W^j)-\i(\hat E_0,W^j)]W^j+\i|\xi|\sum_{j=6,7}e^{\lambda_j(|\xi|)t}(\omega\times \hat B_0,W^j)W^j\nnm\\
 &+|\xi|\sum_{j=-1}^5e^{\lambda_j(|\xi|)t}((T_j(\xi)\hat{V}_0)_1,v\sqrt{M}) +|\xi|^2\sum_{j=6,7}e^{\lambda_j(|\xi|)t}((T_j(\xi)\hat{V}_0)_1,v\sqrt{M}),\label{2F_3}
  \\
(h_1(t),\chi_4)
 =&\sqrt{\frac16}\sum_{j=\pm1}e^{\lambda_j(|\xi|)t}\hat{n}_0
   +e^{\lambda_0(|\xi|)t}\Big(\hat{q}_0-\sqrt{\frac23}\hat{n}_0\Big)
   +|\xi|\sum_{j=-1}^1 e^{\lambda_j(|\xi|)t}((T_j(\xi)\hat{V}_0)_1,\chi_4),\label{2F_4}
   \\
\P_1h_1(t)=&-\frac12\dsum_{j=\pm1}e^{t\lambda_j(|\xi|)}
    \hat{n}_0j\i(L-j\i \P_1)^{-1}\P_1(v\cdot\omega)^2\sqrt{M}\nnm\\
    &+ |\xi|\sum^5_{j=-1}e^{t\lambda_j(|\xi|)}\P_1(T_j(\xi)\hat{V}_0)_1+|\xi|^2\sum_{j=6,7}e^{t\lambda_j(|\xi|)}\P_1(T_j(\xi)\hat{V}_0)_1,  \label{2B_2}\\
H_1(t)
 =&-\frac12e^{\lambda_2(|\xi|)t}[\i(\omega\times\hat{m}_0)-(\omega\times \hat E_0)]
 +\frac12e^{\lambda_4(|\xi|)t}[\i(\omega\times\hat{m}_0)-(\omega\times \hat E_0)]\nnm\\
 &+|\xi|\sum^5_{j=2}e^{t\lambda_j(|\xi|)}(T_j(\xi)\hat{V}_0)_2+|\xi|^3\sum_{j=6,7}e^{t\lambda_j(|\xi|)}(T_j(\xi)\hat{V}_0)_2,\label{1F_3a}
 \\
J_1(t)
 =&\sum_{j=6,7}e^{\lambda_j(|\xi|)t}(\omega\times\hat B_0,W^j) W^j+\i|\xi|\sum_{j=6,7}e^{\lambda_j(|\xi|)t}(\hat m_0\cdot W^j) W^j\nnm\\
 &+|\xi|\sum_{j=2}^5e^{\lambda_j(|\xi|)t}(T_j(\xi)\hat{V}_0)_3+|\xi|^2\sum_{j=6,7}e^{\lambda_j(|\xi|)t}(T_j(\xi)\hat{V}_0)_3, \label{1F_4a}
 \ema
where $(\hat{n}_0,\hat{m}_0,\hat{q}_0)$ is the Fourier transform of
the macroscopic density, momentum and energy $(n_0,m_0,q_0)$ of the initial data $f_0$ defined by
$
  (n_0,m_0,q_0)=:((f_0,\chi_0),(f_0,v\sqrt{M}),(f_0,\chi_4)),
$
$W^j,\ 2\le j\le 7$ is given by \eqref{1eigf1},  and $T_j(\xi)=((T_j(\xi)\hat{V}_0)_1,(T_j(\xi)\hat{V}_0)_2,(T_j(\xi)\hat{V}_0)_3), \ -1\leq j\leq 7,$ are the  linear operators with the norm $\|T_j(\xi)\|$ being uniformly bounded for $|\xi|\leq r_0$.

Since
 \bmas
 \mathrm{Re}\lambda_j(|\xi|)
 &=a_j|\xi|^2(1+O(|\xi|))
 \le -\eta_1 |\xi|^2,\quad |\xi|\leq r_0,\,\, j=-1,0,1,2,3,4,5,     
 \\
  \mathrm{Re}\lambda_k(|\xi|)
 &=a_k|\xi|^2(1+O(|\xi|))
 \le -\eta_1 |\xi|^4,\quad |\xi|\leq r_0,\,\, k=6,7,
 \emas
where $\eta_1>0$ denotes a generic constant, we obtain by \eqref{1F_2}--\eqref{1F_4a} that
 \bma
 \|\xi^\alpha(h_1(t),\sqrt{M})\|^2_{L^2_\xi}
 &\leq
  C(1+t)^{-(3/2+k)}\|\dx^{\alpha'}U_0\|^2_{Z^1},\label{2D_8}
\\
 \|\xi^\alpha(h_1(t),v\sqrt{M})\|^2_{L^2_\xi}
&\leq
 C(1+t)^{-(1/2+k)}
 \|\dx^{\alpha'}U_0\|^2_{Z^1}+C(1+t)^{-(5/4+k/2)}  \|\dx^{\alpha'}U_0\|^2_{Z^1},\label{2D_9}
 \\
  \|\xi^\alpha(h_1(t),\chi_4)\|^2_{L^2_\xi}
&\leq
 C(1+t)^{-(3/2+k)}
 \|\dx^{\alpha'}U_0\|^2_{Z^1},\label{2D_10}
  \\
    \|\xi^\alpha \P_1h_1(t)\|^2_{L^2_{\xi,v}}
&\leq
 C(1+t)^{-(3/2+k)}\|\dx^{\alpha'}U_0\|^2_{Z^1}+C(1+t)^{-(7/4+k/2)} \|\dx^{\alpha'}U_0\|^2_{Z^1},\label{2D_11}
 \\
  \|\xi^\alpha H_1(t)\|^2_{L^2_\xi}
 &\leq
  C(1+t)^{-(3/2+k)}\|\dx^{\alpha'}U_0\|^2_{Z^1}
+C(1+t)^{-(9/4+k/2)} \|\dx^{\alpha'}U_0\|^2_{Z^1},\label{2D_12}
\\
 \|\xi^\alpha J_1(t)\|^2_{L^2_\xi}
&\leq
  C(1+t)^{-(3/4+k/2)}\|\dx^{\alpha'}U_0\|^2_{Z^1}+C(1+t)^{-(5/2+k)}\|\dx^{\alpha'}U_0\|^2_{Z^1},\label{2D_13}
\ema
with $k=|\alpha-\alpha'|$.

In the high frequency region, by \eqref{E_5b1}, we have
\bma
  S_2(t,\xi)\hat{V}_0=\sum^4_{j=1}e^{t\lambda_j(|\xi|)}[(\hat{f}_0,\overline{\phi_j})-(\omega\times\hat{E}_0,\overline{X_j})-(\omega\times\hat{B}_0,\overline{Y_j})] (\phi_j,X_j,Y_j)1_{\{|\xi|\ge r_1\}}, \label{h1}
\ema and in particular $(h_2(t),\sqrt M)=(h_2(t),\chi_4)=0.$
Since
 $$
 \mathrm{Re}\lambda_j(|\xi|)
  \le -c_1 |\xi|^{-1},\quad |\xi|\ge r_1,
 $$
we obtain by \eqref{h1} that
\bma
 \|\xi^\alpha (h_2(t),v\sqrt M)\|^2_{L^2_{\xi}}&\le
 C(1+t)^{-(2m+1)}(\|\Tdx^m\dxa f_0\|^2_{L^{2}_{x,v}}+\|\Tdx^m\dxa E_0\|^2_{L^2_x}+\|\Tdx^m\dxa B_0\|^2_{L^2_x}),\label{1V_1}
\\
 \|\xi^\alpha \P_1h_2(t)\|^2_{L^2_{\xi,v}}&\le
 C(1+t)^{-(2m+1)}(\|\Tdx^m\dxa f_0\|^2_{L^{2}_{x,v}}+\|\Tdx^m\dxa E_0\|^2_{L^2_x}+\|\Tdx^m\dxa B_0\|^2_{L^2_x}),\label{1V_2}
 \\
 \|\xi^\alpha H_2(t)\|^2_{L^2_\xi}&\le
C(1+t)^{-2m}(\|\Tdx^m\dxa f_0\|^2_{L^{2}_{x,v}}+\|\Tdx^m\dxa E_0\|^2_{L^2_x}+\|\Tdx^m\dxa B_0\|^2_{L^2_x}),\label{1V_3}
\\
  \|\xi^\alpha J_2(t)\|^2_{L^2_\xi}&\le
C(1+t)^{-2m}(\|\Tdx^m\dxa f_0\|^2_{L^{2}_{x,v}}+\|\Tdx^m\dxa E_0\|^2_{L^2_x}+\|\Tdx^m\dxa B_0\|^2_{L^2_x}).\label{1V_4}
\ema
Combining  \eqref{1D1a}--\eqref{1D_1}, \eqref{2D_8}--\eqref{2D_13} and \eqref{1V_1}--\eqref{1V_4} leads to \eqref{2D_3a}--\eqref{2D_5}.

Now we turn to show the lower bound of time decay rates for the global solution under the assumptions of Theorem \ref{time1a}. Note that
\bma
   \| ( f(t),\chi_j)\|_{L^2_{x}}
 &\ge
  \|(h_1(t),\chi_j)\|_{L^2_{\xi}}-C(1+t)^{-1}-Ce^{-Ct}, \label{1H_4}
 \\
 \| E(t)\|_{L^2_x}
 &\ge
  \frac{\sqrt2}2(\| \frac1{|\xi|}(h_1(t),\sqrt M)\|_{L^2_\xi}+\|H_1(t)\|_{L^2_\xi})-C(1+t)^{-1}-Ce^{-Ct}, \label{1H_4a}
 \\
 \|B(t)\|_{L^2_x}
 &\ge
  \| J_1(t)\|_{L^2_\xi}-C(1+t)^{-1}-Ce^{-Ct}, \label{1Q_4}
 \ema
where we have used \eqref{1D_1} and \eqref{1V_1}--\eqref{1V_4} for $|\alpha|=0$.

By \eqref{2F_2} and
that fact that $\lambda_{-1}(|\xi|)=\overline{\lambda_1(|\xi|)}$, we have
$$
(h_1(t),\sqrt{M})=e^{\mathrm{Re}\lambda_1(|\xi|)t}\cos(\mathrm{Im}\lambda_1(|\xi|)t)\hat{n}_0+|\xi|T_8(t,\xi)\hat V_0,
$$
where $ T_8(t,\xi)\hat V_0$ is the remainder term satisfying $\|T_8(t,\xi)\hat V_0\|\le Ce^{-\eta_1 |\xi|^2t}\|\hat V_0\|$. This leads to
 \bma
|(h_1(t),\sqrt{M})|^2
 \geq
   \frac12 e^{2\mathrm{Re}\lambda_1(|\xi|)t}\cos^2(\mathrm{Im}\lambda_1(|\xi|)t)|\hat{n}_0 |^2-C|\xi|^2e^{-2\eta_1 |\xi|^2t}\|\hat{V}_0\|^2.\label{1B_1}
 \ema
Since
$$ \cos^2(\mathrm{Im}\lambda_1(|\xi|)t)\geq
\frac12\cos^2[(1+b_1|\xi|^2)t]-O([|\xi|^3t]^2),
$$
and
$$\mathrm{Re}\lambda_j(|\xi|)=a_j|\xi|^2(1+O(|\xi|))\ge -\eta_2 |\xi|^2, \quad |\xi|\leq r_0,\,\, j=-1,0,1,2,3,4,5,
$$
for some constant $\eta>0$, we obtain by \eqref{1B_1} that \bma
\| (h_1(t),\sqrt{M})\|^2_{L^2_{\xi}}
&\geq
 \frac{d_0^2}4\int_{|\xi|\leq r_0} e^{-2\eta_2 |\xi|^2t}\cos^2(t+b_1|\xi|^2t)d\xi
 -C (1+t)^{-5/2}\nnm\\
&=:I_1- C(1+t)^{-5/2}.    \label{2E_7a}
 \ema
Since it holds for $t \geq t_0=:\frac{L^2}{r_0^2}$ with the constant $L\ge\sqrt{\frac{4\pi}{b_1}}$ that
 \bma
 I_1
  &=\frac{d_0^2}4t^{-3/2}\int_{|\zeta|\leq r_0\sqrt{t}}
   e^{-2\eta|\zeta|^2}\cos^2(t+b_1|\zeta|^2)d\zeta
 \geq \pi d_0^2(1+t)^{-3/2}\int^L_0 r^{2}e^{-2\eta_2 r^2}\cos^2(t+b_1r^2)dr
 \nnm\\&
  \geq
   (1+t)^{-3/2}\frac{\pi d_0^2L}{2} e^{-2\eta L^2 }\int^{L}_{L/2}r\cos^2(t+b_1r^2)dr
  =  (1+t)^{-3/2}\frac{\pi d_0^2L}{4b_1}e^{-2\eta L^2 }
     \int^{t+b_1L^2}_{t+\frac {b_1L^2}4}\cos^2ydy\nnm\\
 &\geq
    (1+t)^{-3/2}\frac{\pi d_0^2L}{4b_1}e^{-2\eta L^2 }\int^\pi_0 \cos^2ydy
 \geq C_3(1+t)^{-3/2},   \label{1E_9b}
 \ema
where $C_3> 0$ denotes a generic positive constant. We can  substitute \eqref{2E_7a} and \eqref{1E_9b} into  \eqref{1H_4} with $j=0$  to obtain \eqref{2H_1}.

By \eqref{2F_3},  we have
 \bmas
 |(h_1(t),v\sqrt{M})|^2 \geq
 \frac1{2|\xi|^2}e^{2\mathrm{Re}\lambda_1(|\xi|)t}
  \sin^2(\mathrm{Im}\lambda_1(|\xi|)t)|\hat{n}_0|^2
 -Ce^{-2\eta_1 |\xi|^2t}\|\hat{V}_0\|^2 -C|\xi|^2e^{-2\eta_1 |\xi|^4t}\|\hat{V}_0\|^2.
 \emas
Then
 \bmas
 \|(h_1(t),v\sqrt{M})\|^2_{L^2_\xi}
\geq &
 \frac{d_0^2}4\int_{|\xi|\le r_0}\frac1{|\xi|^2}
   e^{-2\eta_2|\xi|^2t}\sin^2(t+b_1|\xi|^2t)d\xi
\nnm\\
&-C\int_{|\xi|\leq r_0} e^{-2\eta_1 |\xi|^2 t}\|\hat{V}_0\|^2d\xi-C\int_{|\xi|\leq r_0} |\xi|^2e^{-2\eta_1 |\xi|^4 t}
  \|\hat{V}_0\|^2d\xi
\nnm\\
\geq &  C_3(1+t)^{-1/2}- C(1+t)^{-3/2}- C(1+t)^{-5/4},
\emas
which together with \eqref{1H_4} for $j=1,2,3$ lead to \eqref{2H_2} for $t>0$ being large enough.

By \eqref{2F_4} and   the fact that
$\lambda_0(|\xi|)$ is real, we have
 \bmas
 |(h_1(t),\chi_4)|^2\geq \frac12e^{2\lambda_0(|\xi|)t}|\hat{q}_0|^2
 -Ce^{-2\eta_1 |\xi|^2t}(|\hat{n}_0|^2+|\xi|^2\|\hat{V}_0\|^2),
\emas
which leads to
 \bmas
 \|(h_1(t),\chi_4)\|^2_{L^2_\xi} &\geq
 \frac12\int_{|\xi|\le r_0}e^{-2\eta_2 |\xi|^2t}|\hat{q}_0|^2d\xi
  -C\int_{|\xi|\leq r_0}e^{-2\eta_1 |\xi|^2 t}
   (|\hat{n}_0|^2+|\xi|^2\|\hat{V}_0\|^2)d\xi\\
   &\ge C_3\big[
     \inf_{ |\xi|\le r_0} |\hat{q}_0(\xi)|^2(1+t)^{-3/2}
    -d_1\sup_{ |\xi|\le r_0} |\hat{n}_0(\xi)|(1+t)^{-3/2}\big]
   -C(1+t)^{-5/2}.
 \emas
This and \eqref{1H_4} with $j=4$  lead to  \eqref{2H_3} for $t>0$  being
large enough.

By \eqref{2B_2}, we have
\bmas
 \|\P_1h_1(t)\|^2_{L^2_v}
 \ge &
  \frac12|\hat n_0|^2 e^{2{\rm Re}\lambda_1(|\xi|)t}
   \| \sin({\rm Im}\lambda_1(|\xi|)t)L\Psi
     +\cos({\rm Im}\lambda_1(|\xi|)t)\Psi\|^2_{L^2_{v}} \\
  & -C|\xi|^2e^{-2\eta_1 |\xi|^2t}\|\hat V_0\|^2-C|\xi|^4e^{-2\eta_1 |\xi|^4t}\|\hat V_0\|^2,
\emas
where $\Psi\in N_0^\bot$ is a non-zero real function given by
$$
\Psi=(L-\i\P_1)^{-1}(L+\i\P_1)^{-1}\P_1(v\cdot\omega)^2\sqrt M\neq 0.
$$
Then
 \bma
 \| \P_1h_1(t)\|^2_{L^2_{\xi,v}}
\ge &
  \frac{d_0^2}4 \int_{|\xi|\le r_0}e^{-2\eta_2 |\xi|^2t}\|\sin(t+b_1|\xi|^2t)L\Psi+\cos(t+b_1|\xi|^2t)\Psi\|^2_{L^2_{v}}d \xi
  \nnm\\
 & -C (1+t)^{-5/2}-C (1+t)^{-7/4}\nnm\\
=&:I_2 -C (1+t)^{-5/2}-C (1+t)^{-7/4}.  \label{I6}
\ema
Since it holds for time $t \geq t_0=:\frac{L^2}{r_0^2}$ with  $L\ge\sqrt{\frac{4\pi}{b_1}}$ that
\bma
I_2&=\frac{d_0^2}4t^{-3/2}\int_{|\zeta|\le r_0\sqrt t}e^{-2\eta |\zeta|^2}\|\sin(t+b_1|\zeta|^2)L\Psi+\cos(t+b_1|\zeta|^2)\Psi\|^2_{L^2(\R^3_{v})}d\zeta
 \nnm\\
&\ge \pi d_0^2(1+t)^{-3/2}\int^L_0 r^2e^{-2\eta r^2}
  \|\sin(t+b_1r^2)L\Psi+\cos(t+b_1r^2)\Psi\|^2_{L^2(\R^3_{v})}dr
 \nnm\\
 &\ge
    \frac{L\pi d_0^2}{4b_1}(1+t)^{-3/2}  e^{-2\eta L^2}\|\Psi\|^2_{L^2_v}
    \int^\pi_{\frac\pi2}\cos^2 y d y
    \ge C_3 (1+t)^{-3/2} ,       \label{I7}
\ema
we can substitute \eqref{I7} and \eqref{I6} into  \eqref{1H_4} imply \eqref{2Q_2a} for sufficiently large $t>0$.

By \eqref{1F_3a}, we obtain
\bmas
\frac1{|\xi|^2}|(h_1(t),\sqrt{M})|^2+|H_1(t)|^2
 \geq
   \frac1{2|\xi|^2} e^{2\mathrm{Re}\lambda_1(|\xi|)t}\cos^2(\mathrm{Im}\lambda_1(|\xi|)t)|\hat{n}_0 |^2-C e^{-2\eta_1 |\xi|^2t}\|\hat{V}_0\|^2-C|\xi|^6e^{-2\eta_1 |\xi|^4t}\|\hat{V}_0\|^2,
 \emas
which leads to
\bmas
\|\frac1{|\xi|}(h_1(t),\sqrt{M})\|^2_{L^2_\xi}+\| H_1(t)\|^2_{L^2_\xi}
\ge  C_3(1+t)^{-1/2} -C(1+t)^{-3/2}-C(1+t)^{-9/4}.
\emas
This together with \eqref{1H_4a} lead to \eqref{2H_2a} for sufficiently large $t>0$.

Finally, by  \eqref{1F_4a} and
the fact that $\lambda_6(|\xi|)=\lambda_7(|\xi|)$ are real, we obtain
\bmas
|J_1(t)|^2
 \ge\frac12e^{2 \lambda_6(|\xi|)t} |\omega\times\hat{B}_0|^2-C|\xi|^4e^{-2\eta_1 |\xi|^4t}\|\hat V_0\|^2 -C|\xi|^2e^{-2\eta_1 |\xi|^2t}\|\hat V_0\|^2.
 \emas
 Since
$$\mathrm{Re}\lambda_j(|\xi|)=a_j|\xi|^4(1+O(|\xi|))\ge -\eta_2 |\xi|^4, \quad |\xi|\leq r_0,\,\, j=6,7,
$$
for some constant $\eta_2>0$, we have
\bmas
 \| J_1(t)\|^2_{L^2_\xi}\geq &
 \frac{d_0^2}2\int_{|\xi|\le r_0}|\xi|^2  e^{-2\eta_2|\xi|^4t}d\xi-C\int_{|\xi|\le r_0}|\xi|^4  e^{-2\eta_1 |\xi|^4t}\|\hat{V}_0\|^2d\xi
\nnm\\
&-C\int_{|\xi|\leq r_0}|\xi|^2 e^{-2\eta_1  |\xi|^2 t} \|\hat{V}_0\|^2d\xi
\nnm\\
\ge&  C_3(1+t)^{-3/4} -C(1+t)^{-7/4}-C(1+t)^{-5/2}.
\emas
This together with  \eqref{1Q_4} lead to  \eqref{2H_3a}.  The proof is then completed.
\end{proof}

 \begin{proof}[\underline{Proof of Theorem \ref{time2a}}]

In the case of $\Tdx\cdot E_0=(f_0,\sqrt M)$, the high frequency term $S_2(t,\xi)\hat V_0$  are same as those in the case of $\Tdx\cdot E_0\ne (f_0,\sqrt M)$. The different part lie in the expansions of the low frequency term $S_1(t,\xi)\hat V_0$. That is, in the low frequency region, by \eqref{E_5a1}, we have
\bmas
S_1(t,\xi)\hat{V}_0=&\sum^7_{j=-1}e^{t\lambda_j(|\xi|)} [(\hat{f}_0,\overline{\psi_j}\,)+\frac1{|\xi|^2}\i(\hat E_0\cdot\xi)(\psi_j,\sqrt M)-(\omega\times \hat E_0,\overline{X_j})-(\omega\times \hat B_0,\overline{Y_j})]\\
&\qquad\quad\times(\psi_j,X_j,Y_j) 1_{\{|\xi|\leq r_0\}},
\emas
 where  we have used $(\hat f_0,\sqrt M)=\i(\hat E_0\cdot \xi)$ to replace $(\hat f_0,\sqrt M)$.
From \eqref{1eigfr0} and \eqref{1eigf1},
 the macroscopic part and microscopic part of $ h_1(t)$ and $h_2(t),h_3(t)$ satisfy
 \bma
 (h_1(t),\sqrt{M})
=&|\xi|\sum_{j=-1}^1e^{\lambda_j(|\xi|)t}((T_j(\xi)\hat{V}_0)_1,\sqrt{M}),\label{1F_2}
 \\
(h_1(t),v\sqrt{M})
=&\frac12\sum_{j=\pm1}e^{\lambda_j(|\xi|)t}[(\hat{m}_0\cdot\omega)
  -j\i(\hat{E}_0\cdot\omega)]\omega+\frac12\sum_{j=2,3}e^{\lambda_j(|\xi|)t}[(\hat{m}_0\cdot W^j)+\i(\hat E_0,W^j)]W^j
  \nnm\\
 &+\frac12\sum_{j=4,5}e^{\lambda_j(|\xi|)t}[(\hat{m}_0\cdot W^j)-\i(\hat E_0,W^j)]W^j+\i|\xi|\sum_{j=6,7}e^{\lambda_j(|\xi|)t}(\omega\times \hat B_0,W^j)W^j\nnm\\
 &+|\xi|\sum_{j=-1}^5e^{\lambda_j(|\xi|)t}((T_j(\xi)\hat{V}_0)_1,v\sqrt{M}) +|\xi|^2\sum_{j=6,7}e^{\lambda_j(|\xi|)t}((T_j(\xi)\hat{V}_0)_1,v\sqrt{M}),\label{1F_3}
  \\
(h_1(t),\chi_4)
 =&e^{\lambda_0(|\xi|)t}\hat{q}_0
   +|\xi|\sum_{j=-1}^1 e^{\lambda_j(|\xi|)t}((T_j(\xi)\hat{V}_0)_1,\chi_4),\label{1F_4}
   \\
\P_1h_1(t)=& |\xi|\sum^5_{j=-1}e^{t\lambda_j(|\xi|)}\P_1(T_j(\xi)\hat{V}_0)_1+|\xi|^2\sum_{j=6,7}e^{t\lambda_j(|\xi|)}\P_1(T_j(\xi)\hat{V}_0)_1,  \label{1B_2}
 \ema
and $H_1(t),J_1(t)$ are same as \eqref{1F_3a} and \eqref{1F_4a}.
Thus we obtain by \eqref{1F_2}--\eqref{1F_4a} that
 \bma
 \|\xi^\alpha(h_1(t),\sqrt{M})\|^2_{L^2_\xi}
 &\leq
  C(1+t)^{-(5/2+k)}\|\dx^{\alpha'}U_0\|^2_{Z^1},\label{D_8}
\\
 \|\xi^\alpha(h_1(t),v\sqrt{M})\|^2_{L^2_\xi}
&\leq
 C(1+t)^{-(3/2+k)}
 \|\dx^{\alpha'}U_0\|^2_{Z^1})
+C(1+t)^{-(5/4+k/2)} \|\dx^{\alpha'}U_0\|^2_{Z^1},\label{D_9}
 \\
  \|\xi^\alpha(h_1(t),\chi_4)\|^2_{L^2_\xi}
&\leq
 C(1+t)^{-(3/2+k)}
\|\dx^{\alpha'}U_0\|^2_{Z^1},\label{D_10}
  \\
    \|\xi^\alpha \P_1h_1(t)\|^2_{L^2_{\xi,v}}
&\leq
 C(1+t)^{-(5/2+k)}\|\dx^{\alpha'}U_0\|^2_{Z^1}+C(1+t)^{-(7/4+k/2)} \|\dx^{\alpha'}U_0\|^2_{Z^1},\label{D_11}
 \\
  \|\xi^\alpha H_1(t)\|^2_{L^2_\xi}
 &\leq
  C(1+t)^{-(3/2+k)}\|\dx^{\alpha'}U_0\|^2_{Z^1}
+C(1+t)^{-(9/4+k/2)} \|\dx^{\alpha'}U_0\|^2_{Z^1},\label{D_12}
\\
 \|\xi^\alpha J_1(t)\|^2_{L^2_\xi}
&\leq
  C(1+t)^{-(3/4+k/2)}\|\dx^{\alpha'}U_0\|^2_{Z^1} +C(1+t)^{-(5/2+k)}\|\dx^{\alpha'}U_0\|^2_{Z^1},\label{D_13}
\ema
with $k=|\alpha-\alpha'|$.
Thus by \eqref{D_8}--\eqref{D_13} and \eqref{1V_1}--\eqref{1V_4}, we can prove \eqref{1D_2}--\eqref{1D_0}.

Now we turn to show the lower bound of time decay rates for the global solution under the assumptions of Theorem \ref{time1a}.
By \eqref{1F_2}, we have
$$
(h_1(t),\sqrt{M})=- |\xi|e^{{\rm Re}\lambda_1(|\xi|)t}\sin({\rm Im}\lambda_1(|\xi|)t)(\hat E_0\cdot\omega)+|\xi|^2T_8(t,\xi)\hat V_0,
$$
where $ T_8(t,\xi)\hat V_0$ is the remainder term satisfying $\|T_8(t,\xi)\hat V_0\|\le Ce^{-\eta_1 |\xi|^2t}\|\hat V_0\|$. This leads to
 $$
|(h_1(t),\sqrt{M})|^2
 \geq
   \frac12|\xi|^2e^{2\mathrm{Re}\lambda_1(|\xi|)t}\sin^2(\mathrm{Im}\lambda_1(|\xi|)t)|\hat{E}_0\cdot\omega|^2-C|\xi|^4e^{-2\eta_1 |\xi|^2t}\|\hat{V}_0\|^2.
 $$
It follows that \bma
\| (h_1(t),\sqrt{M})\|^2_{L^2_{\xi}}
&\geq
 \frac{d_0^2}4\int_{|\xi|\leq r_0}|\xi|^2 e^{-2\eta_2 |\xi|^2t}\sin^2(t+b_1|\xi|^2t)d\xi
 -C (1+t)^{-7/2}\nnm\\
&\ge C_3(1+t)^{-5/2}- C(1+t)^{-7/2}.    \label{1E_7a}
 \ema
We can  substitute \eqref{1E_7a}  into  \eqref{1H_4} with $j=0$  to obtain \eqref{1H_1}.

By \eqref{1F_3} and \eqref{1F_4},  we have
 \bmas
 |(h_1(t),v\sqrt{M})|^2 &\geq
\frac12|\xi|^2e^{2\lambda_6(|\xi|)t}|\omega\times\hat{B}_0|^2
  -Ce^{-2\eta_1 |\xi|^2t}(|\hat{E}_0|^2+|\xi|^2\|\hat{V}_0\|^2)-C|\xi|^4e^{-2\eta_1 |\xi|^4t}\|\hat{V}_0\|^2,\\
   |(h_1(t),\chi_4)|^2&\geq  \frac12e^{2\lambda_0(|\xi|)t}|\hat{q}_0|^2-C|\xi|^2e^{-2\eta_1 |\xi|^2t}\|\hat{V}_0\|^2,
 \emas
which leads to
 \bmas
 \|(h_1(t),v\sqrt{M})\|^2_{L^2_\xi}
 &\geq  C_3(1+t)^{-5/4}-C(1+t)^{-7/4}- C(1+t)^{-3/2},\\
\|(h_1(t),\chi_4)\|^2_{L^2_\xi}
   &\ge C_3(1+t)^{-3/2}-C(1+t)^{-5/2}.
\emas
This together with \eqref{1H_4} lead to \eqref{1H_2} and \eqref{1H_3} for $t>0$ being large enough.

By \eqref{1B_2}, we have
\bmas
    \P_1h_1(t)
 = &|\xi|^2\sum_{j=6,7}e^{\lambda_j(|\xi|)t}(\omega\times \hat B_0,W^j)L^{-1}\P_1(v\cdot\omega)(v\cdot W^j)\sqrt M\nnm\\
   &+|\xi|^3T_9(t,\xi)\hat V_0+|\xi|T_{10}(t,\xi)\hat V_0,
\emas
where $ T_9(t,\xi)$ and $ T_{10}(t,\xi)$ are the remainder terms satisfying $\|T_9(t,\xi)\hat V_0\|\le Ce^{-\eta_1 |\xi|^4t}\|\hat V_0\|$ and $\|T_{10}(t,\xi)\hat V_0\|\le Ce^{-\eta_1 |\xi|^2t}\|\hat V_0\|$.
Then
 \bmas
 \| \P_1h_1(t)\|^2_{L^2_v}
 \ge &
  \frac12\| L^{-1}\P_1(v_1\chi_2)\|^2_{L^2_v} |\xi|^4e^{2 \lambda_6(|\xi|)t}|\omega\times \hat B_0|^2 -C|\xi|^6e^{-2\eta_1 |\xi|^4t}\|\hat V_0\|^2
\nnm\\
  & -C|\xi|^2e^{-2\eta_1 |\xi|^2t}\|\hat V_0\|^2.
\emas
This leads to
\bmas
  \| \P_1h_1(t)\|^2_{L^2_{\xi,v}}
\ge  C_3(1+t)^{-7/4} -C (1+t)^{-9/4} -C (1+t)^{-5/2},  
\emas
which together with  \eqref{1H_4} imply \eqref{1Q_2a} for sufficiently large $t>0$.

Finally, by \eqref{1F_3a} and \eqref{1F_4a} we obtain
\bmas
\frac1{|\xi|^2}|(h_1(t),\sqrt{M})|^2+|H_1(t)|^2
 &\ge \frac12e^{2\mathrm{Re}\lambda_1(|\xi|)t}\sin^2(\mathrm{Im}\lambda_1(|\xi|)t)|\hat{E}_0\cdot\omega|^2 +\frac12e^{2\mathrm{Re}\lambda_4(|\xi|)t}\cos^2(\mathrm{Im}\lambda_4(|\xi|)t)|\omega\times \hat E_0|^2\nnm\\
 &\quad-C|\xi|^2e^{-2\eta_1 |\xi|^2t}\|\hat V_0\|^2-C|\xi|^6e^{-2\eta_1 |\xi|^4t}\|\hat V_0\|^2,\\
|J_1(t)|^2
 &\ge\frac12e^{2 \lambda_6(|\xi|)t} |\omega\times\hat{B}_0|^2-C|\xi|^4e^{-2\eta_1 |\xi|^4t}\|\hat V_0\|^2-C|\xi|^2e^{-2\eta_1 |\xi|^2t}\|\hat V_0\|^2,
 \emas
which lead to
\bmas
  \|\frac1{|\xi|}(h_1(t),\sqrt{M})\|^2_{L^2_\xi}+\| H_1(t)\|^2_{L^2_\xi}
&\ge  C(1+t)^{-3/2} -C(1+t)^{-5/2}-C(1+t)^{-9/4},
\\
 \| J_1(t)\|^2_{L^2_\xi}
&\ge  C(1+t)^{-3/4} -C(1+t)^{-7/4}-C(1+t)^{-5/2}.
\emas
This together with \eqref{1H_4a} and \eqref{1Q_4} lead to \eqref{1H_2a} and \eqref{1H_3a}.  The proof is then completed.
\end{proof}

\section{The  nonlinear system}
\label{behavior-nonlinear}
 \setcounter{equation}{0}
In this section, we prove the large time decay rates of the solution to the Cauchy problem for Vlasov-Maxwell-Boltzmann systems with the estimates on the linearized problem obtained in Section~\ref{behavior-linear}.

\subsection{Energy estimates for two species}
We first obtain some energy estimates.
Let $N$ be a positive integer and $U=(f_1,f_2,E,B)$, and
\bma
E_{N,k}(U)&=\sum_{|\alpha|+|\beta|\le N}\|w^k\dxa\dvb (f_1,f_2)\|^2_{L^2_{x,v}}+\sum_{|\alpha|\le N}\|\dxa(E,B)\|^2_{L^2_x},\label{energy3}\\
H_{N,k}(U)&= \sum_{|\alpha|+|\beta|\le N}\|w^k\dxa\dvb
(\P_1f_1,P_rf_2)\|^2_{L^2_{x,v}}+\sum_{1\le|\alpha|\le N}\|\dxa (E,B)\|^2_{L^2_x}+\|E\|^2_{L^2_x}\nnm\\
&\quad +\sum_{|\alpha|\le N-1}\|\dxa\Tdx  (\P_0f_1,P_{\rm d}f_2)\|^2_{L^2_{x,v}}+\| P_{\rm d}f_2\|^2_{L^2_{x,v}},\\
D_{N,k}(U)&=\sum_{|\alpha|+|\beta|\le N}\|w^{\frac12+k}\dxa\dvb  (\P_1f_1,P_rf_2)\|^2_{L^2_{x,v}}+\sum_{1\le |\alpha|\le N-1}\|\dxa  B\|^2_{L^2_x}\nnm\\
&\quad+\sum_{|\alpha|\le N-1}(\|\dxa\Tdx  (\P_0f_1,P_{\rm d}f_2)\|^2_{L^2_{x,v}}+\|\dxa E\|^2_{L^2_x})+\| P_{\rm d}f_2\|^2_{L^2_{x,v}},
\ema
for $k\ge 0$. For brevity, we write $E_N(U)=E_{N,0}(U)$, $H_N(U)=H_{N,0}(U)$ and $D_N(U)=D_{N,0}(U)$ for $k=0$.

Firstly, by taking the inner product between $\chi_j\ (j=0,1,2,3,4)$ and \eqref{VMB3}, we obtain  a  compressible Euler-Maxwell type system
\bma \dt n_1+\divx  m_1&=0,\label{G_3}\\
\dt  m_1+\Tdx n_1+\sqrt{\frac23}\Tdx q_1&=n_2 E+m_2\times B-\intr v\cdot\Tdx(  \P_1f_1) v\sqrt Mdv,\label{G_5}\\
\dt q_1+\sqrt{\frac23}\divx m_1&=\sqrt{\frac23} E\cdot m_2-\intr v\cdot\Tdx(  \P_1f_1) \chi_4 dv, \label{G_6}
\ema
where
$$(n_1,m_1,q_1)=((f_1,\sqrt M),(f_1,v\sqrt M),(f_1,\chi_4)),\quad (n_2,m_2)=((f_2,\sqrt M),(f_2,v\sqrt M)).$$
Taking the microscopic projection $  \P_1$ on \eqref{VMB3}, we have
\bma
\dt(  \P_1f_1)+  \P_1(v\cdot\Tdx   \P_1f_1)-L(  \P_1f_1)&=-  \P_1(v\cdot\Tdx  \P_0f_1)+  \P_1 G_1 ,\label{GG1}
\ema
where the nonlinear term $G_1$ is denoted by
\bq G_1=\frac12 (v\cdot E)f_2-(E+v\times B)\cdot\Tdv f_2+\Gamma(f_1,f_1).\label{G1}\eq
By \eqref{GG1}, we can express the microscopic part $  \P_1f_1$ as
\bq   \P_1f_1=L^{-1}[\dt(  \P_1f_1)+  \P_1(v\cdot\Tdx   \P_1f_1)-  \P_1 G_1]+L^{-1}  \P_1(v\cdot\Tdx  \P_0f_1). \label{p_1}\eq
Substituting \eqref{p_1} into \eqref{G_3}--\eqref{G_6}, we obtain
a compressible Navier-Stokes-Maxwell type system
\bma
\dt n_1+\divx  m_1&=0,\label{G_9}\\
\dt  m_1+\dt R_1+\Tdx n_1+\sqrt{\frac23}\Tdx q_1&=\kappa_1 (\Delta_x m_1+\frac13\Tdx{\rm div}_x m_1)+n_2 E+m_2\times B+R_2,\label{G_7}\\
\dt q_1+\dt R_3+\sqrt{\frac23}\divx m_1&=\kappa_2 \Delta_x q_1+\sqrt{\frac23} E\cdot m_2+R_4,\label{G_8}
\ema
where the viscosity and heat conductivity
 coefficients $\kappa_1,\kappa_2>0$ and the remainder terms $R_1, R_2, R_3, R_4$ are given by
\bmas
\kappa_1&=-(L^{-1}  \P_1(v_1\chi_2),v_1\chi_2),\quad \kappa_2=-(L^{-1}  \P_1(v_1\chi_4),v_1\chi_4),\\
R_1&=( v\cdot\Tdx L^{-1} \P_1f_1,v\sqrt M),\quad R_2=-(v\cdot\Tdx L^{-1}(  \P_1(v\cdot\Tdx \P_1f_1)-  \P_1 G_1),v\sqrt M),\\
R_3&=( v\cdot\Tdx L^{-1} \P_1f_1,\chi_4),\quad R_4=-(v\cdot\Tdx L^{-1}(  \P_1(v\cdot\Tdx \P_1f_1)-  \P_1 G_1),\chi_4).
\emas

By taking the inner product between $\sqrt M$ and \eqref{VMB3a}, we obtain
\bma
\dt n_2+\divx m_2&=0. \label{G_3a}
\ema
Taking the microscopic projection $ P_r$ on \eqref{VMB3a}, we have
\bma
\dt( P_rf_2)+ P_r(v\cdot\Tdx  P_rf_2)-v\sqrt M\cdot E-L_1( P_rf_2)&=- P_r(v\cdot\Tdx  P_{\rm d}f_2)+ P_r G_2,\label{GG2}
\ema
where the nonlinear term $G_2$ is denoted by
\bq G_2=\frac12 (v\cdot E)f_1-(E+v\times B)\cdot\Tdv f_1+\Gamma(f_2,f_1).\label{G2}\eq
By \eqref{GG2}, we can express the microscopic part $ P_rf_2$ as
\bq   P_rf_2=L_1^{-1}[\dt( P_rf_2)+ P_r(v\cdot\Tdx  P_rf_2)- P_r G_2]+L_1^{-1} P_r(v\cdot\Tdx  P_{\rm d}f_2)-L_1^{-1} (v\sqrt M\cdot E). \label{p_c}\eq
Substituting \eqref{p_c} into \eqref{G_3a} and \eqref{VMB3b}, we obtain
\bma
\dt n_2+\dt\divx R_5&=-\kappa_3 n_2+\kappa_3 \Delta_x n_2-\divx R_6,\label{G_9a}\\
\dt E+\dt R_5&=\Tdx\times B+\kappa_3\Tdx n_2-\kappa_3E+R_6,\label{G_9b}\\
\dt B&=-\Tdx\times E,\label{G_9c}
\ema
where the viscosity coefficient $\kappa_3>0$ and the remainder terms $R_5, R_6$ are defined by
\bmas \kappa_3&=-(L_1^{-1}\chi_1,\chi_1),\quad
R_5=(L_1^{-1}P_r f_2,v\sqrt M), \\
R_6&=(L_1^{-1}(P_r(v\cdot\Tdx P_rf_2)-P_rG_2),v\sqrt M).\emas

The following  lemma is from \cite{Duan2,Guo2}.

\begin{lem}[\cite{Duan2,Guo2}]\label{e1}It holds that
\bmas
\|\nu^{k}\dvb\Gamma(f,g)\|_{L^2_v}&\le C\sum_{\beta_1+\beta_2\le\beta}
(\|\dv^{\beta_1}f\|_{L^2_v}\|\nu^{k+1}\dv^{\beta_2}g\|_{L^2_v}+\|\nu^{k+1}\dv^{\beta_1}f\|_{L^2_v}\|\dv^{\beta_2}g\|_{L^2_v}),
\emas for $k\ge -1$, and
$$ \|\Gamma(f,g)\|_{L^{2,1}}\le C(\|f\|_{L^2_{x,v}}\|\nu g\|_{L^2_{x,v}}+\|\nu
f\|_{L^2_{x,v}}\|g\|_{L^2_{x,v}}). $$
\end{lem}

\begin{lem}[Macroscopic dissipation] \label{macro-en} Let $(n_1,m_1,q_1)$ and $(n_2,E,B)$ be the strong solutions to \eqref{G_9}--\eqref{G_8} and \eqref{G_9a}--\eqref{G_9c} respectively. Then, there are two constants $s_0,s_1>0$ such that
\bma
&\Dt \sum_{k\le |\alpha|\le N-1}s_0(\|\dxa(n_1, m_1,q_1)\|^2_{L^2_x}+2\intr \dxa R_1\dxa m_1dx+2\intr \dxa R_3\dxa q_1dx)\nnm\\
&+\Dt \sum_{k\le |\alpha|\le N-1}4\intr \dxa m_1 \dxa\Tdx n_1dx+\sum_{k\le |\alpha|\le N-1} \|\dxa\Tdx (n_1, m_1,q_1)\|^2_{L^2_x}
\nnm\\
\le & C\sqrt{E_N(U)}D_N(U)+C\sum_{k\le |\alpha|\le N-1}\|\dxa\Tdx \P_1f_1\|^2_{L^2_{x,v}},\label{E_1}
\\
&\Dt \sum_{k\le |\alpha|\le N-1}s_1(\|\dxa (n_2,E,B)\|^2_{L^2_x}+2 \intr\dxa\divx R_5\dxa n_2 dx+2 \intr\dxa R_5\dxa E dx)\nnm\\
&-\Dt \sum_{k\le |\alpha|\le N-2}4\intr \dxa E\dxa(\Tdx\times B)dx\nnm\\
&+\sum_{k\le |\alpha|\le N-1}(\|\dxa n_2\|^2_{L^2_x}+\|\dxa \Tdx n_2\|^2_{L^2_x} +\|\dxa E\|^2_{L^2_x})+\sum_{k+1\le |\alpha|\le N-1}\|\dxa B\|^2_{L^2_x}
\nnm\\
\le& CE_N(U)D_N(U)+C\sum_{k\le |\alpha|\le N}\|\dxa  P_rf_2\|^2_{L^2_{x,v}},\label{E_1a}
\ema
with $0\le k\le N-2$.
\end{lem}

\begin{proof}
First of all, we prove \eqref{E_1}.  Taking the inner product between $\dxa m_1$ and $\dxa\eqref{G_7}$ with $|\alpha|\le N-1$, we have
\bma
&\frac12\Dt \|\dxa m_1\|^2_{L^2_x}+{ \intr \dxa\dt R_1 \dxa m_1dx}+\intr \dxa\Tdx n_1 \dxa m_1dx+\sqrt{\frac23}\intr \dxa\Tdx q_1\dxa m_1dx\nnm\\
&\quad+\kappa_1 (\|\dxa\Tdx  m_1\|^2_{L^2_x}+\frac13\|\dxa\divx  m_1\|^2_{L^2_x})\nnm\\
&=\intr \dxa(n_2 E) \dxa m_1dx+\intr \dxa(m_2\times B) \dxa m_1dx+\intr\dxa R_1\dxa m_1dx.  \label{en_1}
\ema
For the second and third terms on the left hand side of \eqref{en_1}, we have{
\bma
 \intr \dxa\dt R_1 \dxa m_1dx&=\Dt\intr\dxa R_1\dxa m_1dx\nnm\\
&\quad-\intr \dxa R_1\dxa[-\Tdx n_1-\sqrt{\frac23}\Tdx q_1+n_2E+m_2\times B-(v\cdot\Tdx \P_1f_1,v\sqrt M)]dx\nnm\\
&\ge \Dt\intr\dxa R_1\dxa m_1dx-\epsilon(\|\dxa\Tdx n_1\|^2_{L^2_x}+\|\dxa\Tdx q_1\|^2_{L^2_x})\nnm\\
&\quad-C\sqrt{E_N(U)}D_N(U)-\frac{C}{\epsilon}\|\dxa\Tdx \P_1f_1\|^2_{L^2_x},  \label{en_2}
\ema}
and
\bma
\intr \dxa\Tdx n_1 \dxa m_1dx=-\intr\dxa n\dxa\divx mdx=\intr\dxa n_1\dxa\dt n_1dx=\frac12\Dt \|\dxa n_1 \|^2_{L^2_x}. \label{en_3}
\ema
The first and second terms on the right hand side  of \eqref{en_1} are bounded by $C\sqrt{E_N(U)}D_N(U)$. The last term can be estimated by
\bma
\intr\dxa R_2\dxa m_1dx
&\le C \|\dxa\Tdx   \P_1f_1\|_{L^2_{x,v}}
\|\dxa\Tdx m_1\|_{L^2_x}+C(\|\dxa( E f_2)\|_{L^2_{x,v}}\nnm\\
&\quad+\|\dxa( B f_2)\|_{L^2_{x,v}}+\|w^{-\frac12}\dxa \Gamma(f_1,f_1)\|_{L^2_{x,v}})\|\dxa\Tdx m_1\|_{L^2_x}\nnm\\
&\le \frac{\kappa_1}2\|\dxa\Tdx m_1\|_{L^2_x}^2+C\|\dxa\Tdx  \P_1f_1\|^2_{L^2_{x,v}}+C\sqrt{E_N(U)}D_N(U),\label{en_4}
\ema
where we have used  Lemma \ref{e1} to obtain
\bma
&\|w^{-\frac12}\dxa \Gamma(f_1,f_1)\|^2_{L^2_{x,v}}+\|\dxa( E f_2)\|^2_{L^2_{x,v}}+\|\dxa( B f_2)\|^2_{L^2_{x,v}}\nnm\\
&\le C\|  f_1\|^2_{L^2_v(H^{N}_x)}\|w^{\frac12}\Tdx f_1\|^2_{L^2_v(H^{N-1}_x)}+ C(\| E\|^2_{H^N_x}+\| B\|^2_{H^N_x})\|\Tdx f_2\|^2_{L^2_v(H^{N-1}_x)}\nnm\\
&\le C E_N(U) D_N(U).\label{gamma}
\ema
Therefore, it follows from \eqref{en_1}--\eqref{en_4} that
\bma
&\frac12\Dt (\|\dxa m_1\|^2_{L^2_x}+\|\dxa n_1\|^2_{L^2_x})+\Dt\intr\dxa R_1\dxa m_1dx+\sqrt{\frac23}\intr \dxa\Tdx q_1\dxa m_1dx\nnm\\
&\quad+\frac{\kappa_1}2 (\|\dxa\Tdx  m_1\|^2_{L^2_x}+\frac13\|\dxa\divx  m_1\|^2_{L^2_x})\nnm\\
&\le C\sqrt{E_N(U)}D_N(U)+\frac{C}{\epsilon} \|\dxa\Tdx   \P_1f_1\|^2_{L^2_{x,v}}+\epsilon(\|\dxa\Tdx n_1\|^2_{L^2_x}+\|\dxa\Tdx q_1\|^2_{L^2_x}).\label{m_1}
\ema
Similarly, taking the inner product between $\dxa q_1$ and $\dxa\eqref{G_8}$ with $|\alpha|\le N-1$, we have
\bma
&\frac12\Dt \|\dxa q_1\|^2_{L^2_x}+\Dt\intr \dxa R_3\dxa q_1dx+\sqrt{\frac23}\intr \dxa\divx m_1\dxa q_1dx+\frac{\kappa_2}2 \|\dxa\Tdx q_1\|^2_{L^2_x}
\nnm\\
&\le C\sqrt{E_N(U)}D_N(U)+\frac{C}{\epsilon}\|\dxa\Tdx   \P_1f_1\|^2_{L^2_{x,v}}+\epsilon\|\dxa\Tdx m_1\|^2_{L^2_x}.\label{q_1}
\ema
Again, taking the inner product between $\dxa\Tdx n_1$ and $\dxa\eqref{G_5}$ with $|\alpha|\le N-1$ gives
\bma &\Dt\intr \dxa m_1 \dxa\Tdx n_1dx+\frac12\|\dxa\Tdx n_1\|^2_{L^2_x}\nnm\\
&\le C \sqrt{E_N(U)}D_N(U)+\|\dxa\divx  m_1\|^2_{L^2_x}+\|\dxa \Tdx q_1\|^2_{L^2_x}+C\|\dxa\Tdx  \P_1f_1\|^2_{L^2_{x,v}}.\label{abc}
\ema
Taking the summation of $2s_0\sum\limits_{k\le |\alpha|\le N-1}[\eqref{m_1}+\eqref{q_1}]+4\sum\limits_{k\le |\alpha|\le N-1}\eqref{abc}$ with $s_0>0$ large enough, $\epsilon>0$ small enough and $0\le k\le N-1$, we obtain \eqref{E_1}.

Next, we turn to show \eqref{E_1a}. Taking the inner product between $\dxa n_2$ and $\dxa\eqref{G_9a}$ with $|\alpha|\le N-1$, we have
\bma
&\Dt \|\dxa n_2\|^2_{L^2_x}+2\Dt \intr\dxa\divx R_5\dxa n_2 dx+\kappa_3\|\dxa n_2\|^2_{L^2_x}+\kappa_3\|\dxa \Tdx n_2\|^2_{L^2_x}\nnm\\
&\le C\|\dxa\Tdx P_rf_2\|^2_{L^2_{x,v}}+C E_N(U)D_N(U).\label{en_5}
\ema
Similarly, taking the inner product between $\dxa E$ and $\dxa\eqref{G_9b}$ with $|\alpha|\le N-1$ gives
\bma
&\Dt \|\dxa (E,B)\|^2_{L^2_x}+2\Dt \intr\dxa R_5\dxa E dx+\kappa_3\|\dxa  E\|^2_{L^2_x}+\kappa_3\|\dxa  n_2\|^2_{L^2_x}\nnm\\
&\le \frac{C}{\epsilon}(\|\dxa  P_r f_2\|^2_{L^2_{x,v}}+\|\dxa\Tdx P_rf_2\|^2_{L^2_{x,v}})+\epsilon\|\dxa (\Tdx\times B)\|^2_{L^2_{x}}+CE_N(U)D_N(U).\label{en_6}
\ema
And taking the inner product between $\dxa \Tdx\times B$ and $\dxa\eqref{G_9b}$ with $|\alpha|\le N-1$ gives
\bmas
-\Dt\intr \dxa E\dxa(\Tdx\times B)dx-\|\dxa(\Tdx\times E)\|^2_{L^2_x}+\|\dxa(\Tdx\times B)\|^2_{L^2_x}=\intr \dxa m_2\dxa(\Tdx\times B)dx.
\emas
This and the fact that $\|\Tdx\times B\|^2_{L^2_x}=\|\Tdx B\|^2_{L^2_x}$ imply that
\bma
&-2\Dt\intr \dxa E\dxa(\Tdx\times B)dx+\|\dxa\Tdx B\|^2_{L^2_x}\le C\|\dxa  P_r f_2\|^2_{L^2_{x,v}}+2\|\dxa(\Tdx\times E)\|^2_{L^2_x}.\label{en_7}
\ema

Taking the summation of $s_1\sum\limits_{k\le |\alpha|\le N-1}[\eqref{en_5}+\eqref{en_6}]+2\sum\limits_{k\le |\alpha|\le N-2}\eqref{en_7}$ with $s_1>0$ large enough, $\epsilon>0$ small enough and $0\le k\le N-1$, we obtain \eqref{E_1a}.
And this completes the proof of the lemma.
\end{proof}

In the following, we shall estimate the  microscopic terms to close the
energy estimate.

\begin{lem}[Microscopic dissipation]
\label{micro-en}
Let $N\ge 4$ and $(f_1,f_2,E,B)$ be a strong solution to Vlasov-Maxwell-Boltzmann system  \eqref{VMB3}--\eqref{VMB3e}.
Then, there are constants $p_k>0$, $1\le k\le N$ such that
\bma
&\frac12\Dt (\|(f_1,f_2)\|^2_{L^2_{x,v}}+\| (E,B)\|^2_{L^2_x})+\mu \|w^{\frac12}  (\P_1f_1,P_rf_2)\|^2_{L^2_{x,v}}
\le C\sqrt{E_N(U)}D_N(U),\label{E_3}\\
&\frac12\Dt \sum_{1\le|\alpha|\le N}(\|\dxa (f_1,f_2)\|^2_{L^2_{x,v}}+\|\dxa  (E,B)\|^2_{L^2_x})+\mu\sum_{1\le|\alpha|\le N} \|w^{\frac12}\dxa  (\P_1f_1,P_rf_2)\|^2_{L^2_{x,v}}\nnm\\
&\le C\sqrt{E_N(U)}D_N(U),\label{E_2}\\
&\Dt \sum_{1\le k\le N}p_k\sum_{|\beta|=k \atop |\alpha|+|\beta|\le N}\|\dxa\dvb (\P_1f_1,P_rf_2)\|^2_{L^2_{x,v}}
+\mu\sum_{1\le k\le N}p_k\sum_{|\beta|=k \atop |\alpha|+|\beta|\le N}\|w^{\frac12}\dxa\dvb (\P_1f_1,P_rf_2)\|^2_{L^2_{x,v}}\nnm\\
&\le C\sum_{|\alpha|\le N-1}(\|\dxa\Tdx ( f_1, f_2)\|^2_{L^2_{x,v}}+\|\dxa E\|^2_{L^2_x}) +C\sqrt{E_N(U)}D_N(U). \label{E_5}
\ema
\end{lem}

\begin{proof}
Taking the inner product between $\dxa  f_1$ and $\dxa\eqref{VMB3}$ with $|\alpha|\le N$ $(N\ge 4)$, we have
\bma
&\frac12\Dt \|\dxa f_1\|^2_{L^2_{x,v}}-\intr (L\dxa f_1)\dxa f_1dxdv
\nnm\\
&=\frac12\intrr \dxa (v\cdot E f_2)\dxa f_1dxdv-\intrr \dxa ( (E+v\times B)\cdot \Tdv f_2)\dxa f_1dxdv+\intrr\dxa \Gamma(f_1,f_1)\dxa f_1dxdv\nnm\\
&=:I_1+I_2+I_3.\label{G_0}
\ema
For $I_1$, it holds that
\bma
I_1&\le C\sum_{1\le |\alpha'|\le |\alpha|-1}\intr |v|\|\dx^{\alpha'} E\|_{L^3_x} \|\dx^{\alpha-\alpha'}f_2\|_{L^{6}_x}\|\dxa f_1\|_{L^2_{x}}dv\nnm\\
&\quad+C\intr |v|(\| E\|_{L^\infty_x}\|\dxa f_2\|_{L^2_{x}}+\|\dxa E\|_{L^2_x}\|f_2\|_{L^{\infty}_x})\|\dxa f_1\|_{L^2_{x}}dv\nnm\\
&\le C\| E\|_{H^{N}_x}\|w^{\frac12}\Tdx f_2\|_{L^2_v(H^{N-1}_x)}\|w^{\frac12}\dxa f_1\|_{L^2_{x,v}}\le C\sqrt{E_N(U)}D_N(U),\label{I_1}
\ema
for $|\alpha|\ge1$, and
\bma
I_1\le \intr |v|\|E\|_{L^3_x} \|f_2\|_{L^{2}_x}\| f_1\|_{L^6_{x}}dv\le C\sqrt{E_N(U)}D_N(U),
\ema
for $|\alpha|=0$. For $I_2$, it holds that
\bma
I_2&\le C\sum_{1\le|\alpha'|\le N/2}\intr(\|\dx^{\alpha'} E\|_{L^\infty_x}+|v|\|\dx^{\alpha'} B\|_{L^\infty_x})\|\dx^{\alpha-\alpha'}\Tdv f_2\|_{L^{6}_x}\|\dxa f_1\|_{L^2_{x}}dv\nnm\\
&\quad+C\sum_{|\alpha'|\ge N/2}\intr (\|\dx^{\alpha'} E\|_{L^2_x}+|v|\|\dx^{\alpha'} B\|_{L^2_x})\|\dx^{\alpha-\alpha'}\Tdv f_2\|_{L^2_{x}}\|\dxa f_1\|_{L^2_{x}}dv\nnm\\
&\quad-\intrr (E+v\times B)\dxa\Tdv f_2\dxa f_1dxdv\nnm\\
&\le C\sqrt{E_N(U)}D_N(U)-\intrr (E+v\times B)\dxa\Tdv f_2\dxa f_1dxdv.
\ema
For $I_3$, by Lemma 4.1, we obtain
\bma
I_3&\le \|w^{-\frac12}\dxa \Gamma(f_1,f_1)\|_{L^2_{x,v}}\|w^{\frac12}\dxa   \P_1 f_1\|_{L^2_{x,v}}
\le C\sqrt{E_N(U)}D_N(U).\label{I_2}
\ema
Therefore, it follows from \eqref{G_0}--\eqref{I_2} that
\bma
\frac12\Dt \|\dxa f_1\|^2_{L^2_{x,v}}+\mu \|w^{\frac12}\dxa \P_1 f_1\|^2_{L^2_{x,v}} \le C\sqrt{E_N(U)}D_N(U)- \intrr (E+v\times B)\dxa\Tdv f_2\dxa f_1dxdv.\label{G_01}
\ema

Similarly, taking inner product between $\dxa f_2$ and $\dxa\eqref{VMB3a}$ with $|\alpha|\le N$ $(N\ge 4)$, we have
\bma
&\frac12\Dt (\|\dxa f_2\|^2_{L^2_{x,v}}+\|\dxa (E,B)\|^2_{L^2_x})+\mu \|w^{\frac12}\dxa P_r f_2\|^2_{L^2_{x,v}}
\nnm\\
&\le C\sqrt{E_N(U)}D_N(U)-\intrr (E+v\times B)\dxa\Tdv f_1\dxa f_2dxdv.\label{G_00}
\ema

Taking the summation of \eqref{G_01}+\eqref{G_00} for $|\alpha|=0$ and $\sum_{1\le|\alpha|\le N}[\eqref{G_01}+\eqref{G_00}]$, we obtain \eqref{E_3} and \eqref{E_2} respectively.

 In order to close the energy estimate, we need to estimate the terms $\dxa\Tdv f$ with $|\alpha|\le N-1$. For this, we rewrite \eqref{VMB3} and \eqref{VMB3a} as
 \bma
&\partial_t(  \P_1f_1)+v\cdot\Tdx \P_1f_1+ (E+v\times B)\cdot\Tdv
P_rf_2-L(  \P_1f_1)\nnm\\
&=\Gamma(f_1,f_1)+\frac12v\cdot E P_rf_2+ \P_0(v\cdot\Tdx   \P_1f_1-\frac12v\cdot E P_rf_2+ (E+v\times B)\cdot\Tdv P_r f_2)-  \P_1(v\cdot\Tdx  \P_0f_1),\label{G_2}
\ema
and
\bma
&\dt (P_r f_2)+v\cdot\Tdx P_rf_2-v\sqrt M\cdot E+ (E+v\times B)\cdot\Tdv  \P_1f_1+L_1(P_rf_2)\nnm\\
&=\Gamma(f_2,f_1)+\frac12v\cdot E  \P_1f_1+P_{\rm d}(v\cdot\Tdx P_rf_2)-(v\cdot\Tdx P_{\rm d}f_2-\frac12v\cdot E \P_0f_1
+ (E+v\times B)\cdot\Tdv \P_0f_1).\label{b_1}
\ema

Let $1\le k\le N$, and choose $\alpha,\beta $ with $|\beta|=k$ and $|\alpha|+|\beta|\le N$.
Taking inner product between $\dxa\dvb \P_1f_1$ and \eqref{G_2}, between $\dxa\dvb P_rf_2$ and \eqref{b_1} respectively, and then taking summation of the resulted equations, we have
\bma
&\Dt\sum_{|\beta|=k \atop |\alpha|+|\beta|\le N}\|\dxa\dvb (\P_1f_1,P_rf_2)\|^2_{L^2_{x,v}}+\mu\sum_{|\beta|=k \atop |\alpha|+|\beta|\le N}\|w^{\frac12}\dxa\dvb (\P_1f_1,P_rf_2)\|^2_{L^2_{x,v}}\nnm\\
&\le C\sum_{|\alpha|\le N-k}(\|\dxa\Tdx (\P_0f_1,P_{\rm d}f_2)\|^2_{L^2_{x,v}}+\|\dxa\Tdx (\P_1f_1,P_rf_2)\|^2_{L^2_{x,v}}+\|\dxa E\|^2_{L^2_x})\nnm\\
&\quad+C_k\sum_{ |\beta|\le k-1\atop |\alpha|+|\beta|\le N}\|\dxa\dvb (\P_1f_1,P_rf_2)\|^2_{L^2_{x,v}}+C\sqrt{E_N(U)}D_N(U).\label{aa}
\ema
Then taking summation $\sum_{1\le k\le N}p_k\eqref{aa}$ with  constants $p_k$ chosen by
$$\mu p_k\ge 2\sum_{1\le j\le N-k}p_{k+j}C_{k+j},\quad 1\le k\le N-1,\quad p_N=1,$$
we obtain \eqref{E_5}.
The proof of the lemma is then completed.
\end{proof}

With the help of Lemma \ref{macro-en}--\ref{micro-en}, we have

\begin{lem}\label{energy1}
For  $N\ge 4$, there are two equivalent energy functionals
$E^f_{N}(\cdot)\sim E_N(\cdot)$, $H^f_{N}(\cdot)\sim H_N(\cdot)$
such that the following holds. If
$E_N(U_0)$ is sufficiently small, then the Cauchy problem
\eqref{VMB3}--\eqref{VMB3e} of the two-species
Vlasov-Maxwell-Boltzmann
 system admits a unique global
solution $U=(f_1,f_2,E,B)$ satisfying
\bma
\Dt E^f_{N}(U)(t) + \mu D_N(U)(t) &\le 0, \label{G_1}\\
\Dt H^f_{N}(U)(t)+\mu D_N(U)(t)&\le C\|\Tdx (n_1,m_1,q_1)\|^2_{L^2_{x}}+C\|\Tdx B\|^2_{L^2_{x}}.\label{G_4}
\ema
\end{lem}
\begin{proof}
Assume that $$E_N(U)(t)\le \delta$$ for $\delta>0$ being small.

Taking the summation of $A_1[\eqref{E_1}+\eqref{E_1a}]+A_2[\eqref{E_3}+\eqref{E_2}]+\eqref{E_5}$ with $A_2>C_0A_1>0$ large enough and taking $k=0$ in \eqref{E_1} and \eqref{E_1a}, we obtain \eqref{G_1}.

Taking the inner product between $E$ and \eqref{G_9b},   between $f_2$ and \eqref{VMB3a} respectively, we have
\bma
&\frac12\Dt \| E\|^2_{L^2_x}+\Dt \intr R_5 E dx+\kappa_3\| E\|^2_{L^2_x}+\kappa_3\| n_2\|^2_{L^2_x}\nnm\\
&\le \epsilon \|E\|^2_{L^2_x}+\frac{C}{\epsilon}\|\Tdx B\|^2_{L^2_x}+C(\|P_rf_2\|^2_{L^2_{x,v}}+\|\Tdx P_rf_2\|^2_{L^2_{x,v}})+C\sqrt{E_N(U)}D_N(U),\label{I_5}
\ema
and
\bma
\frac12\Dt (\| f_2\|^2_{L^2_{x,v}}+\| E\|^2_{L^2_x})-\intr (L_1 f_2) f_2dxdv
\le \epsilon \|E\|^2_{L^2_x}+\frac{C}{\epsilon}\|\Tdx B\|^2_{L^2_x}+C\sqrt{E_N(U)}D_N(U).\label{I_6}
\ema
Taking the summation of $s_2\eqref{I_5}+\eqref{I_6}$ with $s_2>0$ large enough and $\epsilon>0$ small enough, we have
\bma
&\Dt s_2(\| f_2\|^2_{L^2_{x,v}}+\| E\|^2_{L^2_x})+2\Dt \intr R_5 E dx+\kappa_3\| E\|^2_{L^2_x}+\kappa_3\| n_2\|^2_{L^2_x}+\mu\|w^{\frac12}P_rf_2\|^2_{L^2_{x,v}}\nnm\\
&\le C\|\Tdx B\|^2_{L^2_x}+C\|\Tdx P_rf_2\|^2_{L^2_{x,v}}+C\sqrt{E_N(U)}D_N(U).\label{I_7}
\ema
Taking the inner product between \eqref{G_2} and $  \P_1f_1$, we have
\bma
\Dt \|   \P_1f_1\|^2_{L^2_{x,v}}+\|w^{\frac12} \P_1f_1\|^2_{L^2_{x,v}}\le C\|\Tdx \P_0f_1\|^2_{L^2_{x,v}}+E_N(U)D_N(U). \label{low1}
\ema
Taking the summation of $A_3[\eqref{E_1}+\eqref{E_1a}]+A_4[\eqref{I_7}+\eqref{low1}]+A_5\eqref{E_2}+\eqref{E_5}$ with $A_5>C_0A_4, A_4>C_1A_3$ large enough and taking $k=1$ in \eqref{E_1} and  \eqref{E_1a}, we obtain \eqref{G_4}.
\end{proof}

Repeating the proofs of Lemmas \ref{macro-en}-\ref{micro-en}, we can show
\begin{lem}\label{energy2}
For $N\ge 4$, there are the equivalent energy functionals $E^f_{N,1}(\cdot)\sim E_{N,1}(\cdot)$, $H^f_{N,1}(\cdot)\sim H_{N,1}(\cdot)$
such that if $E_{N,1}(U_0)$ is sufficiently small, then the solution $U=(f_1,f_2,E,B)(t,x,v)$ to the two-species
 Vlasov-Maxwell-Boltzmann system \eqref{VMB3}--\eqref{VMB3e} satisfies
\bgr \Dt E^f_{N,1}(U)(t)+\mu D_{N,1}(U)(t)\le 0,\label{G_4b}\\
\Dt H^f_{N,1}(U)(t)+\mu D_{N,1}(U)(t)\le C\|\Tdx (n_1,m_1,q_1)\|^2_{L^2_{x}}+C\|\Tdx B\|^2_{L^2_{x}}.\label{G_4a}
\egr
\end{lem}

\subsection{Convergence rates for two-species}

Based on the above energy estimates and the convergence rates
of the solution to the linearized system,
the convergence rates of the solution to the two-species VMB
can be summarized in the following theorem.

\begin{thm}\label{time6}
Under the assumptions of Theorem \ref{time3}, there exists  a globally unique solution $(f_1,f_2,E,B)$ to the  system~\eqref{VMB3}--\eqref{VMB3e} satisfying
 \bgr
\|\dx^k(f_1(t),\chi_j)\|_{L^2_{x}}\le C\delta_0(1+t)^{-\frac34-\frac k2},\quad j=0,1,2,3,4,\label{t_1a}\\
\|\dx^k  \P_1f_1(t)\|_{L^2_{x,v}}\le C\delta_0(1+t)^{-\frac54-\frac k2},\label{t_2a}\\
\|\dx^k  (f_2(t),\sqrt M)\|_{L^2_{x}}\le C\delta_0(1+t)^{-2-\frac k2},\label{t_3a}\\
\|\dx^k P_rf_2(t)\|_{L^2_{x,v}}+\|\dx^k E(t)\|_{L^2_{x}}\le C\delta_0(1+t)^{-\frac54-\frac k2},\label{t_4a}\\
\|\dx^k B(t)\|_{L^2_{x,v}}\le C\delta_0(1+t)^{-\frac34-\frac k2},\label{t_5a}\\
\|(\P_1f_1,P_rf_2)(t)\|_{H^N_w} +\|\Tdx (\P_0f_1,P_{\rm d}f_2)(t)\|_{L^2_v(H^{N-1}_x)}+\|\Tdx (E,B)(t)\|_{H^{N-1}_x}\le
C\delta_0(1+t)^{-\frac54},\label{t_6a}\egr
for $k=0,1$.
\end{thm}

\begin{proof}
Let $(f_1,f_2,E,B)$ be a solution to the Cauchy problem \eqref{VMB3}--\eqref{VMB3e} for $t>0$. We can represent the solution in terms of the semigroups $e^{t\BB_0}$ and $e^{t\AA_0}$  as
 \bma
 f_1(t)&=e^{t\BB_0}f_{1,0}+\intt e^{(t-s)\BB_0}G_1(s)ds, \label{Duh1}\\
 (f_2(t),E(t),B(t))&=e^{t\AA_0}(f_{2,0},E_0,B_0)+\intt e^{(t-s)\AA_0}(G_2(s),0,0)ds, \label{Duh2}
 \ema
where the nonlinear terms $G_1$ and $G_2$ are given by \eqref{G1} and \eqref{G2} respectively. Define a functional $Q(t)$  for any $t>0$ by
\bmas Q(t)=\sup_{0\le s\le t}\sum_{k=0}^1&\Big\{\sum_{j=0}^4\|\dx^k(f_1(s),\chi_j)\|_{L^2_{x}}(1+s)^{\frac34+\frac k2}+\|\dx^k  \P_1f_1(s)\|_{L^2_{x,v}}(1+s)^{\frac54+\frac k2}\nnm\\
&+\|\dx^k  P_{\rm d}f_2(s)\|_{L^2_{x,v}}(1+s)^{2+\frac k2}
+(\|\dx^k  P_rf_2(s)\|_{L^2_{x,v}}+\|\dx^k  E(s)\|_{L^2_{x}})(1+s)^{\frac54+\frac k2}\\
&+\|\dx^k  B(s)\|_{L^2_{x}}(1+s)^{\frac34+\frac k2}+(\|(\P_1f_1,P_rf_2)(s)\|_{H^N_w} \nnm\\
&+\|\Tdx (\P_0f_1,P_{\rm d}f_2)(s)\|_{L^2_v(H^{N-1}_x)}+\|\Tdx (E,B)(s)\|_{H^{N-1}_x})(1+s)^{\frac54}\Big\}. \emas
We claim that it holds under the assumptions of Theorem~\ref{time6} that
  \bq
 Q(t)\le C\delta_0.  \label{assume}
 \eq
It is straightforward
 to verify that the estimates  \eqref{t_1a}--\eqref{t_6a}  follow  from \eqref{assume}. Hence, it remains to prove \eqref{assume}.

By Lemma \ref{e1},
we can estimate the nonlinear term $G_1(s),G_2(s)$ for $0\le s\le t$ in terms of $Q(t)$ as
 \bma
 \| G_1(s)\|_{L^2_{x,v}}
 &\le
 C\{\|wf_1\|_{L^{2,3}}\|f_1\|_{L^{2,6}}
    +\|E\|_{L^3_x}(\|wf_2\|_{L^{2,6}}
    +\|\Tdv f_2\|_{L^{2,6}})+\|B\|_{L^3_x}\|w\Tdv f_2\|_{L^{2,6}}\}
    \nnm\\
&\le C(1+s)^{-2}Q(t)^2,  \label{GG_1}
\\
 \| G_1(s)\|_{L^{2,1}}
 & \le
 C\{ \|f_1\|_{L^2_{x,v}}\|w f_1\|_{L^2_{x,v}}
    +\| E\|_{L^3_x}(\|w f_2\|_{L^2_{x,v}}
    +\|\Tdv  f_1\|_{L^2_{x,v}})+\|B\|_{L^3_x}\|w\Tdv f_2\|_{L^2_{x,v}}\}
    \nnm\\
 &\le C(1+s)^{-\frac32}Q(t)^2,\label{GG_2}
 \\
 \| G_2(s)\|_{L^2_{x,v}}
 &\le
 C\{\|wf_2\|_{L^{2,3}}\|f_1\|_{L^{2,6}}+\|wf_1\|_{L^{2,3}}\|f_2\|_{L^{2,6}}\}\nnm\\
   &\quad+ C\{\|E\|_{L^3_x}(\|wf_1\|_{L^{2,6}}
    +\|\Tdv f_1\|_{L^{2,6}})+\|B\|_{L^3_x}\|w\Tdv f_1\|_{L^{2,6}}\}
    \nnm\\
&\le C(1+s)^{-2}Q(t)^2,  \label{GG_3}
\\
 \| G_2(s)\|_{L^{2,1}}
 & \le
 C\{ \|f_2\|_{L^2_{x,v}}\|w f_1\|_{L^2_{x,v}} +\|f_1\|_{L^2_{x,v}}\|w f_2\|_{L^2_{x,v}}\}\nnm\\
   &\quad +C\{\| E\|_{L^3_x}(\|w f_1\|_{L^2_{x,v}}
    +\|\Tdv  f_1\|_{L^2_{x,v}})+\|B\|_{L^3_x}\|w\Tdv f_1\|_{L^2_{x,v}}\}
    \nnm\\
 &\le C(1+s)^{-\frac32}Q(t)^2,\label{GG_4}
 \ema
 and similarly
 \bma
 \|G_2\|_{L^2_v(H^k_x)}
 \le C(1+s)^{-\frac52}Q(t)^2,\label{GG_5}
 \ema
 for $1\le k\le N-1$.
First, we consider the time decay rate of the macroscopic density, momentum and energy of $f_1$. It follows from \eqref{V_1}, \eqref{GG_1} and \eqref{GG_2} that
\bma
\|(\Tdx^k f_1(t),\chi_j)\|_{L^2_x}&\le
C(1+t)^{-\frac34-\frac k2}(\|\Tdx^k f_{1,0}\|_{L^2_{x,v}}+\|f_{1,0}\|_{L^{2,1}})\nnm\\
&\quad+C\intt (1+t-s)^{-\frac34-\frac k2}(\|\Tdx^k G_1(s)\|_{L^2_{x,v}}+\|G_1(s)\|_{L^{2,1}})ds\nnm\\
&\le C\delta_0(1+t)^{-\frac34-\frac k2}+C\intt (1+t-s)^{-\frac34-\frac k2}(1+s)^{-\frac32}Q(t)^2ds\nnm\\
&\le
C\delta_0(1+t)^{-\frac34-\frac k2}+C(1+t)^{-\frac34-\frac k2}Q(t)^2, \quad k=0,1, \label{macro_1}
 \ema
 where we have used
 $$\|\Tdx^k G_1(s)\|_{L^2_{x,v}}+\| G_1(s)\|_{L^{2,1}}\le C(1+s)^{-\frac32}Q(t)^2,\quad k=0,1.$$

Second, we estimate the microscopic part $  \P_1f_1(t)$ as follows. 
By 
\eqref{V_2}, \eqref{GG_1} and \eqref{GG_2}, we have
\bma
\|  \P_1f_1(t)\|_{L^2_{x,v}}&\le
C(1+t)^{-\frac54}(\|f_{1,0}\|_{L^2_{x,v}}+\|f_{1,0}\|_{L^{2,1}})\nnm\\
&\quad+C\intt (1+t-s)^{-\frac54}(\| G_1(s)\|_{L^2_{x,v}}+\|G_1(s)\|_{L^{2,1}})ds\nnm\\
&\le
C\delta_0(1+t)^{-\frac54}+C(1+t)^{-\frac54}Q(t)^2,\label{micro_1}
\ema
and
\bma
\|\Tdx   \P_1f_1(t)\|_{L^2_{x,v}}&\le
C(1+t)^{-\frac74}(\|\Tdx f_{1,0}\|_{L^2_{x,v}}+\|f_{1,0}\|_{L^{2,1}})\nnm\\
&\quad+C\int^{t/2}_0 (1+t-s)^{-\frac74}(\|\Tdx G_1(s)\|_{L^2_{x,v}}+\|G_1(s)\|_{L^{2,1}})ds\nnm\\
&\quad+C\int^t_{t/2} (1+t-s)^{-\frac54}(\|\Tdx G_1(s)\|_{L^2_{x,v}}+\|\Tdx G_1(s)\|_{L^{2,1}})ds\nnm\\
&\le
C\delta_0(1+t)^{-\frac74}+C(1+t)^{-\frac74}Q(t)^2,\label{micro_2}\ema
where we have used
$$
\|\Tdx G_1(s)\|_{L^2_{x,v}}+\|\Tdx G_1(s)\|_{L^{2,1}}\le
C(1+s)^{-2}Q(t)^2.
$$

Next, we turn to consider the time decay rate of the $f_2,E,B$. By \eqref{D_3}, \eqref{D_4},  \eqref{GG_3}, \eqref{GG_4} and \eqref{GG_5}, we have
\bma
\|P_rf_2(t)\|_{L^2_{x,v}}+\|E(t)\|_{L^2_x}\le& C(1+t)^{-\frac54}(\|U_0\|_{Z^2}+\|U_0\|_{Z^1}+\|\Tdx^2U_0\|_{Z^2})\nnm\\
&+C\intt (1+t-s)^{-\frac54}(\|G_2\|_{L^2_{x,v}}+\|G_2\|_{L^{2,1}})ds\nnm\\
&+C\intt (1+t-s)^{-2}\|\Tdx^2G_2\|_{L^2_{x,v}}ds\nnm\\
\le &C\delta_0(1+t)^{-\frac54}+C(1+t)^{-\frac54}Q(t)^2,
\ema
and
\bma
\|\Tdx P_rf_2(t)\|_{L^2_{x,v}}+\|\Tdx E(t)\|_{L^2_x}
\le& C(1+t)^{-\frac74}(\|\Tdx U_0\|_{Z^2}+\|U_0\|_{Z^1}+\|\Tdx^3U_0\|_{Z^2})\nnm\\
&+C\int^{t/2}_0 (1+t-s)^{-\frac74}(\|\Tdx G_2\|_{L^2_{x,v}}+\|G_2\|_{L^{2,1}})ds\nnm\\
&+C\int^t_{t/2} (1+t-s)^{-\frac54}(\|\Tdx G_2\|_{L^2_{x,v}}+\|\Tdx G_2\|_{L^{2,1}})ds\nnm\\
&+C\intt (1+t-s)^{-2}\|\Tdx^3G_2\|_{L^2_{x,v}}ds\nnm\\
\le &C\delta_0(1+t)^{-\frac74}+C(1+t)^{-\frac74}Q(t)^2.
\ema
By \eqref{D_0},  \eqref{GG_3}, \eqref{GG_4} and \eqref{GG_5}, we have
\bma
\|\Tdx^kB(t)\|_{L^2_x}\le& C(1+t)^{-\frac34-\frac k2}(\|\Tdx^k U_0\|_{Z^2}+\|U_0\|_{Z^1}+\|\Tdx^2U_0\|_{Z^2})\nnm\\
&+C\intt (1+t-s)^{-\frac34-\frac k2}(\|\Tdx^k G_2\|_{L^2_{x,v}}+\|G_2\|_{L^{2,1}})ds\nnm\\
&+C\intt (1+t-s)^{-2}\|\Tdx^{k+2}G_2\|_{L^2_{x,v}}ds\nnm\\
\le &C\delta_0(1+t)^{-\frac34-\frac k2}+C(1+t)^{-\frac34-\frac k2}Q(t)^2,\quad k=0,1. \label{mag_1}
\ema

Finally, we estimate the higher order terms.  Since
\bq d_1H^f_{N,1}(U)\le D_{N,1}(U)+C\sum_{|\alpha|=N}\|\dxa(E, B)\|^2_{L^2_{x}}, \label{J_5}\eq
for $d_1>0$, we still need to estimate the decay rate of $\|\dxa(E, B)\|^2_{L^2_{x}}$ for $|\alpha|=N$.
 By  Theorem 1.3 in \cite{Duan4}, one has
\bma
E_{k,1}(U)(t)\le C(1+t)^{-3/2}(E_{k+2,1}(U_0)+(\delta_0+Q(t)^2)^2),\label{J_6z}
\ema
for any integer $k \ge 4$, where $E_{k,1}$ is defined by \eqref{energy3}.
Then
\bma
\|\Tdx^N(E,B)(t)\|_{L^2_x}\le& C(1+t)^{-\frac54}(\|\Tdx^NU_0\|_{Z^2}+\|U_0\|_{Z^1}+\|\Tdx^{N+2}U_0\|_{Z^2})\nnm\\
&+C\intt (1+t-s)^{-\frac54}(\|\Tdx^NG_2(s)\|_{L^2_{x,v}}+\|G_2(s)\|_{L^{2,1}})ds\nnm\\
&+C\intt (1+t-s)^{-2}\|\Tdx^{N+2}G_2(s)\|_{L^2_{x,v}}ds\nnm\\
\le &C\delta_0(1+t)^{-\frac54}+C(1+t)^{-\frac54}(\delta_0+Q(t)^2)^2,\label{mag_3z}
\ema
where we have used \eqref{J_6z} to obtain
$$\|G_2(s)\|_{L^2_{v}(H^{N+2}_x)}\le CE_{N+3,1}(U)(s)\le C(E_{N+5,1}(U_0)+(\delta_0+Q(t)^2)^2)(1+s)^{-\frac32}.
$$
Then, by \eqref{G_4a} and \eqref{J_5}, we have
\bma
&\Dt H^f_{N,1}(U)(t)+d_1\mu H^f_{N,1} (U)(t)\nnm\\
&\le C\|\Tdx (n_1,m_1,q_1)(t)\|^2_{L^2_{x}}+C\|\Tdx B(t)\|^2_{L^2_{x}}+C\sum_{|\alpha|=N}\|\dxa(E,B)(t)\|^2_{L^2_{x}}.
\ema
This and \eqref{mag_3z} give
\bma
H^f_{N,1}(U)(t)\le& e^{-d_1\mu t}H^f_{N,1}(U_0)+C\intt e^{-d_1\mu(t-s)}(\|\Tdx (n_1,m_1,q_1)(s)\|^2_{L^2_{x}}+\|\Tdx B(s)\|^2_{L^2_{x}})ds\nnm\\
&+C\intt e^{-d_1\mu(t-s)}\sum_{|\alpha|=N}\|\dxa(E,B)(s)\|^2_{L^2_{x}}ds\nnm\\
\le& C(1+t)^{-\frac52}(\delta_0+Q(t)^2+(\delta_0+Q(t)^2)^2.\label{other1}
\ema
By summing \eqref{macro_1}--\eqref{mag_1} and \eqref{other1}, we have
$$Q(t)\le C(\delta_0+Q(t)^2)+C(\delta_0+Q(t)^2)^2,$$
which yields \eqref{assume} when $\delta_0>0$ is chosen small enough.
This completes the proof of the theorem.
\end{proof}

Indeed, some of the above convergence rates can be shown to be optimal even
for the nonlinear system.

\begin{thm}
Under the assumption of Theorem \ref{time4}, the global solution $(f_1,f_2,E,B)$  to the two-species
 Vlasov-Maxwell-Boltzmann system~\eqref{VMB3}--\eqref{VMB3e} satisfies
 \bma
 C_1\delta_0(1+t)^{-\frac34}&\le \|(f_1(t),\chi_j)\|_{L^2_{x}}\le C_2\delta_0(1+t)^{-\frac34},\quad j=0,1,2,3,4,\label{B_1z}
 \\
 C_1\delta_0(1+t)^{-\frac54}&\le \|  \P_1f_1(t)\|_{L^2_{x,v}}\le C_2\delta_0(1+t)^{-\frac54},\label{B_2z}
 \\
 C_1\delta_0(1+t)^{-\frac54}&\le \|P_rf_2(t)\|_{L^2_{x,v}}\le C_2\delta_0(1+t)^{-\frac54},\label{B_3z}
 \\
 C_1\delta_0(1+t)^{-\frac54}&\le \|E(t)\|_{L^2_{x}}\le C_2\delta_0(1+t)^{-\frac54},\label{B_4z}
 \\
 C_1\delta_0(1+t)^{-\frac34}&\le \|B(t)\|_{L^2_{x}}\le C_2\delta_0(1+t)^{-\frac34},\label{B_5z}
\ema
 for $t>0$ large with two constants $C_2>C_1$.
\end{thm}

\begin{proof}
By \eqref{Duh1}, \eqref{Duh2}, Theorem~\ref{time1b} and Theorem~\ref{time6}, we can establish the lower bounds of the time decay rates of macroscopic density, momentum and energy of the global solution $(f_1,f_2,E,B)$ to the system \eqref{VMB3}--\eqref{VMB3e} and its microscopic part for $t>0$ large enough. For example, it holds that
 \bmas
\|(f_1(t),\chi_j)\|_{L^2_{x}}
&\ge
 \| (e^{t\BB_0}f_{1,0},\chi_j)\|_{L^2_{x}}
 -\intt\| (e^{(t-s)\BB_0}G_1(s),\chi_j)\|_{L^2_{x}}ds
 \\
&\ge
 C_1\delta_0(1+t)^{-\frac34 }-C_2\delta_0^2(1+t)^{-\frac34 },
\\
 \| E(t)\|_{L^2_{x}}
&\ge
 \| (e^{t\AA_0}(f_{2,0},E_0,B_0))_2\|_{L^2_{x}}
 -\intt\|  (e^{(t-s)\AA_0}(G_2(s),0,0))_2\|_{L^2_{x}}ds
 \\
&\ge C_1\delta_0(1+t)^{-\frac54}-C_2\delta_0^2(1+t)^{-\frac54},
\emas
Therefore, for $\delta_0>0$
small and $t>0$ large, we can prove \eqref{B_1z}--\eqref{B_5z}. 
\end{proof}

\subsection{Corresponding results on one-species}
%


Finally, we give the corresponding results on the one-species VMB.
Let $N$ be a positive integer and $U=(f,E,B)$, and
\bmas
E^1_{N,k}(U)&=\sum_{|\alpha|+|\beta|\le N}\|w^k\dxa\dvb f\|^2_{L^2_{x,v}}+\sum_{|\alpha|\le N}\|\dxa(E,B)\|^2_{L^2_x},\\
H^1_{N,k}(U)&= \sum_{|\alpha|+|\beta|\le N}\|w^k\dxa\dvb
\P_1f\|^2_{L^2_{x,v}}+\sum_{1\le|\alpha|\le N}\|\dxa (E,B)\|^2_{L^2_x} \\
&\quad +\sum_{|\alpha|\le N-1}\|\dxa\Tdx \P_0f\|^2_{L^2_{x,v}}+\| P_{\rm d}f\|^2_{L^2_{x,v}},\\
D^1_{N,k}(U)&=\sum_{|\alpha|+|\beta|\le N}\|w^{\frac12+k}\dxa\dvb \P_1f\|^2_{L^2_{x,v}}+\sum_{|\alpha|\le N-1}\|\dxa\Tdx \P_0f\|^2_{L^2_{x,v}}+\| P_{\rm d}f\|^2_{L^2_{x,v}}\nnm\\
&\quad+\sum_{1\le |\alpha|\le N-1}\|\dxa E\|^2_{L^2_x}+\sum_{2\le |\alpha|\le N-1}\|\dxa  B\|^2_{L^2_x},
\emas
for $k\ge 0$. For brevity, we write $E^1_N(U)=E^1_{N,0}(U)$, $H^1_N(U)=H^1_{N,0}(U)$ and $D^1_N(U)=D^1_{N,0}(U)$ for $k=0$.

Applying the similar argument as to derive the equation~\eqref{G_9}--\eqref{G_8} and making use of the system \eqref{1VMB3}--\eqref{1VMB3c}, we can obtain a  compressible Navier-Stokes-Maxwell  type equations with inhomogeneous terms for the macroscopic density, momentum and energy $(n,  m, q)=:((f,\chi_0),(f,v\chi_0),(f,\chi_4))$ and $E,B$ as follows
\bma
\dt n+\divx  m&=0,\label{1G_9a}\\
\dt  m+\dt R_1+\Tdx n+\sqrt{\frac23}\Tdx q-E&=\kappa_1 (\Delta_x m+\frac13\Tdx{\rm div}_x m)+n E+m\times B+R_2,\label{1G_7a}\\
\dt q+\dt R_3+\sqrt{\frac23}\divx m&=\kappa_2 \Delta_x q+\sqrt{\frac23} E\cdot m+R_4,\label{1G_8a}\\
\dt E&=\Tdx\times B-m,\label{1G_9b}\\
\dt B&=-\Tdx\times E,\label{1G_9c}
\ema
where the viscosity and heat conductivity
coefficients $\kappa_1,\kappa_2>0$ and the remainder terms $R_1, R_2, R_3, R_4$ are defined by
\bmas
\kappa_1&=-(L^{-1}   \P_1(v_1\chi_2),v_1\chi_2),\quad \kappa_2=-(L^{-1}   \P_1(v_1\chi_4),v_1\chi_4),\\
R_1&=( v\cdot\Tdx L^{-1}  \P_1f,v\sqrt M),\quad R_2=-(v\cdot\Tdx L^{-1}(   \P_1(v\cdot\Tdx  \P_1f)-   \P_1 G),v\sqrt M),\\
R_3&=( v\cdot\Tdx L^{-1}  \P_1f,\chi_4),\quad R_4=-(v\cdot\Tdx L^{-1}(   \P_1(v\cdot\Tdx  \P_1f)-   \P_1 G),\chi_4).
\emas
Here,
\bq G=\frac12 (v\cdot E)f-(E+v\times B)\cdot\Tdv f+\Gamma(f,f).\label{1G1}\eq

Similar to Section 5.1, we have the energy estimates of the one-species VMB system \eqref{1VMB3}--\eqref{1VMB3d} as  follows.

\begin{lem}[Macroscopic dissipation] \label{macro-en1} Let $(n,m,q,E,B)$  be the strong solutions to \eqref{1G_9a}--\eqref{1G_9c}. Then, there are two constants $s_0,s_1>0$ such that
\bmas
&\Dt \sum_{k\le |\alpha|\le N-1}s_0 (\|\dxa(n, m,q)\|^2_{L^2_x}+\|\dxa (E,B)\|^2_{L^2_x}+2\intr \dxa R_1\dxa mdx+2\intr \dxa R_3\dxa qdx )\nnm\\
&+\Dt \sum_{k\le |\alpha|\le N-1}s_1\intr \dxa m \dxa\Tdx ndx-\Dt \sum_{k+1\le |\alpha|\le N-1}8\intr \dxa m\dxa Edx\nnm\\
&-\Dt \sum_{k+1\le |\alpha|\le N-2}2\intr \dxa E\dxa(\Tdx\times B)dx-\Dt s_0\sqrt{\frac23}\intr m^2qdx\nnm\\
&+\sum_{k\le |\alpha|\le N-1} (\|\dxa\Tdx (n, m,q)\|^2_{L^2_x}+\|\dxa n\|^2_{L^2_x})+\sum_{k+1\le |\alpha|\le N-1} \|\dxa E\|^2_{L^2_x}+\sum_{k+2\le |\alpha|\le N-1} \|\dxa B\|^2_{L^2_x}
\nnm\\
\le & C(\sqrt{E^1_N(U)}+E^1_N(U))D^1_N(U)+C\sum_{k\le |\alpha|\le N}\|\dxa  \P_1f\|^2_{L^2_{x,v}},
\emas
with $0\le k\le N-3$.
\end{lem}


\begin{lem}[Microscopic dissipation]
\label{micro-en1}
Let $N\ge 4$ and $(f,E,B)$ be a strong solution to VMB system  \eqref{1VMB3}--\eqref{1VMB3d}.
Then, there are constants $p_k>0$, $1\le k\le N$ such that
\bmas
&\frac12\Dt (\|f\|^2_{L^2_{x,v}}+\| (E,B)\|^2_{L^2_x}-\sqrt{\frac23}\intr m^2qdx)+\mu \|w^{\frac12} \P_1f\|^2_{L^2_{x,v}}\le C(\sqrt{E^1_N(U)}+E^1_N(U))D^1_N(U),
\\
&\frac12\Dt \sum_{1\le|\alpha|\le N}(\|\dxa f\|^2_{L^2_{x,v}}+\|\dxa  (E,B)\|^2_{L^2_x})+\mu\sum_{1\le|\alpha|\le N} \|w^{\frac12}\dxa  \P_1f\|^2_{L^2_{x,v}}\le C\sqrt{E^1_N(U)}D^1_N(U),
\\
&\Dt \|\P_1f\|^2_{L^2_{x,v}}+\|w^{\frac12} \P_1f\|^2_{L^2_{x,v}}\le C\|\Tdx \P_0f\|^2_{L^2_{x,v}}+E^1_N(U)D^1_N(U),
\\
&\Dt \sum_{1\le k\le N}p_k\sum_{|\beta|=k \atop |\alpha|+|\beta|\le N}\|\dxa\dvb \P_1f\|^2_{L^2_{x,v}}
+\mu\sum_{1\le k\le N}p_k\sum_{|\beta|=k \atop |\alpha|+|\beta|\le N}\|w^{\frac12}\dxa\dvb \P_1f\|^2_{L^2_{x,v}}\nnm\\
&\le C\sum_{|\alpha|\le N-1}\|\dxa\Tdx  f\|^2_{L^2_{x,v}} +C\sqrt{E^1_N(U)}D^1_N(U).
\emas
\end{lem}

\begin{lem}\label{energy1z}
Let $N\ge 4$. Then, there are two equivalent energy functionals
$E^{1,f}_{N}(\cdot)\sim E^1_N(\cdot)$, $H^{1,f}_{N}(\cdot)\sim H^1_N(\cdot)$
such that the following holds. If
$E^1_N(U_0)$ is sufficiently small, then the Cauchy problem
\eqref{1VMB3}--\eqref{1VMB3d} of the one-species
VMB system admits a unique global
solution $U=(f,E,B)$ satisfying
\bma
\Dt E^{1,f}_{N}(U)(t) + \mu D^1_N(U)(t) &\le 0, \label{Gb}\\
\Dt H^{1,f}_N(U)(t)+\mu D^1_N(U)(t)&\le C\|\Tdx (n,m,q)\|^2_{L^2_{x}}+\|\Tdx E\|^2_{L^2_{x}}+\|\Tdx^2 B\|^2_{L^2_{x}}.\label{G_4c}
\ema
\end{lem}

\begin{lem}\label{energy2a}
Let $N\ge 4$.  There are the equivalent energy functionals $E^{1,f}_{N,1}(\cdot)\sim E^1_{N,1}(\cdot)$, $H^{1,f}_{N,1}(\cdot)\sim H^1_{N,1}(\cdot)$
such that if $E^1_{N,1}(U_0)$ is sufficiently small, then the solution $U=(f,E,B)(t,x,v)$ to the one-species VMB system \eqref{1VMB3}--\eqref{1VMB3d} satisfies
\bgr \Dt E^{1,f}_{N,1}(U)(t)+\mu D^1_{N,1}(U)(t)\le 0,\label{G_4d}\\
\Dt H^{1,f}_{N,1}(U)(t)+\mu D^1_{N,1}(U)(t)\le C\|\Tdx (n,m,q)\|^2_{L^2_{x}}+\|\Tdx E\|^2_{L^2_{x}}+\|\Tdx^2 B\|^2_{L^2_{x}}.\label{G_4e}
\egr
\end{lem}


With  these energy estimates, we can prove  Theorems~\ref{time3a} for  the nonlinear one-species VMB system~\eqref{1VMB3}--\eqref{1VMB3d}.

\begin{proof}[\underline{Proof of Theorem \ref{time3a}}]
Let $f,E,B$ be a solution to the  problem \eqref{1VMB3}--\eqref{1VMB3d} for $t>0$.  We can represent its solution in terms of the semigroup $e^{t\AA_2}$  by
 \bq
 (f,E,B)(t)=e^{t\AA_2}(f_0,E_0,B_0)+\intt e^{(t-s)\AA_2}(G,0,0)(s)ds,     \label{1Duh}
 \eq
where the nonlinear term $G$ is given by \eqref{1G1}.  For this global
solution $f$,  we define a functional $Q(t)$  for
any $t>0$ as
\bmas
 Q(t)=\sup_{0\le s\le t}\sum_{k=0}^1
 &\Big\{
    (1+s)^{1+\frac k4}\|\dx^k(f(s),\sqrt M)\|_{L^2_x}
   +(1+s)^{\frac58+\frac k4}\|\dx^k(f(s),v\sqrt M)\|_{L^2_x}
\\
&+(1+s)^{\frac34+\frac k2}\|\dx^k(f(s),\chi_4)\|_{L^2_x}+(1+s)^{\frac78+\frac k4}\|\dx^k\P_1f(s)\|_{L^2_{x,v}}
\\
&   +(1+s)^{\frac34+\frac k4}\ln(1+t)\|\dx^kE(s)\|_{L^2_x}+(1+s)^{\frac38+\frac k4} \|\dx^kB(s)\|_{L^2_x}
\\
&+(1+s)^{\frac58}
  (\| \P_1f(s)\|_{H^N_w} +\|\Tdx \P_0f(s)\|_{L^2(\R^3_{v},H^{N-1}_x)}+\|\Tdx (E,B)(s)\|_{H^{N-1}_x})\, \Big\},
 \emas

In the case of $\Tdx\cdot E_0=(f_0,\sqrt M)$ and $B_0=0$, we can obtain by
\eqref{1F_2}--\eqref{1B_2} and \eqref{1F_3a}--\eqref{1F_4a} that
  \bma
   \|\dxa (f(t),\chi_j)\|_{L^2_{x}}
   &\le
    C[(1+t)^{-\frac34-\frac{k}2}+(1+t)^{-\frac78-\frac{k}4}]
  (\|\da_x U_0\|_{Z^2}+\|\dx^{\alpha'}U_0\|_{Z^1})\nnm\\
  &\quad+C(1+t)^{-m-\frac12}\|\Tdx^{m}\da_x U_0\|_{Z^2}, \quad j=1,2,3, \label{de_1}\\
     \|\dxa (f(t),\chi_4)\|_{L^2_{x}}
   &\le
    C(1+t)^{-\frac34-\frac{k}2}
     ( \|\dxa U_0\|_{Z^2 }+ \|\dx^{\alpha'}(f_0,\chi_4)\|_{L^1_x }+  \|\dx^{\alpha'}\Tdx U_0\|_{Z^1}),\label{de_1a}\\
   \|\dxa \P_1f(t)\|_{L^2_{x,v}}
   &\le
    C(1+t)^{-\frac98-\frac{k}4}
     ( \|\dxa U_0\|_{Z^2 }+   \|\dx^{\alpha'}U_0\|_{Z^1})+C(1+t)^{-m-\frac12}\|\Tdx^{m}\da_x U_0\|_{Z^2},\label{de_2}\\
  \|\dxa B(t)\|_{L^2_x}
   &\le
    C(1+t)^{-\frac58-\frac{k}4}
     ( \|\dxa U_0\|_{Z^2}+   \|\dx^{\alpha'}U_0\|_{Z^1})+C(1+t)^{-m}\|\Tdx^{m}\da_x U_0\|_{Z^2},\label{de_4}
  \ema
where $(f,E,B)=e^{t\AA_2}U_0$ with $U_0=(f_0,E_0,B_0)$, $\alpha'\le \alpha$ and $k=|\alpha-\alpha'|$.

By Lemma \ref{e1},
we can estimate the nonlinear term $G(s)$ for $0\le s\le t$ in terms of $Q(t)$ as
 \bma
 \| G(s)\|_{L^2_{x,v}}
 &\le
 C\{\|wf\|_{L^{2,3}}\|f\|_{L^{2,6}}
    +\|E\|_{L^3_x}(\|wf\|_{L^{2,6}}
    +\|\Tdv f\|_{L^{2,6}})+\|B\|_{L^3_x}\|w\Tdv f\|_{L^{2,6}}\}
    \nnm\\
&\le C(1+s)^{-5/4}Q(t)^2,  \label{1GG_1}
\\
 \| G(s)\|_{L^{2,1}}
 & \le
 C\{ \|f\|_{L^2_{x,v}}\|w f\|_{L^2_{x,v}}
    +\| E\|_{L^3_x}(\|w f\|_{L^2_{x,v}}
    +\|\Tdv  f\|_{L^2_{x,v}})+\|B\|_{L^3_x}\|w\Tdv f\|_{L^2_{x,v}}\}
    \nnm\\
 &\le C(1+s)^{-1}Q(t)^2,\label{1GG_2}
 \ema
 and similarly
 \bma
 \|G(s)\|_{L^2_v(H^k_x)} \le C(1+s)^{-5/4}Q(t)^2,\label{1GG_3}
 \ema
 for $1\le k\le N-1$.
Then, it follows from \eqref{1D_2}, \eqref{1GG_1} and \eqref{1GG_2} that
 \bma
\|\Tdx^k(f(t),\sqrt M)\|_{L^2_x}&\le
C(1+t)^{-\frac54-\frac k2}(\|\Tdx^kU_0\|_{Z^2}+\|U_0\|_{Z^1})\nnm\\
&\quad+C\intt (1+t-s)^{-\frac54}(\|\Tdx^k G(s)\|_{L^2_{x,v}}+\|\Tdx^kG(s)\|_{L^{2,1}})ds\nnm\\
&\le
C\delta_0(1+t)^{-\frac54-\frac k2}+C(1+t)^{-1-\frac k4}Q(t)^2, \label{density_1}
 \ema
for $k=0,1$.

Similarly, in terms of \eqref{1D_3} and \eqref{de_1} we have
 \bma
\|\Tdx^k(f(t),v\sqrt M)\|_{L^2_x}&\le
C(1+t)^{-\frac58-\frac k4}(\|\Tdx^kU_0\|_{Z^2}+\|U_0\|_{Z^1}+\|\Tdx^2 U_0\|_{Z^2})\nnm\\
&\quad+C\intt (1+t-s)^{-\frac34-\frac {3k}8}(\|\Tdx^k
G(s)\|_{L^2_{x,v}}+\|\Tdx^k  G(s)\|_{L^{2,1}})ds\nnm\\
&\quad+C\intt (1+t-s)^{-2}\|\Tdx^{2+k}G\|_{L^2_{x,v}}ds\nnm\\
&\le C\delta_0(1+t)^{-\frac58-\frac k4}+C(1+t)^{-\frac58-\frac k4}Q(t)^2 \label{momentum_1},
 \ema for $k=0,1$.

In terms of \eqref{1D_4} and \eqref{de_1a},
we can estimate the macroscopic energy $(f(t),\chi_4)$ and its spatial derivative as
 \bma
  \|\Tdx^k(f(t),\chi_4)\|_{L^2_x} &\le
   C(1+t)^{-\frac34-\frac k2}(\|\Tdx^kU_0\|_{Z^2}+\|U_0\|_{Z^1})\nnm\\
   &\quad+C\intt (1+t-s)^{-\frac34-\frac k2}(\|\Tdx^kG(s)\|_{L^2_{x,v}}+\| (G(s),\chi_4)\|_{L^{1}_x} +\|\Tdx G(s)\|_{L^{2,1}})ds\nnm\\
&\le
 C\delta_0(1+t)^{-\frac34-\frac k2}+C(1+t)^{-\frac34-\frac k2}Q(t)^2, \label{energy_1}
\ema
for $k=0,1$, where we have used
$$
 \|(G(s),\chi_4)\|_{L^1_{x}}=\sqrt{\frac23}\|E\cdot m\|_{L^1_{x}}
 \le C(1+s)^{-\frac54}Q(t)^2.
$$
In addition, the microscopic part $ \P_1f(t)$ can be estimated by
\eqref{1D_5}  and \eqref{de_2} as follows
\bma
\| \Tdx^k\P_1f(t)\|_{L^2_{x,v}}
&\le C(1+t)^{-\frac78-\frac k4}(\|\Tdx^kU_0\|_{Z^2}+\|U_0\|_{Z^1}+\|\Tdx^2 U_0\|_{Z^2})
\nnm\\
&\quad+C\intt(1+t-s)^{-\frac98}
         (\|\Tdx^kG(s)\|_{L^2_{x,v}}+\|\Tdx^kG(s)\|_{L^{2,1}})ds
 \nnm\\
 &\quad+C\intt (1+t-s)^{-2}\|\Tdx^{2+k}G\|_{L^2_{x,v}}ds\nnm\\
&\le C\delta_0(1+t)^{-\frac78-\frac k4}+C(1+t)^{-\frac78-\frac k4}Q(t)^2,  \label{miscro1}
\ema
 for $k=0,1$.

Moreover, the electricity potential $E$ is bounded by
   \bma
 \|\Tdx^kE(t)\|_{L^2_x}
 &\le
  C(1+t)^{-\frac34-\frac k2}(\|\Tdx^kU_0\|_{Z^2}+\|U_0\|_{Z^1}+\|\Tdx^2 U_0\|_{Z^2})
\nnm\\
&\quad  +C\intt (1+t-s)^{-\frac34-\frac k2}(\|\Tdx^k G(s)\|_{L^2_{x,v}}+\|G(s)\|_{L^{2,1}})ds
\nnm\\
&\quad+C\intt (1+t-s)^{-2}\|\Tdx^{2+k}G\|_{L^2_{x,v}}ds\nnm\\
&\le
 C\delta_0(1+t)^{-\frac34-\frac k2}+C(1+t)^{-\frac34-\frac k4}\ln(1+t)Q(t)^2, \label{potential1}
\ema
for $k=0,1$.
And the magnetic potential $B$ is bounded by
   \bma
 \|\Tdx^k B(t)\|_{L^2_x}
 &\le
  C(1+t)^{-\frac38-\frac k4}(\|\Tdx^k U_0\|_{Z^2}+\|U_0\|_{Z^1}+\|\Tdx^{k+1} U_0\|_{Z^2})
\nnm\\
&\quad  +C\intt (1+t-s)^{-\frac58-\frac k4}(\| G(s)\|_{L^2_{x,v}}+\|G(s)\|_{L^{2,1}})ds
\nnm\\
&\quad+C\intt (1+t-s)^{-1}\|\Tdx^{k+1}G\|_{L^2_{x,v}}ds\nnm\\
&\le
 C\delta_0(1+t)^{-\frac38-\frac k4}+C(1+t)^{-\frac38-\frac k4}Q(t)^2, \label{potential3}
\ema for $k=0,1,2$.

Next, we estimate the higher order terms as below. By  a similar argument as in Theorem 1.3 in \cite{Duan4}, we have
\bma
E^{1,f}_{k,1}(U)(t)\le C(1+t)^{-3/4}(E_{k+1,1}(U_0)+(\delta_0+Q(t)^2)^2),\label{J_6}
\ema
for any integer $k \ge 4$.
Then
\bma
\|\Tdx^N(E,B)(t)\|_{L^2_x}\le& C(1+t)^{-\frac58}(\|\Tdx^NU_0\|_{Z^2}+\|U_0\|_{Z^1}+\|\Tdx^{N+1}U_0\|_{Z^2})\nnm\\
&+C\intt (1+t-s)^{-\frac78}(\|\Tdx^NG(s)\|_{L^2_{x,v}}+\|G(s)\|_{L^{2,1}})ds\nnm\\
&+C\intt (1+t-s)^{-1}\|\Tdx^{N+1}G(s)\|_{L^2_{x,v}}ds\nnm\\
\le &C\delta_0(1+t)^{-\frac58}+C(1+t)^{-\frac58}(\delta_0+Q(t)^2)^2,\label{mag_3}
\ema
where we have used
$$\|G(s)\|_{L^2_{v}(H^{N+1}_x)}\le E_{N+2,1}(U)(s)\le C(E_{N+3,1}(U_0)+(\delta_0+Q(t)^2)^2)(1+s)^{-\frac34}.
$$
 By \eqref{G_4a} and \eqref{J_5}, we have
\bma
&\Dt H^{1,f}_{N,1}(U)(t)+d_1\mu H^{1,f}_{N,1} (U)(t)\nnm\\
&\le C(\|\Tdx (n,m,q,E,B)(t)\|^2_{L^2_{x}}+\|\Tdx^2 B(t)\|^2_{L^2_{x}})+C\sum_{|\alpha|=N}\|\dxa(E,B)(t)\|^2_{L^2_{x}},
\ema
which and \eqref{J_6} lead to
\bma
H^{1,f}_{N,1}(U)(t)\le& e^{-d_1\mu t}H^{1,f}_{N,1}(U_0)+C\intt e^{-d_1\mu(t-s)}\sum_{|\alpha|=N}\|\dxa(E,B)(s)\|^2_{L^2_{x}}ds\nnm\\
&+C\intt e^{-d_1\mu(t-s)}(\|\Tdx (n,m,q,E,B)(s)\|^2_{L^2_{x}}+\|\Tdx^2 B(s)\|^2_{L^2_{x}})ds\nnm\\
\le& C(1+t)^{-\frac54}(\delta_0+Q(t)^2+(\delta_0+Q(t)^2)^2.\label{other}
\ema
By summing \eqref{density_1}--\eqref{potential3} and \eqref{other}, we have
$$Q(t)\le C(\delta_0+Q(t)^2)+C(\delta_0+Q(t)^2)^2,$$
from which \eqref{t4.2}--\eqref{t4.4} can be  verified provided that $\delta_0>0$ is small enough.
Similarly,  as Theorem \ref{time4}, we can prove \eqref{B_1e}--\eqref{B_3}.
\end{proof}

\bigskip
\noindent {\bf Acknowledgements:}
The research of the first author
was supported by the NNSFC grants No. 11171228, 11231006
and 11225102, and by the Key Project of Beijing Municipal Education
Commission no. CIT$\&$TCD2014\\
0323. The research of the second author was supported by the  NSFC-RGC Grant, N-CityU102/12. And research of
the third author was supported by the NNSFC grants No. 11301094 and Project supported by Beijing Postdoctoral Research Foundation No. 2014ZZ-96.



\medskip


\begin{thebibliography}{99}
\setlength{\itemsep}{-4pt}
\renewcommand{\baselinestretch}{1}
\small

\bibitem{Cercignani} Cercignani C., Illner  R., Pulvirenti M.: \emph{The Mathematical Theory of Dilute Gases}, AMS. Vol. 106. Springer-Verlag, New York (1994)

\bibitem{ChapmanCowling} Chapman S.,  Cowling T.G.: \emph{The mathematical theory of non-uniform gases}, 3rd edition. Cambridge University Press, London (1970)

\bibitem{Duan1} R.J. Duan and R. M. Strain, Optimal time decay of the Vlasov-Poisson-Boltzmann system in $\R^3$,
Arch. Ration. Mech. Anal., 199 (2011), no. 1, 291-328

\bibitem{Duan2} Duan  R.J., Yang T.: Stability of the one-species Vlasov-Poisson-Boltzmann system. SIAM J. Math. Anal. \textbf{41}, 2353-2387 (2010)

\bibitem{Duan3} Duan R.J., Ukai S.: Yang T., Zhao H.J.: Optimal decay estimates
on the linearized Boltzmann equation with time-dependent forces and
their applications. Comm. Math. Phys. \textbf{277}, 189-236 (2008)


\bibitem{Duan5} Duan  R.J.:  Dissipative property of the Vlasov-Maxwell-Boltzmann System with a uniform ionic background. SIAM J. Math. Anal. \textbf{43}(6), 2732-2757 (2011)

\bibitem{Duan4} Duan R.J., Strain R.M.: Optimal large-time behavior of the Vlasov-Maxwell-Boltzmann system in the whole space. Comm. Pure Appl. Math. \textbf{64}, 1497-1546 (2011)

\bibitem{Ellis} Ellis R.S., Pinsky M.A.: The first and second fluid approximations to the linearized Boltzmann equation.  J. Math. pure et appl. \textbf{54}, 125-156 (1975)


\bibitem{Guo2} Guo Y.: The Vlasov-Poisson-Boltzmann system near Maxwellians. Comm. Pure Appl. Math. \textbf{55}, 1104-1135 (2002)



\bibitem{Guo4} Guo Y.: The Vlasov-Maxwell-Boltzmann system near Maxwellians. Invent. Math. \textbf{153}(3), 593-630 (2003)

\bibitem{Jang} Jang J.: Vlasov-Maxwell-Boltzmann diffusive limit. Arch. Ration. Mech. Anal. \textbf{194},  531-584 (2009).

\bibitem{Kato} Kato T.: \emph{Perturbation Theory of Linear Operator}. Springer, New York (1996)


\bibitem{Li2} Li H.-L., Yang T., Zhong M.-Y.: Spectral analysis for the Vlasov-Poisson-Boltzmann system.
preprint (2013).

\bibitem{Li3} Li H.-L., Yang T., Zhong M.-Y.: Spectrum analysis and optimal decay rates of  the bipolar Vlasov-Poisson-Boltzmann equations. preprint (2014).

\bibitem{Liu1} Liu T.-P., Yu S.-H.: The Green’s function and large-time behavior of solutions for the
one-dimensional Boltzmann equation.  Comm. Pure Appl. Math. \textbf{57}, 1543-1608 (2004)

\bibitem{Liu2} Liu T.-P., Yang T., Yu S.-H.: Energy method for the Boltzmann equation. Physica D, \textbf{188}(3-4), 178-192 (2004)


\bibitem{Markowich} Markowich P.A., Ringhofer C.A., Schmeiser C.: \emph{Semiconductor Equations}, Springer-Verlag, Vienna (1990)



\bibitem{Pazy} Pazy A.: \emph{Semigroups of Linear Operators and
Applications to Partial Differential Equations}. AMS Vol. 44. Springer-Verlag, New York, (1983)

\bibitem{Strain} Strain  R. M.: The Vlasov-Maxwell-Boltzmann system in the whole space. Comm. Math. Phys.
\textbf{268}(2), 543-567 (2006)

\bibitem{Ukai1} Ukai S.: On the existence of global solutions of mixed problem for non-linear Boltzmann
equation. Proceedings of the Japan Academy, \textbf{50}, 179-184 (1974)

\bibitem{Ukai2} Ukai S., Yang T.: The Boltzmann equation in the space $L^2\cap L^\infty_\beta$: Global and time-periodic
solutions. Analysis and Applications \textbf{4}, 263-310 (2006)

\bibitem{Ukai3} Ukai S., Yang T.: Mathematical Theory of Boltzmann
Equation. Lecture Notes Series-No. 8,
Hong Kong: Liu Bie Ju Center for Mathematical Sciences, City University of Hong Kong, March 2006.





\bibitem{Yang4} T. Yang, H.J. Yu, Optimal convergence rates of classical solutions
for Vlasov-Poisson-Boltzmann system. Commun. Math. Phys. 301 (2011), 319-355.

\bibitem{Yu} Alexander S., Yu S.-H.: On the Solution of a Boltzmann System for Gas Mixtures. Arch. Rational Mech. Anal. \textbf{195}, 675-700 (2010)


\bibitem{Zhong2012Sci} Zhong M.-Y.: Optimal time-decay rate of the Boltzmann equation. Sci. China Math. \textbf{57}, 807-822 (2014)

\end{thebibliography}
\end{document}